\documentclass[12pt]{amsart}
\usepackage{subfiles}
\usepackage{fullpage} 
\usepackage{graphicx} % Required for inserting images
\usepackage{prefix}
\usepackage{amsmath}
\usepackage{amsfonts}
\usepackage{tikz-cd}
\usepackage{tfrupee}
\usepackage{braket} % simplest option
\usepackage{quiver}

\title{A Hilbert Module Approach to Classical and Quantum Reference Frames}
\author{Lucus Brady}

\address{Department of Mathematical Sciences\\Montana State University\\Bozeman, MT 59717}
\email{lucusbrady@montana.edu}
\date{\today}

\begin{document}

\maketitle
	
	\begin{abstract}
		
		In this paper, we present a model for classical and quantum reference frames for $C^*$-algebras by utilizing Hilbert modules and equivariant correspondences of $C^*$-dynamical systems. We show the classical model provides a commutative baseline for the quantum model, and we discuss how the quantum model encodes the Hilbert space operational formulation of quantum reference frames. We also include various explicit examples for both classical and quantum reference frames.
	\end{abstract}
	
\tableofcontents

\section{Introduction}

Quantum reference frames are a recent growing tool in algebraic approaches to developing models for quantum gravity. Conceptually, a quantum reference frame is a physical system utilized for constructing relational invariant observables for a quantum system. Philosophically, these objects are utilized to remove the idea of an absolute background by establishing an environment where the meaningful observables are those which are both relational and invariant under a predetermined group of symmetries. For example, the notion of absolute position of a particle is both physically meaningless as well as not invariant under translation; however, by introducing another particle to act as the origin (reference system), the relative position represents both a physically meaningful relational quantity as well as a translation invariant observable.

While there are different approaches to modeling quantum reference frames, one formulation that has found recent influence in algebraic quantum field theory is the Hilbert space operational approach \cite{Loveridge_2018, Carette_2025}. In the work of \cite{Fewster_2024}, it was shown that for suitable Hilbert space operational quantum reference frames, one can reduce a type $\text{III}_1$ von Neumann algebra to type $\text{II}_{\infty}$ von Neumann algebra (Corollary 5.17 in \cite{Fewster_2024}). As an application, the authors of \cite{Fewster_2024} demonstrate the type reduction result for quantum fields on de Sitter spacetime (see Section $6$ of \cite{fewster2025semilocalobservablesedgemodes} for details). Additionally, using this type reduction result and extending some of the ideas of the relativization map in the Hilbert space operational formulation to $C^*$-algebras, a semi-local and operational perspective of edge modes in quantum electromagnetism has been developed \cite{fewster2025semilocalobservablesedgemodes}.

In the physical setup of \cite{Fewster_2024}, the algebra of observables for the main system is derived in the algebraic formulation of quantum mechanics \cite{strocchi2008introduction, haag2012local}, while the quantum reference frame is derived in the Hilbert space operational formulation of quantum mechanics \cite{busch1997operational, busch2016quantum}.  The difference of the formulations for building the two objects presents the question as to whether one can model quantum reference frames using von Neumann algebras or, more generally, $C^*$-algebras. While there are many approaches for modeling quantum reference frames using von Neumann algebras \cite{Vanrietvelde2020changeof, PhysRevD.110.065003, KLINGER2024116453, De_Vuyst_2025}, the literature is sparse with approaches using $C^*$-algebras due to the ability to pass to a von Neumann algebra through a representation or via the enveloping von Neumann algebra. As far as we can tell, the only recent work describing quantum reference frames in a representation/von Neumann algebra free setting for $C^*$-algebras is in the work \cite{fewster2025semilocalobservablesedgemodes} where the authors extend the notion of a relativization map to nuclear $C^*$-algebras (see Section 5.4 in \cite{fewster2025semilocalobservablesedgemodes} for more information). While this approach is a start to a $C^*$-algebra formulation, this perspective does not accommodate natural non-nuclear $C^*$-algebras which show up in physics such as bounded operators on an infinite dimensional Hilbert space. Therefore, one of the goals of this paper is to develop a model of quantum reference frames using $C^*$-algebras which is agnostic to the algebras being nuclear.

To start modeling quantum reference frames using $C^*$-algebras, we first begin investigating how to use $C^*$-algebras to model classical reference frames. In the classical setting, we tend to think of a reference frame being modeled by a choice of local coordinates, a choice of a local frame for a vector bundle, or, more generally, a choice of section of some frame bundle. In all three cases, we typically view these objects representing an abstract deterministic choice on the system of interest rather than arising from the use of another physical system. However, by keeping track of the spaces involved, these models for a classical reference frames can be viewed as using another space to derive relational information. For example, if we have an $N$-dimensional manifold, a choice of coordinates uses $\mathbb{R}^N$ (or $\mathbb{H}^N$) as a reference space. 

Beyond deterministic classical reference frames, there also exist probabilistic frames. For example, suppose we have an unmarked two-state switch. As the switch is unmarked, we are unable to distinguish its absolute states; however, by utilizing another unmarked two-state switch as a reference system, we can define the relational observables of aligned and anti-aligned. If the states of the two switches are correlated deterministically, the probability of being aligned or anti-aligned is strictly zero or one. If the switches are correlated non-deterministically, these probabilities become non-trivial. Note the deterministic correlation shares the same flavor of determinism found in the abstract notions of classical reference frames, whereas the non-deterministic correlation characterizes an unsharp classical reference frame. Therefore, a $C^*$-algebraic model of classical reference frames should incorporate deterministic/sharp frames as well as non-deterministic/unsharp frames.  

As with $C^*$-algebra models of quantum reference frames, the literature appears to lack models of classical reference frames using $C^*$-algebras. The lack of a classical description of reference frames in terms of  $C^*$-algebras presents a structural problem for building a quantum model as we don't have a baseline providing appropriate structure and interpretation that can be utilized in the quantum model. Therefore, it is the goal of this paper to present a $C^*$-algebra model of both classical reference frames and quantum reference frames. We achieve this goal by constructing reference frames using the data of Hilbert modules and non-degenerate positive equivariant $C^*$-correspondences.

The use of Hilbert modules to describe classical and quantum reference frames for $C^*$-algebras arises from our interpretation of the relationship between a system and a reference system. In particular, if $A$ is the $C^*$-algebra associated to the main system and $B$ is the $C^*$-algebra for the reference system, we interpret $B$ as a background system utilized to understand $A$. For example, the extra particle to act as an origin is utilized so that we can construct relational information about the original particle of interest. Similarly, the reference two-state switch is utilized so that we can understand relational information about the original two-state switch. Due to the structural asymmetry, the joint system for $A$ and $B$ needs to remember $B$ is a background system so that the observables for the joint system represent relational observables. As the $C^*$-algebra tensor product for $A$ and $B$ treats both algebras symmetrically, the joint system for $A$ and $B$ cannot be modeled by  $A\otimes_\gamma B$ for any $C^*$-norm $\gamma$ when $B$ is viewed as a background.

Outside of the symmetric relationship encoded by $A\otimes_{\gamma}B$, there are additional structural problems for trying to utilize $A\otimes_{\gamma}B$ to model relational observables. Critically, the notion of relational quantities makes sense in the absence of symmetries. Being able to construct the relative position of two particles does not require considering the translation symmetry. Similarly, forming the relational observables of aligned or anti-aligned for an unmarked two state switch does not require the consideration of the transitive $\mathbb{Z}/2\mathbb{Z}$ action. The fact that relational quantities do not depend on symmetries requires models for relational observables to make sense when one takes the symmetry group to be trivial. When working with $A\otimes_\gamma B$, the invariant observables are provided by the subalgebra $(A\otimes_\gamma B)^G$. Therefore, when $G$ is trivial, one recovers all of $A\otimes_\gamma B$ which does not model information about $A$ relative to $B$.

To codify the asymmetric relationship between the algebras and to have a model of relational observables in the absence of symmetries, we view the $C^*$-algebras as Hilbert modules over themselves and apply both the KSGNS construction and interior tensor product to obtain a Hilbert $B$-module. This module represents the relational joint space of the two systems while the $C^*$-algebra of adjointable operators on the module represents the relational joint system. When $A$ and $B$ carry symmetries from a group (forming $C^*$-dynamical systems), we apply an equivariant version of both the KSGNS construction and interior tensor product to obtain a canonical action on both the relational joint space and the relational joint system. Using the induced action on the relational joint system, we obtain invariant relational observables as elements in the fixed-point subalgebra. 

Many mathematically and physically interesting Hilbert modules do not arise simply by viewing $C^*$-algebras as modules over themselves. Common examples include Hilbert spaces and continuous sections of a complex Hermitian vector bundle that vanish at infinity. To accommodate these examples, we generalize the input data of our model from the $C^*$-algebras $A$ and $B$ to the data of Hilbert modules over $A$ and $B$. In this generalized framework, the KSGNS construction and interior tensor product still construct the relational joint space and the relational joint system. Similarly, when handling group symmetries, we use the equivariant version of both the KSGNS construction and interior tensor product to build the $C^*$-algebra of invariant relational observables.

	In addition to using the KSGNS construction and interior tensor product to build the joint space, we utilize the map induced by the interior tensor product on the algebra of adjointable operators to define the relativization map. Specifically, a relativization map is a choice of an extension for the induced unital $*$-algebra map on the fixed-point algebras derived from the equivariant interior tensor product. We note that this approach does not rule out building relativization maps from generalized semi-spectral measures; indeed, when $A = \mathbb{C} = B$, the Hilbert space operational formulation provides explicit examples within our generalized framework (see section 4.2 for details).

Now, let us briefly outline the layout of this paper. In Section 2, we review the notions of correspondences for $C^*$-algebras as well as some of the operations associated to them. In Section 3, we build a model of classical reference frames using $G$-invariant weakly continuous Markov kernels. In Section 4, we generalize the classical reference frame model to quantum reference frames. In this section, we also show how the Hilbert space operational framework provides examples in our framework.

We conclude the introduction with a comment on how this paper fits into a research program for building $C^*$-algebra models of quantum reference frames for algebraic quantum field theories. First, in the paper \cite{brady2026functorialityksgnsconstructionintertwiners}, we built a category to view strict positive $C^*$-correspondences between $C^*$-algebras as objects with morphisms given by intertwiners. Using this category, we proved the KSGNS construction defines an idempotent endofunctor which is topological under reasonable conditions. Using this functorial perspective, we provided a functorial proof  of a general equivariant KSGNS construction for strict positive equivariant $C^*$-correspondences. In this paper, we utilize the equivariant KSGNS construction result as a technical tool for building the joint system in our model of both classical and quantum reference frames. Combining the framework developed here with the categorical structure in \cite{brady2026functorialityksgnsconstructionintertwiners}, a future paper we will outline a categorical description of our model of quantum reference frames. This categorical description will enable us to do two things. First, it will allow us to build a unitary-like group for a quantum reference frame representing the transformation group of the frame. Second, it will provide a pathway to building a $C^*$-algebra model of quantum reference frames for algebraic quantum field theories.

\subsection*{Acknowledgments}

I would like to thank my advisor, Ryan Grady, for his support and encouragement as I have worked on this project. Additionally, I would like to thank Chris Fewster, Daan Janssen,  Kasia Rejzner,  Leon Loveridge, and James Waldron for the insightful conversations related to  this work.

\subsection*{Funding} I was sponsored in part by the Air Force Research Laboratory under Agreement Number FA8750-24-1-1019. The U.S.Government is authorized to reproduce and distribute reprints for government purposes not withstanding any copyright notation thereon. Any opinions, findings, and conclusions or recommendations expressed in this material are those of the authors and do not necessarily reflect the view of the funder.

\section{Correspondences for $C^*$-algebras}
		
In this section, we review correspondences of $C^*$-algebras and operations associated to them. We assume the reader is familiar with both Hilbert modules and completely positive maps. For information on Hilbert modules, see \cite{lance1995hilbert},\cite{raeburn1998morita}, or \cite{blackadar2006operator}. For information on completely positive maps, see \cite{paulsen2002completely} or \cite{blackadar2006operator}.

	\subsection{$C^*$-correspondences and the interior tensor product}
	
	A \textit{$C^*$-correspondence} from a $C^*$-algebra $A$ to a $C^*$-algebra $B$ is a pair $(E_B, \rho)$ where $E_B$ is a Hilbert $B$-module and $\rho:A\to \mathcal{L}(E_B)$ is a non-degenerate $*$-algebra homomorphism\footnote{A $*$-algebra map $\rho:A\to \mathcal{L}(E_B)$ is non-degenerate if $\rho(A)E_B:=\text{Span}(\{\rho(a)x: a\in A, x\in E_B\})$ is dense in $E_B$ with respect to the norm topology.}. Conceptually, a $C^*$-correspondence represents a generalized $*$-algebra homomorphism from $A$ to $B$; in fact, this perspective plays a foundational role in building the Morita category for $C^*$-algebras \cite{echterhoff2005categoricalapproachimprimitivitytheorems}. In the Morita category, composition of $C^*$-correspondences is given by utilizing the interior/balanced tensor product. More generally,  the interior tensor product can be viewed as a mechanism for converting Hilbert modules over one $C^*$-algebra to Hilbert modules over another. Let us briefly review this construction; a detailed account can be found in chapter 4 of \cite{lance1995hilbert}. 
    
    Fix $C^*$-algebras $A$ and $B$, a $C^*$-correspondence $(E_B, \rho)$ from $A$ to $B$, and a Hilbert $A$-module $E_A$. Using $\rho$ to view $E_B$ as a left $A$-module, we construct the tensor product of $E_A$ and $E_B$ over $A$, denoted by $E_A\otimes_AE_B$. Recall, $E_A\otimes_A E_B$ is constructed by starting with the algebraic tensor product of the underlying complex vector spaces for $E_A$ and $E_B$, denoted by $E_A\otimes_{\textnormal{alg}}E_B$, and quotienting out by the submodule $N$ generated by the set 
	$$\{xa\otimes y-x\otimes \rho(a)y: x\in E_A, y\in E_B, a\in A\}.$$
	On $E_A\otimes_A E_B$, we define the pairing
	$$\langle\cdot, \cdot\rangle:(E_A\otimes_A E_B)^2\to B\quad\quad \langle x_1\otimes y_1+N, x_2\otimes y_2+N\rangle=\langle y_1, \rho(\langle x_1, x_2\rangle_{E_A})y_2\rangle_{E_B}$$
	which yields a $B$-valued inner product. With respect to the induced norm, we take the completion of $E_A\otimes_A E_B$, denoted by $E_A\otimes_\rho E_B$, to obtain a Hilbert $B$-module called the interior tensor product of $E_A$ and $E_B$ (with respect to $\rho$). 
	
	Given an adjointable operator $T\in \mathcal{L}(E_A)$, we obtain an adjointable operator $T\otimes_\rho I$ on $E_A\otimes_\rho E_B$ as follows. First, define 
	$$T\otimes I:E_A\otimes_{A}E_B\to E_A\otimes_{\rho}E_B\quad\quad (T\otimes I)(x\otimes y+N)=Tx\otimes y+N$$
	which is a well-defined, bounded, $B$-linear map. As $E_A\otimes_A E_B$ is dense in $E_A\otimes_\rho E_B$, continuity and $B$-linearity of $T\otimes I$ imply $T\otimes I$ extends to a unique bounded $B$-linear map on $E_A\otimes_\rho E_B$, denoted by $T\otimes_\rho I$. Observe $T\otimes_\rho I$ is adjointable with adjoint $T^*\otimes_\rho I$. As we vary over $\mathcal{L}(E_A)$, we obtain a unital $*$-algebra homomorphism 
	$$\mathcal{L}(E_A)\to \mathcal{L}(E_A\otimes_\rho E_B)\quad\quad T\mapsto T\otimes_\rho I$$
    which is strictly\footnote{Given a Hilbert $A$-module $E_A$, the strict topology on $\mathcal{L}(E_A)$ is the topology generated by the semi-norms $||\cdot||_{x}:\mathcal{L}(E_A)\xrightarrow{T\mapsto ||Tx||} [0, \infty)$ and $||\cdot||_{x, *}:\mathcal{L}(E_A)\xrightarrow{T\mapsto ||T^*x||} [0, \infty)$ for all $x\in E_A$.} continuous on the unit ball in $\mathcal{L}(E_A)$.
		
	\subsection{Non-degenerate positive $C^*$-correspondences and the KSGNS construction}
	
	Just as $*$-algebra homomorphisms can be weakened to completely positive maps, so too can we weaken $C^*$-correspondences to an appropriate notion of a generalized completely positive map. Unlike the $*$-algebra case, there are two appropriate definitions of a generalized completely positive morphism of $C^*$-algebras: strict positive $C^*$-correspondences and non-degenerate positive $C^*$-correspondences. As we will only utilize non-degenerate positive $C^*$-correspondences in this work, we will only review these correspondences.
	
	Fix $C^*$-algebras $A$ and $B$ as well as a Hilbert $B$-module $E_B$. First, a completely positive map $\phi:A\to \mathcal{L}(E_B)$ is \textit{non-degenerate} if there exists an approximate unit in $A$ which converges under $\phi$ strictly to $1_{E_B}$. Equivalently, the map $\phi$ is non-degenerate if there exists a unital completely positive map $\tilde{\phi}:M(A)\to \mathcal{L}(E_B)$ which is strictly continuous on the unit ball in $M(A)$ and restricts to $\phi$ on $A\subset M(A)$ (Corollary 5.7 in \cite{lance1995hilbert}). A \textit{non-degenerate positive $C^*$-correspondence} from $A$ to $B$ is the data $(E_B, \phi)$ where $E_B$ is a Hilbert $B$-module, and $\phi:A\to \mathcal{L}(E_B)$ is a non-degenerate completely positive map. 
	
	Given a non-degenerate completely positive map $\phi:A\to \mathcal{B}(H)$ for $H$ a Hilbert space (that is, a non-degenerate positive $C^*$-correspondence from $A$ to $\mathbb{C}$), the Stinespring Dilation Theorem (Theorem 4.1 and Proposition 4.2 in \cite{paulsen2002completely}) can be applied to obtain a unitarily unique triple $(H_\phi, \pi_\phi, V_\phi)$ where $H_\phi$ is a Hilbert space, $\pi_\phi:A\to \mathcal{B}(H_\phi)$ is a $*$-algebra homomorphism, and $V_\phi:H\to H_\phi$ is a bounded linear map, and the triple satisfies the two following conditions:
	\begin{enumerate}
		\item for all $a\in A$, $\phi(a)=V_\phi^*\pi_\phi(a)V_\phi$, and
		\item $\pi_\phi(A)V_\phi H$ is dense in $H_\phi$.
	\end{enumerate}
	Note, $(1)$ and non-degeneracy of $\phi$ imply $V_\phi$ is an isometry. Furthermore, $(2)$ implies $(H_\phi, \pi_\phi)$ is a $C^*$-correspondence from $A$ to $\mathbb{C}$. Thus, the Stinespring Dilation Theorem provides a way of dilating the non-degenerate positive $C^*$-correspondence to a $C^*$-correspondence. This dilation result holds more generally for non-degenerate positive $C^*$-correspondences between any two $C^*$-algebras, and it goes under the name of the KSGNS\footnote{Named after Kasparov, Stinespring, Gelfand, Naimark, and Segal} construction. We note that when $B=\mathbb{C}$, the KSGNS dilation coincides with Stinespring dilation. Let us briefly review the KSGNS construction, details can be found in chapter 5 of \cite{lance1995hilbert}. 
    
    Again, fix $C^*$-algebras $A$ and $B$ as well as a non-degenerate positive $C^*$-correspondence $(E_B, \phi)$ from $A$ to $B$. Similar to the interior tensor product, we begin with the algebraic tensor product $A\otimes_{\textnormal{alg}}E_B$ which can be made into a right  $B$-module where $B$ acts on $E_B$. On this $B$-module, we define the $B$-valued pairing
	$$\langle\cdot, \cdot\rangle_\phi:(A\otimes_{\textnormal{alg}}E_B)^2\to B\quad\quad\langle a_1\otimes x_1, a_2\otimes x_2\rangle_\phi=\langle x_1, \phi(a_1^*a_2)x_2\rangle_{E_B}$$
	which satisfies all the conditions for making $A\otimes_{\textnormal{alg}}E_B$ into an inner product $B$-module except potentially non-degeneracy; that is, there may exist $z\in A\otimes_{\textnormal{alg}}E_B$ with $\langle z, z\rangle_\phi=0$ though $z\neq 0$. Let $N_\phi$ be the submodule of $A\otimes_{\textnormal{alg}}E_B$ for which $\langle z, z\rangle_\phi=0$.
	
	Pushing $\langle\cdot, \cdot\rangle_\phi$ onto the quotient $(A\otimes_{\textnormal{alg}}E_B)/N_\phi$ makes the quotient $B$-module into an inner product $B$-module. Completing the inner product $B$-module with respect to the induced norm yields a Hilbert $B$-module which we denote by $A\otimes_\phi E_B$. Similar to the interior tensor product, we obtain a unital $*$-algebra homomorphism
	$$\tilde{\pi}_\phi:M(A)\to \mathcal{L}(A\otimes_\phi E_B)\quad\quad T\mapsto T\otimes_\phi I$$
	which is strictly continuous on the unit ball in $M(A)$. Therefore, restricting $\tilde{\pi}_\phi$ to $A\subset M(A)$ yields a non-degenerate $*$-algebra homomorphism which we denote by $\pi_\phi$.
	
	As $\phi$ is non-degenerate, there exists an approximate unit $(a_\lambda)_{\lambda\in\Lambda}$ such that $(\phi(a_\lambda))_{\lambda\in\Lambda}$ converges strictly to $1_{E_B}$. Using the net $(a_\lambda)_{\lambda\in\Lambda}$, we define for each $\lambda\in \Lambda$
	$$V_\lambda:E_B\to A\otimes_\phi E_B\quad\quad V_\lambda x=a_\lambda\otimes x+N_\phi$$
	which is a well-defined, $B$-linear map. As $(\phi(a_\lambda))_{\lambda\in\Lambda}$ is a convergent net with respect to the strict topology, for each $x\in E_B$, the net $(V_\lambda x)_{\lambda\in \Lambda}$ converges in the norm topology on $A\otimes_\phi E_B$ to a point $V_\phi(x)$.  Varying over $x\in E_B$ yields a map $V_\phi \in \mathcal{L}(E_B, A\otimes_\phi E_B)$ for which $V_\phi x=\lim\limits_{\lambda\to \infty }V_\lambda x$ for all $x\in E_B$. Observe 
	$$V_\phi^*:A\otimes_\phi E_B\to E_B\quad\quad V_\phi^*(a\otimes x+N_\phi)=\phi(a)x$$
	
	While the net $(V_\lambda)_{\lambda\in\Lambda}$ is dependent on the approximate unit $(a_\lambda)_{\lambda\in\Lambda}$, the limit $V_\phi$ is independent of the approximate unit. Indeed, if $(a_\mu)_{\mu\in M}$ is another approximate unit, then using $\tilde{\phi}:M(A)\to \mathcal{L}(E_B)$ shows $(\phi(a_\mu))_{\mu\in M}$ converges to $\tilde{\phi}(1)=1_{E_B}$ in the strict topology. Therefore, we can construct a net $(V_\mu)_{\mu\in M}$ and a function $V'\in \mathcal{L}(E_B, A\otimes_\phi E_B)$ which is the strong limit of $(V_\mu)_{\mu\in M}$. From the observation at the end of the last paragraph, we have $V'^*=V_\phi^*$. Thus, by uniqueness of adjoints, $V'=V_\phi$.
	
	With respect to $V_\phi$ and $\pi_\phi$, we have $V_\phi^*\pi_\phi(a)V_\phi=\phi(a)$ for all $a\in A$ and the submodule $\pi_\phi(A)V_\phi E_B\subset A\otimes_\phi E_B$ defined as
	$$\pi_\phi(A)V_\phi E_B=\textnormal{Span}(\{\pi_\phi(a)V_\phi y: a\in A, y\in E_B\})$$
	is dense in $A\otimes_\phi E_B$. As in the Stinespring Dilation Theorem, the triple $(A\otimes_\phi E_B, \pi_\phi, V_\phi)$ is unitarily uniquely characterized by these  two properties. 
	
	\begin{theorem}[KSGNS Construction \cite{lance1995hilbert}]\label{Theorem: KSGNS}\
		
		\noindent Let $A$ and $B$ be $C^*$-algebras. Let $E_B$ be a Hilbert $B$-module, and let $\phi:A\to \mathcal{L}(E_B)$ be a non-degenerate completely positive map. Then there exists a triple $(F_\phi, \pi_\phi, V_\phi)$ where $F_\phi$ is a Hilbert $B$-module, $\pi_\phi\in \textnormal{hom}_{*\textnormal{-alg}}(A, \mathcal{L}(F_\phi))$, and $V_\phi\in \mathcal{L}(E_B, F_\phi)$, and the triple satisfies the following conditions:
		\begin{enumerate}
			\item for all $a\in A$, $\phi(a)=V_\phi^*\pi_\phi(a)V_\phi$.
			\item the submodule $\pi_\phi(A)V_\phi E_B$ is dense in $F_\phi$.
		\end{enumerate}
		If $(F', \pi', V')$ is another triple satisfying the two conditions, then there exists a $B$-linear unitary $U\in \mathcal{L}(F_\phi, F')$ such that for all $a\in A$, $\pi'(a)=U\pi_\phi(a)U^*$ and $V'=UV_\phi$.
		
	\end{theorem}
	
	Observe, as $\phi$ is non-degenerate, the first condition on the triple $(F_\phi, \pi_\phi, V_\phi)$ implies $V_\phi$ is an isometry. Indeed,  suppose we have an approximate unit $(a_\lambda)_{\lambda\in\Lambda}$ for $A$, then $\phi(a_\lambda)\to 1_{E_B}$ strictly. As $(a_\lambda)_{\lambda\in\Lambda}$ is in the unit ball of $M(A)$ and converges strictly to $1\in M(A)$, $\pi_\phi$ being non-degenerate implies $\pi_\phi(a_\lambda)\to 1_{F_\phi}$ strictly. Therefore, with respect to the strict topology,
		$$1_{F_\phi}=\lim\limits_{\lambda\to\infty}\phi(a_\lambda)=\lim\limits_{\lambda\to\infty}V_\phi^* \pi_\phi(a_\lambda)V_\phi=V_\phi^*V_\phi.$$
	Additionally, observe the second condition implies $(F_\phi, \pi_\phi)$ is a $C^*$-correspondence from $A$ to $B$.
	
	\subsection{Equivariant correspondences}
	
	We conclude our review of correspondences with equivariant $C^*$-correspondences and non-degenerate positive equivariant $C^*$-correspondences. As with the other correspondences we've discussed so far, these correspondences can be viewed as generalized equivariant $*$-algebra morphisms of $C^*$-dynamical systems and generalized equivariant completely positive morphisms of $C^*$-dynamical systems, respectively. Before introducing these correspondences, we quickly review the notion of $C^*$-dynamical systems and Hilbert module dynamical systems. For more information about $C^*$-dynamical systems, see \cite{williams2007crossed}.
	
	\subsubsection{Review of $C^*$-dynamical systems}
	
	Let $A$ be a $C^*$-algebra, and let $G$ be a group. A group action of $G$ on $A$ is the data of a group homomorphism $\alpha:G\to \textnormal{Aut}_{*\textnormal{-alg}}(A)$. If $G$ is a topological group, then the action is continuous if $\alpha$ is continuous with respect to the point-norm topology\footnote{Given $C^*$-algebras $A$ and $B$, the point-norm topology on $\text{hom}_{*\text{-alg}}(A, B)$ is the topology generated by the family of semi-norms $||\cdot||_a:\text{hom}_{*\text{-alg}}(A, B)\xrightarrow{\rho\mapsto ||\rho(a)||} [0, \infty)$ for all $a\in A$.} on $\textnormal{Aut}_{*\textnormal{-alg}}(A)$. We refer to the triple $(A, G, \alpha)$ as a \textit{$C^*$-dynamical system}. In the case $G$ is a topological group and $\alpha$ is continuous, we refer to the triple as a \textit{continuous $C^*$-dynamical system}.

	We note that there is typically no distinction made between $C^*$-dynamical systems and continuous $C^*$-dynamical systems; in fact, if $G$ is not a topological group, one can equip $G$ with the discrete topology to make any $C^*$-dynamical system into a continuous $C^*$-dynamical system. We make the distinction in this work as we want to keep track of whether a group is topological and if the action from the group is continuous. Additionally, we note that  topological groups for continuous $C^*$-dynamical systems are typically assumed to be locally compact and Hausdorff. We do not make such an assumption here.
	
	\subsubsection{Group actions on Hilbert modules and Hilbert module dynamical systems}
	
	The notion of a group action on a Hilbert module is conceptually akin to representations of groups on Hilbert spaces; however, as Hilbert modules are over an arbitrary $C^*$-algebra rather than $\mathbb{C}$, a group action on the module has the additional degree of freedom to also act on the underlying $C^*$-algebra. This additional degree of freedom allows the group action to twist the module structure which causes the maps for the representation to be linear and adjointable with respect to the twisting. Note, in the Hilbert space setting, a group action cannot twist a Hilbert space as the group of $*$-automorphisms for $\mathbb{C}$ is trivial.
	
	Let us first review how $*$-automorphisms on the underlying $C^*$-algebra of a Hilbert module can twist the module structure. Let $E_B$ be a Hilbert $B$-module, and let $\beta$ be a $*$-algebra automorphism of $B$. Letting $\langle\cdot, \cdot\rangle$ and $r$ denote the $B$-valued pairing for $E_B$ and right action of $B$ on $E_B$, respectively, we utilize the automorphism $\beta$ to twist these maps as follows: 
	$$(\cdot, \cdot):E_B\times E_B\to B\quad\quad (x, y)=\beta(\langle x, y\rangle)$$
	and
	$$r':E_B\times B\to E_B\quad\quad r'(x, b)=r(x, \beta^{-1}(b)).$$
	It is clear that $r'$ makes $E_B$ into a right $B$-module as well as $(\cdot, \cdot)$ makes $E_B$ into an inner product $B$-module with respect to the action $r'$. Since the induced norm from $(\cdot, \cdot)$ is the same as $\langle\cdot, \cdot\rangle$, then $E_B$ is a Hilbert $B$-module with respect to $(\cdot, \cdot)$ and $r'$. Let us denote $E_B$ with this $\beta$-twisted structure as $E^\beta_B$. An alternative construction of the twisted $B$-module $E^{\beta}_B$ is obtained using the interior tensor product. 
	
	\begin{lemma}[Lemma 6.1 in \cite{brady2026functorialityksgnsconstructionintertwiners}]\label{Lemma: beta unitary}\

		\noindent Let $E_B$ be a Hilbert $B$-module, and let $\beta$ be a $*$-automorphism of $B$. Then the map
		$$U_\beta:E_B\otimes_{\textnormal{inc}\circ\beta} B\to E^\beta_B\quad\quad U_\beta(x\dot{\otimes}b)=x\beta^{-1}(b)$$
		defines a $B$-linear unitary.
		
	\end{lemma}
	
	As the underlying Banach spaces of $E_B$ and $E^\beta_B$ are the same, the map $U_\beta$ is a well-defined, bijective \textit{$\mathbb{C}$-linear} isometry from $E_B\otimes_{\textnormal{inc}\circ\beta} B$ to $E_B$ satisfying the following two properties:
	\begin{enumerate}
		\item for all $x\in E_B\otimes_{\textnormal{inc}\circ\beta} B$ and $y\in E_B$, $\langle U_\beta(x), y\rangle_{E_B}=\beta^{-1}(\langle x, U^{-1}_\beta(y)\rangle_{E_B\otimes_{\textnormal{inc}\circ\beta} B})$.
		\item for all $x\in E_B\otimes_{\textnormal{inc}\circ\beta} B$ and $b\in B$, $U_\beta(xb)=U_\beta(x)\beta^{-1}(b)$.
		
	\end{enumerate}
	Generalizing these properties for $\mathbb{C}$-linear maps between Hilbert modules yields the following definitions.

	\begin{definition}[Definition 6.3 in \cite{brady2026functorialityksgnsconstructionintertwiners}]\

		\noindent Let $E_1$ and $E_2$ be Hilbert $B$-modules, and let $\beta\in\textnormal{Aut}_{*\textnormal{-alg}}(B)$. 
		\begin{itemize}
			\item A (bounded) $\beta$-linear map from $E_1$ to $E_2$ is a (bounded) $\mathbb{C}$-linear map $T:E_1\to E_2$ such that for all $x\in E_1$ and $b\in B$, $T(xb)=T(x)\beta(b)$.
			\item A function $T:E_1\to E_2$  is $\beta$-adjointable if there exists a map $T^*:E_2\to E_1$ such that for all $x\in E_1$ and $y\in E_2$, $\langle T(x), y\rangle_{E_2}=\beta(\langle x, T^*(y)\rangle_{E_1})$. Denote the set of $\beta$-adjointable maps from $E_1$ to $E_2$ as $\mathcal{L}^\beta(E_1, E_2)$.
			\item A $\beta$-adjointable map $T:E_1\to E_2$ is unitary if $TT^*=1_{E_2}$ and $T^*T=1_{E_1}$. Denote the set of $\beta$-adjointable unitaries from $E_1$ to $E_2$ as $\mathcal{U}^\beta(E_1, E_2)$.
			
		\end{itemize}
		
	\end{definition}
	
	While not assumed, it follows $\beta$-adjointable maps are $\beta$-linear and have unique adjoints. Using the Uniform Boundedness Prinicple, it follows $\beta$-adjointable maps are bounded. Additionally, using the isomorphism between $E_1^\beta$ and $E_1\otimes_{\text{inc}\circ \beta}B$ from Lemma \ref{Lemma: beta unitary}, we obtain an identification of $\beta$-adjointable maps on $E_1$ and adjointable $B$-linear maps on $E_1\otimes_{\text{inc}\circ \beta} B$.
	
	\begin{proposition}[Proposition 6.4 in \cite{brady2026functorialityksgnsconstructionintertwiners}]\label{Prop: rel between B-linear and beta-linear}\

		\noindent Let $E_1$ and $E_2$ be Hilbert $B$-modules, and let $\beta\in\textnormal{Aut}_{*\textnormal{-alg}}(B)$. Let $U_\beta$ be the $\beta^{-1}$-adjointable unitary given by
		$$U_\beta:E_1\otimes_{\textnormal{inc}\circ\beta} B\to E_1\quad\quad U_\beta(x\dot{\otimes} b)=x\beta^{-1}(b)$$
		Then the map
		$$\mathcal{L}^\beta(E_1, E_2)\to \mathcal{L}(E_1\otimes_{\textnormal{inc}\circ\beta} B, E_2)\quad\quad T\mapsto T\circ U_\beta$$
		defines a linear bijection. Furthermore, the map restricts to a bijection $$\mathcal{U}^\beta(E_1, E_2)\to \mathcal{U}(E_1\otimes_{\textnormal{inc}\circ\beta} B, E_2).$$
		
	\end{proposition}
	
	Using the definition of $\beta$-adjointable maps and $\beta$-adjointable unitary maps, we obtain the following definitions for $G$-actions on Hilbert modules and Hilbert module dynamical systems. 
	
	\begin{definition}\
		
	\begin{itemize}
		\item Let $E_B$ be a Hilbert $B$-module, and let $G$ be a group. A $G$-action on $E_B$ is a pair $(\beta, U)$ where $\beta:G\to\text{Aut}_{*\text{-alg}}(B)$ and $U:G\to \mathcal{B}(E_B)$ are group homomorphisms, and the pair satisfies the following condition: for all $g\in G$, $U_g$ is  $\beta_g$-adjointable. When $G$ is a topological group, the action is continuous if $\beta$ is continuous with respect to the point-norm topology on $\text{Aut}_{*\text{-alg}}(B)$ and $U$ is continuous with respect to the strong operator topology on $\mathcal{B}(E_B)$.
		\item A Hilbert module dynamical system is given by the data $((B, G, \beta), E_B, U)$ where $(B, G, \beta)$ is a $C^*$-dynamical system, $E_B$ is a Hilbert $B$-module, and $U:G\to \mathcal{B}(E_B)$ is a group homomorphism, and the data satisfies the following condition: $(\beta, U)$ is a $G$-action on $E_B$ with $U_g$ a $\beta_g$-adjointable unitary for each $g\in G$. When $G$ is a topological group, the dynamical system is continuous if the $G$-action $(\beta, U)$ is continuous.
		
	\end{itemize}

	\end{definition}

    Observe that when $B=\mathbb{C}$, a $G$-action on $E_{\mathbb{C}}$ is equivalent to the data of a representation of $G$ on $E_{\mathbb{C}}$. Similarly, a Hilbert module dynamical system of the form $((\mathbb{C}, G, \beta), E_{\mathbb{C}}, U)$ is equivalent to the data of a unitary representation of $G$ on $E_{\mathbb{C}}$.

	\subsubsection{Equivariant correspondences and the equivariant KSGNS construction}	
	
    Let $B$ be a $C^*$-algebra. For each $\beta\in \text{Aut}_{*\text{-alg}}(B)$, we obtain a unique $\tilde{\beta}\in \text{Aut}_{*\text{-alg}}(M(B))$  such that for all $T\in M(B)$ and $b\in B$, $\tilde{\beta}(T)b=\beta(T\beta^{-1}(b))$. Note, $\tilde{\beta}|_B=\text{inc}\circ\beta$. Furthermore, if we view $B$ as a Hilbert module over itself so that $M(B)=\mathcal{L}(B)$, then $\beta$ defines a $\beta$-adjointable unitary on $B$ and $\tilde{\beta}=\text{Ad}(\beta)$.

    Let $(A, G, \alpha)$ and $(B, G, \beta)$ be $C^*$-dynamical systems. For each $g\in G$, denote the unique $*$-algebra isomorphism induced from $\beta_g$ as $\tilde{\beta}_g$. Let $\rho:A\to M(B)$ be a non-degenerate, equivariant, $*$-algebra homomorphism; that is, for all $g\in G$ and $a\in A$, $\rho(\alpha_g(a))=\tilde{\beta}_g(\rho(a))$. Viewing $E:=B$ as a Hilbert module over itself, each $U_g:=\beta_g$ defines a $\beta_g$-adjointable unitary on $E$ with adjoint given by $U_{g^{-1}}$. Furthermore, the intertwining condition for $\rho$ implies for all $g\in G$, $a\in A$, $U_g\circ \rho(a)=\rho(\alpha_g(a))\circ U_g$. Generalizing the module $E$ and maps $U_g$ as well as the intertwining relation between $U_g$ and $\rho$, we obtain the appropriate notion of a generalized equivariant $*$-algebra homomorphism and generalized equivariant completely positive map.
	
	\begin{definition}[Definition 6.5 in \cite{brady2026functorialityksgnsconstructionintertwiners}]\

		\noindent Let $(A, G, \alpha)$ and $(B, G, \beta)$ be $C^*$-dynamical systems. An equivariant $C^*$-correspondence from $(A, G, \alpha)$ to $(B, G, \beta)$ is the data $((E_B, \rho), U)$ where 
		\begin{enumerate}
			\item $(E_B, \rho)$ is a $C^*$-correspondence from $A$ to $B$.
			\item $U:G\to \mathcal{B}(E_B)$ is a group homomorphism 
		\end{enumerate}
		and the data satisfies the following conditions:
		\begin{enumerate}
			\item For all $g\in G$, $U_g$ is a $\beta_g$-adjointable unitary on $E_B$.
			\item For all $g\in G$ and $a\in A$, $U_g\circ \rho(a)=\rho(\alpha_g(a))\circ U_g$.
		\end{enumerate}
		If $(E_B, \rho)$ is a non-degenerate positive $C^*$-correspondence, then $((E_B, \rho), U)$ is called a non-degenerate positive equivariant $C^*$-correspondence. If $(A, G, \alpha)$ and $(B, G, \beta)$ are continuous $C^*$-dynamical systems, then the (non-degenerate positive) equivariant $C^*$-correspondence $((E_B, \rho), U)$ is continuous if $U$ is strongly continuous.  
		
	\end{definition}
	
	For Hilbert modules, a $C^*$-correspondence provides a means for converting a Hilbert module over one $C^*$-algebra to a Hilbert module over another. In a similar manner, an equivariant $C^*$-correspondence enables one to convert a Hilbert module dynamical system over one $C^*$-dynamical system to a Hilbert module dynamical system over another $C^*$-dynamical system. Note, in the following proposition and proof, we follow the notational convention discussed in \cite{brady2026functorialityksgnsconstructionintertwiners} for simple tensors in the module constructed from the KSGNS construction and interior tensor product.
	\begin{proposition}\label{Prop: equiv interior tensor product}\
		
		\noindent Let $(A, G, \alpha)$ and $(B, G, \beta)$ be $C^*$-dynamical systems, and let $((E_B, \rho), W)$ be an equivariant $C^*$-correspondence from $(A, G, \alpha)$ to $(B, G, \beta)$. If $((A, G, \alpha), E_A, U)$ is a Hilbert module dynamical system, then
		\begin{enumerate}
			\item for each $g\in G$, there exists a unique $\beta_g$-adjointable unitary $U_g\otimes_{\rho} W_g\in\mathcal{B}(E_A\otimes_\rho E_B)$ such that for all $x\in E_A$ and $y\in E_B$, $(U_g\otimes_\rho W_g)(x\dot{\otimes}y)=U_gx\dot{\otimes}W_gy$.
			\item the map $$U\otimes_{\rho}W:G\to \mathcal{B}(E_A\otimes_\rho E_B)\quad\quad g\mapsto U_g\otimes_{\rho} W_g$$
			is a group homomorphism.
			\item The unital $*$-algebra homomorphism
			$$\mathcal{L}(E_A)\to \mathcal{L}(E_A\otimes_\rho E_B) \quad\quad T\mapsto T\otimes_\rho I$$
			is $G$-equivariant with respect to $\text{Ad}(U)$ and $\text{Ad}(U\otimes_\rho W)$.
			\item $((B, G, \beta), E_A\otimes_\rho E_B, U\otimes_\rho W)$ is a Hilbert module dynamical system which is continuous whenever $((A, G, \alpha), E_A, U)$ is a continuous Hilbert module dynamical system and $((E_B, \rho), W)$ is a continuous equivariant $C^*$-correspondence.
		\end{enumerate} 
		
	\end{proposition}
	
	\begin{proof}\
		
		\noindent We start by proving $(1)$. Fix $g\in G$. As
		$$E_A\times E_B\to E_A\otimes_\rho E_B\quad\quad (x, y)\mapsto U_gx\dot{\otimes}W_gy$$
		is $\mathbb{C}$-bilinear and $\beta_g$-linear in the second slot, then we obtain a well-defined $\beta_g$-linear map
		$$E_A\otimes_{\text{alg}}E_B\to E_A\otimes_\rho E_B\quad\quad x\otimes y\mapsto U_gx\dot{\otimes}W_gy$$
		Given $x\in E_A$, $y\in E_B$, and $a\in A$, we have
		\begin{align*}
			U_g(xa)\dot{\otimes}W_gy-U_g(x)\dot{\otimes}W_g(\rho(a)y)&=U_g(x)\alpha_g(a)\dot{\otimes}W_gy-U_g(x)\dot{\otimes}\rho(\alpha_g(a))W_gy=0
		\end{align*}	
		Therefore, we obtain a well-defined $\beta_g$-linear map
		$$U_g\otimes_\rho W_g:E_A\otimes_A E_B\to E_A\otimes_\rho E_B \quad\quad x\dot{\otimes}y\mapsto U_gx\dot{\otimes}W_gy$$
		As
		\begin{align*}
			\left\|\sum\limits_{k=1}^n U_gx_k\dot{\otimes}W_gy_k\right\|^2&=\left\|\sum\limits_{i,j=1}^n \langle W_g y_i, \rho(\langle U_g x_i, U_gx_j\rangle)W_g y_j\rangle\right\|_B=\left\|\sum\limits_{i,j=1}^n \langle W_g y_i, \rho(\alpha_g(\langle x_i, x_j\rangle))W_g y_j\rangle\right\|_B\\
			&=\left\|\sum\limits_{i,j=1}^n \langle W_g y_i, W_g\rho(\langle x_i, x_j\rangle) y_j\rangle\right\|_B=\left\|\sum\limits_{i,j=1}^n \beta_g(\langle  y_i, \rho(\langle x_i, x_j\rangle) y_j\rangle)\right\|_B\\
			&=\left\|\sum\limits_{i,j=1}^n \langle  y_i, \rho(\langle x_i, x_j\rangle) y_j\rangle\right\|_B=\left\|\sum\limits_{k=1}^n x_k\dot{\otimes}y_k\right\|^2
		\end{align*}	
		the map $U_g\otimes_\rho W_g$ is a $\beta_g$-linear isometry. Therefore the map extends to a $\beta_g$-linear isometry on $E_A\otimes_\rho E_B$ which we denote by $U_g\otimes_\rho W_g$. 
		
		We claim $U_g\otimes_\rho W_g$ is $\beta_g$ adjointable with adjoint $U_{g^{-1}}\otimes_\rho W_{g^{-1}}$. First, by the construction in the prior paragraph, $U_{g^{-1}}\otimes_\rho W_{g^{-1}}$ is a $\beta_{g}^{-1}$-linear map. For the inner product property, it suffices to prove this on the dense submodule $E_A\otimes_A E_B$. Using linearity, it suffices to prove the property on simple tensors in $E_A\otimes_A E_B$. Given $x_1\dot{\otimes}y_1, x_2\dot{\otimes}y_2\in E_A\otimes_AE_B$, we have
		\begin{align*}
			\langle (U_g\dot{\otimes}W_g)(x_1\dot{\otimes}y_1), x_2\dot{\otimes}y_2\rangle&=\langle W_gy_1, \rho(\langle U_gx_1, x_2\rangle) y_2\rangle\\
			&=\beta_g(\langle y_1, W_{g^{-1}}\rho(\alpha_g(\langle x_1, U_{g^{-1}}x_2\rangle))y_2\rangle\\
			&=\beta_g(\langle y_1, \rho(\langle x_1, U_{g^{-1}}x_2)W_{g^{-1}}y_2\rangle)\\
			&=\beta_g\langle x_1\dot{\otimes}y_1, (U_{g^{-1}}\dot{\otimes}W_{g^{-1}})(x_2\dot{\otimes}y_2)\rangle 
		\end{align*}
		Thus $U_g\dot{\otimes} W_g$ is $\beta_g$-adjointable map with adjoint $U_{g^{-1}}\otimes_\rho W_{g^{-1}}$. As $$U_gU_{g^{-1}}=1_{E_A}=U_{g^{-1}}U_g\quad\quad W_gW_{g^{-1}}=1_{E_B}=W_{g^{-1}}W_g$$
		it follows $U_g\dot{\otimes}W_g$ is a $\beta_g$-unitary map. Finally, as $E_A\otimes_A E_B$ is dense in $E_A\otimes_\rho E_B$, any continuous $\beta_g$-linear map on $E_A\otimes_\rho E_B$ is completely determined by evaluation on simple tensors in $E_A\otimes_A E_B$. Thus $U_g\otimes_\rho W_g$ is unique.
		
		Now we prove $(2)$. Let $g, h\in G$, and let $x\dot{\otimes}y\in E_A\otimes_\rho E_B$. Then
		\begin{align*}
			(U_g\otimes_\rho W_g)(U_h\otimes_\rho W_h)(x\dot{\otimes}y)&=U_gU_hx\dot{\otimes}W_gW_hy\\
			&=U_{gh}x\dot{\otimes}W_{gh}y=(U_{gh}\otimes_\rho W_{gh})(x\dot{\otimes}y)
		\end{align*}
		Therefore, by the uniqueness in $(1)$, we have 
		$$(U_g\otimes_\rho W_g)(U_h\otimes_\rho W_h)=U_{gh}\otimes_\rho W_{gh}$$
		Hence 
		$$U\otimes_\rho W:G\to \mathcal{B}(E_A\otimes_\rho E_B)\quad\quad g\mapsto U_g\otimes_\rho W_g$$ 
		is a group homomorphism. 
		
		Now we prove $(3)$. Let $g\in G$, and let $T\in \mathcal{L}(E_A)$. Given $x\in E_A$ and $y\in E_B$, we have
		\begin{align*}
			(U_gTU_{g^{-1}}\otimes_\rho I)(x\dot{\otimes}y)&=U_gTU_{g^{-1}}x\dot{\otimes}y=U_gTU_{g^{-1}}x\dot{\otimes}W_gW_{g^{-1}}y\\
			&=(U_g\otimes_\rho W_g)(T\otimes_\rho I)(U_{g^{-1}}\otimes_\rho W_g^{-1})(x\dot{\otimes}y)
		\end{align*}
		As $E_A\otimes_A E_B$ is dense in $E_A\otimes_\rho E_B$ and both $U_gTU_{g^{-1}}\otimes_\rho I$ and $(U_g\otimes_\rho W_g)(T\otimes_\rho I)(U_{g^{-1}}\otimes_\rho W_g^{-1})$ are continuous linear maps, the equality above implies 
		$$(\text{Ad}(U_g)T)\otimes_\rho I=\text{Ad}(U_g\otimes_\rho W_{g})(T\otimes_\rho I)$$
		Hence the unital $*$-algebra map is $G$-equivariant.
		
		Finally, we prove $(4)$. Using $(1)$ and $(2)$, the data $((B, G, \beta), E_A\otimes_\rho E_B, U\otimes_\rho W)$ defines a Hilbert module dynamical system. Now suppose $((A, G, \alpha), E_A, U)$ is a continuous Hilbert module dynamical system and $((E_B, \rho), W)$ is a continuous equivariant $C^*$-correspondence; that is, $U$ and $W$ are strongly continuous. Let $(g_\lambda)_{\lambda\in\Lambda}$ be a net in $G$ which converges to $g\in G$. As $(U_{g_\lambda}\otimes_\rho W_{g_\lambda})_{\lambda\in\Lambda}$ is uniformly bounded by one and $E_A\otimes_A E_B$ is dense in $E_A\otimes_\rho E_B$, then, by the standard approximation argument, it suffices to prove $(U_{g_\lambda}\otimes_\rho W_{g_\lambda})_{\lambda\in\Lambda}$ converges strongly to $U_g\otimes_\rho W_g$ on $E_A\otimes_\rho E_B$. Using linearity, it suffices to prove the converges on simple tensors. Given $x\dot{\otimes}y\in E_A\otimes_A E_B$ and $\lambda\in\Lambda$, we have
		\begin{align*}
			&\left\|(U_{g_\lambda}\otimes_\rho W_{g_\lambda})(x\dot{\otimes}y)-(U_g\otimes_\rho W_g)(x\dot{\otimes}y)\right\|\\
			&=\left\|(U_{g_\lambda}\otimes_\rho W_{g_\lambda})(x\dot{\otimes}y)+(U_{g_\lambda}\otimes_\rho W_{g})(x\dot{\otimes}y)-(U_{g_\lambda}\otimes_\rho W_{g})(x\dot{\otimes}y)-(U_g\otimes_\rho W_g)(x\dot{\otimes}y)  \right\|\\
			&\leq\left\| (U_{g_\lambda}\otimes_\rho W_{g_\lambda})(x\dot{\otimes}y)-(U_{g_\lambda}\otimes_\rho W_{g})(x\dot{\otimes}y)\right\|+\left\|(U_{g_\lambda}\otimes_\rho W_{g})(x\dot{\otimes}y)-(U_g\otimes_\rho W_g)(x\dot{\otimes}y)\right\|\\
			&=\left\|\langle (W_{g_\lambda}-W_g)y, \rho(\langle U_{g_\lambda}x, U_{g_\lambda}x\rangle)(W_{g_\lambda}-W_g)y \rangle\right\|^{\frac{1}{2}}+\left\|\langle W_gy, \rho(\langle (U_{g_\lambda}-U_g)x, (U_{g_\lambda}-U_g)x\rangle)W_gy \rangle\right\|^{\frac{1}{2}}\\
			&\leq\left\|\rho\right\|^{\frac{1}{2}}\left\|x\right\|^{\frac{1}{2}}\left\|W_{g_\lambda}y-W_gy\right\|+\left\|\rho\right\|^{\frac{1}{2}}\left\|U_{g_\lambda}x-U_gx\right\|^{\frac{1}{2}}\left\|y\right\|
		\end{align*}
		Using $U$ and $W$ are strongly continuous, the inequality above implies $(U_{g_\lambda}x\dot{\otimes} W_{g_\lambda}y)_{\lambda\in\Lambda}$ converges to $U_gx\dot{\otimes}W_gy$. Hence we conclude $U\otimes_\rho W$ is strongly continuous.
		
	\end{proof}	
	
	For non-degenerate positive $C^*$-correspondences, we can apply the KSGNS construction to dilate the correspondence to a $C^*$-correspondence. As non-degenerate positive equivariant $C^*$-correspondences are like non-degenerate positive $C^*$-correspondence, it is reasonable to ask whether this dilation result extends to the equivariant setting. A positive answer to this question was given by Kasparov in the case $A$ is separable and $B$ is $\sigma$-unital (Theorem 3 of \cite{kasparov1980hilbert}) with subsequent results in the case $B=\mathbb{C}$ and $A$ a general $C^*$-algebra (Theorem 2.1 in \cite{Paulsen1982}). In the case of arbitrary $C^*$-algebras and $C^*$-dynamical systems, there is still a positive answer which we record below.
	
	\begin{theorem}[Theorem 6.11 in \cite{brady2026functorialityksgnsconstructionintertwiners}]\label{Thrm: Equiv Dil}\

		\noindent Let $(A, G, \alpha)$ and $(B, G, \beta)$ be $C^*$-dynamical systems. If $((E, \phi), U)$ is a  non-degenerate positive equivariant $C^*$-correspondence from $(A, G, \alpha)$ to $(B, G, \beta)$, then there exists a quadruple $(F_\phi, \pi_\phi, V_\phi, \widetilde{U})$ where
		\begin{enumerate}
			\item $F_\phi$ is a Hilbert $B$-module.
			\item $\pi_\phi:A\to \mathcal{L}(F_\phi)$ is a $*$-algebra homomorphism.
			\item $V_\phi:E\to F_\phi$ is an adjointable $B$-linear map.
			\item $\widetilde{U}:G\to \mathcal{B}(F_\phi)$ is a  representation for which $\widetilde{U}_g\in \mathcal{U}^{\beta_g}(F_\phi)$ for all $g\in G$.
		\end{enumerate}
		and the quadruple satisfies the following conditions:
		\begin{enumerate}
			\item $((F_\phi, \pi_\phi), \widetilde{U})$ is an equivariant $C^*$-correspondence from $(A, G, \alpha)$ to $(B, G, \beta)$.
			\item $\pi_\phi(A)V_\phi E$ is dense in $F_\phi$.
			\item For all $a\in A$, $\phi(a)=V_\phi^*\pi_\phi(a)V_\phi$.
			\item For all $g\in G$, $V_\phi\circ U_g=\tilde{U}_g\circ V_\phi$ 
		\end{enumerate}
		If $(F', \pi', V', U')$ is another quadruple satisfying conditions $(1)-(4)$, then there exists a $B$-linear unitary $W:F'\to F_\phi$ such that
		\begin{enumerate}
			\item $\pi_\phi(a)=W\pi'(a)W^*$ for all $a\in A$.
			\item $V_\phi=WV'$
			\item For all $g\in G$, $\widetilde{U}_g=WU'_gW^*$
		\end{enumerate}
		
	\end{theorem}
	
	In the case the $C^*$-dynamical systems are continuous and the representation $U$ is strongly continuous, the resulting representation $\tilde{U}$ is also strongly continuous (Theorem 6.10 in \cite{brady2026functorialityksgnsconstructionintertwiners}). Additionally, similar to the KSGNS construction, condition $(3)$ on the quadruple and non-degeneracy of $\phi$ imply $V_\phi$ is an isometry.

    \section{Hilbert Module Classical Reference Frames}

In this section we present a model of classical reference frames using commutative $C^*$-algebras. We begin with a brief review of Gelfand duality for commutative $C^*$-algebras. From this duality, we motivate the use of Markov kernels to model one commutative $C^*$-algebra being utilized as a background for another commutative $C^*$-algebra. Using Markov kernels and conditional probability theory, we generate an appropriate notion of a relational joint space, a relational joint system, and invariant relational observables. From this construction, we formulate a definition of a Hilbert module classical reference frame as well as present various examples. We conclude the section with a discussion on how our model of classical reference frames relates to the frame bundle perspective of classical reference frames. 

Throughout this section, we use the following terminology and notational conventions. Let $X$ be a locally compact Hausdorff topological space.
\begin{itemize}
    \item Denote the Borel $\sigma$-algebra on $X$ as $\text{Bor}(X)$.
    \item Denote the set of regular complex measures on $(X, \text{Bor}(X))$ as $\mathcal{M}_r(X, \text{Bor}(X))$.
    \item For each $g\in C_b(X)$, let
$$\Phi_g:\mathcal{M}_r(X, \text{Bor}(X))\to \mathbb{C}\quad\quad \Phi_g(\mu)=\int_{X} g(x)\;d\mu(x)$$
Recall, the weak topology on  $\mathcal{M}_r(X, \text{Bor}(X))$ is the weakest topology making the family $\{\Phi_{g}\}_{g\in C_b(X)}$ of linear functionals continuous. 
\item Denote the set of Radon probability measures on $X$ as $\text{Prob}_R(X,\text{Bor}(X))$. Recall, we have  $\text{Prob}_R(X,\text{Bor}(X))\subset \mathcal{M}_r(X, \text{Bor}(X))$. Define the weak topology on $\text{Prob}_R(X, \text{Bor}(X))$ as the subspace topology from $\mathcal{M}_r(X, \text{Bor}(X))$.

\end{itemize}

\subsection{Review of Gelfand Duality}
Let $\text{lcHTop}$ denote the category of locally compact Hausdorff topological spaces with morphisms given by continuous maps, and let $\text{ComC}^*\text{-Alg}_{\text{nd}}$ denote the category of commutative $C^*$-algebras with morphisms $A\to B$ given by non-degenerate $*$-algebra homomorphisms $A\to M(B)$. From Gelfand Duality, we obtain two contravariant functors
\begin{align*}
\text{lcHTop}&\to\text{ComC}^*\text{-Alg}_{\text{nd}} &   \text{ComC}^*\text{-Alg}_{\text{nd}}&\to\text{lcHTop}\\
X&\mapsto C_0(X) & A&\mapsto \Omega(A)\\
\left(f:X\to Y\right)&\mapsto \left(f^*:C_0(Y)\to C_b(X)\right) & (\rho:A\to M(B))&\mapsto (\rho^*:\Omega(B)\to \Omega(A))
\end{align*}
which form an equivalence of categories. Thus, given a commutative $C^*$-algebra $C_0(X)$, the equivalence implies that for each $\alpha\in \text{Aut}_{*\text{-alg}}(C_0(X))$, there exists a unique $\tilde{\alpha}\in \text{Homeo}(X)$ such that for all $f\in C_0(X)$ and $x\in X$, $\alpha(f)(x)=f(\tilde{\alpha}(x))$. Therefore, we obtain a bijective group homomorphism
$$\text{Aut}_{*\text{-alg}}(C_0(X))\to \text{Homeo}(X)\quad\quad \alpha\mapsto \tilde{\alpha}^{-1}$$
When $\text{Aut}_{*\text{-alg}}(C_0(X))$ is equipped with the point-norm topology and $\text{Homeo}(X)$ is equipped with a modified version of the compact-open topology, both groups form topological groups, and the bijective group homomorphism is a homeomorphism (see \cite{williams2007crossed} for additional details).

\begin{comment}

Let $G$ be a group. Given a group homomorphism
$$\gamma:G\to \text{Aut}_{*\text{-alg}}(C_0(X))\quad\quad g\mapsto \gamma_g,$$
we obtain for each $g\in G$ a unique $\tilde{\gamma}_g\in\text{Homeo}(X)$ such that for all $x\in X$ and $f\in C_0(X)$,
$$\gamma_g(f)(x)=f(\tilde{\gamma}_g(x))$$
Using properties of $\gamma$ as well as uniqueness of $\tilde{\gamma}_g$, the map
$$\tilde{\gamma}:G\to\text{Homeo}(X)\quad\quad g\mapsto \tilde{\gamma}_{g^{-1}}$$
defines a group homomorphism. From the homeomorphism between $\text{Aut}_{*\text{-alg}}(C_0(X))$ and $\text{Homeo}(X)$, the map $\gamma$ is continuous with respect to the point-norm topology if and only if $\tilde{\gamma}$ is continuous with respect to the modified version of compact open topology.
    
\end{comment}

\subsection{Viewing one space from the perspective of another space: Markov kernels}   

An important conceptual component for how we model classical/quantum reference frames is that the reference system acts a background for understanding the original system. While this statement can be understood intuitively, the notion of a "background" is not mathematically defined especially when working with $C^*$-algebras. In the commutative setting, we can understand what it means for a commutative $C^*$-algebra to be a background for another commutative $C^*$-algebra via Gelfand Duality.

Let $C_0(X)$ be a commutative $C^*$-algebra for the main system, and let $C_0(Y)$ be the commutative $C^*$-algebra for the reference system. Having $C_0(Y)$ as a background for $C_0(X)$ requires us to model what it means to view $C_0(X)$ from the perspective of $C_0(Y)$. Algebraically, it is not apparent how we would model an answer to this inquiry. In the spirit of Gelfand Duality, we can adapt our question to the locally compact Hausdorff spaces $X$ and $Y$: how do we model viewing $X$ from the perspective of $Y$. Using probability theory, we can model such a relationship using Markov kernels to condition $X$ by $Y$.

\begin{definition}\

\noindent Let $X$ and $Y$ be locally compact Hausdorff topological spaces. A Markov kernel from $Y$ to $X$ is the data of a map $\kappa:Y\times\text{Bor}(X)\to [0, 1]$ such that
\begin{enumerate}
    \item for all $y\in Y$, $\kappa(y, \cdot):\text{Bor}(X)\to [0, 1]$ defines a Radon probability measure on $(X, \text{Bor}(X))$, and
    \item for all $B\in \text{Bor}(X)$, the function $\kappa(\cdot, B):Y\to [0, 1]$ is Borel measurable.
\end{enumerate}
The Markov kernel is weakly continuous if the induced map $Y\xrightarrow{y\mapsto \kappa(y, \cdot)} \text{Prob}_R(X, \text{Bor}(X))$ is continuous with respect to the weak topology on $\text{Prob}_R(X, \text{Bor}(X))$.
    
\end{definition}

\begin{example}\

\noindent Let $X$ and $Y$ be finite discrete sets with $|X|=n$ and $|Y|=m$. As $Y$ has the discrete topology, every Markov kernel  from $Y$ to $X$ is weakly continuous. Furthermore, as $X$ and $Y$ are both discrete finite sets, Markov kernels from $Y$ to $X$ are in one-to-one correspondence with $n$-by-$m$ real valued matrices for which each column sums to one and each column contains non-negative entries. In particular, if $X=\{x_1, \ldots, x_n\}$, $Y=\{y_1, \ldots, y_m\}$, and we have such a matrix $[a^i_j]$, then the associated Markov kernel is given by
$$\kappa:Y\times\text{Bor}(X)\to [0, 1]\quad\quad \kappa(y_j, B)=\sum\limits_{k=1}^n a^k_j\delta_{x_k}(B)$$
where $\delta_{x_k}$ is the Dirac measure at $x_k$.

\end{example}

\begin{example}\

\noindent Let $X$ and $Y$ be locally compact Hausdorff spaces. Let $f:Y\to X$ be a Borel measurable function, then the map
$$\kappa:Y\times\text{Bor}(X)\to[0, 1]\quad\quad \kappa(y, B)=\delta_{f(y)}(B)$$
defines a Markov kernel. Using Urysohn's Lemma, $\kappa$ is weakly continuous if and only if $f$ is continuous. 

\begin{comment}
    Indeed, if $f$ is continuous, then for all $g\in C_b(X)$, we have the map
$$Y\to \mathbb{C}\quad\quad y\mapsto \int_X g(x)\;\kappa(y, dx)=(g\circ f)(y)$$
Since $g\circ f$ is continuous, then it follows $\kappa$ is weakly continuous. Now suppose $\kappa$ is weakly continuous, and fix $y_0\in Y$. Let $U\subset X$ be an open subset with $f(y_0)\in U$. Using Urysohn's Lemma, there exists a continuous function $g:X\to [0, 1]$ such that $g(y_0)=1$ and $g(X\backslash U)=0$. As $\kappa$ is weakly continuous and $g\in C_b(X)$, the map
$$Y\to [0, 1]\quad\quad y\mapsto \int_X g(x)\;\kappa(y, dx)=(g\circ f)(y)$$
is continuous. Let $V=(g\circ f)^{-1}((0.5, 1])$ which is an open subset of $Y$ with $y_0\in V$. Furthermore, $f(V)\subset U$. Thus, as $y_0\in Y$ was arbitrary, we conclude $f$ is continuous.
\end{comment}

\end{example}

\begin{example}\

\noindent Taking $X=\mathbb{S}^1=Y$ where we identify $\mathbb{S}^1=[0, 2\pi]/_{0\sim 2\pi}$, the map
$$\kappa:\mathbb{S}^1\times\text{Bor}(\mathbb{S}^1)\to [0, 1]\quad\quad \kappa(\varphi, B)=\frac{1}{2\pi}\int_B 1+\sin(\theta-\varphi)\;d\theta$$
defines a weakly continuous Markov kernel.

\end{example}

\begin{example}\

\noindent Let $X=\mathbb{R}=Y$. For each $\sigma\in (0, \infty)$, the map
$$\kappa:\mathbb{R}\times\text{Bor}(\mathbb{R})\to [0, 1]\quad\quad \kappa(y, B)=\frac{1}{\sqrt{2\pi \sigma}}\int_{B} \exp\left(\frac{-(x-y)^2}{2\sigma}\right) dx$$
defines a weakly continuous Markov kernel.
    
\end{example}

Using integration and the Riesz-Markov-Kakutani Representation Theorem, we obtain a bijective correspondence between weakly continuous Markov kernels and non-degenerate completely positive maps. Using this correspondence, we obtain an appropriate answer to the original question of modeling a commutative $C^*$-algebra acting as a background for another commutative $C^*$-algebra. In the case the spaces are compact Hausdorff topological spaces, this correspondence is explicitly written in the literature such as in \cite{parzygnat2017discrete}. However, the locally compact Hausdorff case does not appear to be written up. Therefore, we present the bijective correspondence for locally compact Hausdorff spaces below. As the proof of the correspondence is somewhat technical and sidetracks our development of classical reference frames, we include the proof of the statements in the following proposition in Appendix \ref{App; A}.

\begin{proposition}\label{Prop: bij cor btw MK and NCP maps}\

\noindent Let $X$ and $Y$ be locally compact Hausdorff topological spaces. 
\begin{enumerate}
    \item If $\kappa:Y\times\text{Bor}(X)\to [0, 1]$ is a weakly continuous Markov kernel, then the map
$$\phi_{\kappa}:C_0(X)\to C_b(Y)\quad\quad \phi_{\kappa}(f)(y)=\int_{X}f(x)\;\kappa(y, dx)$$
is a well-defined non-degenerate completely positive map.
\item  Let $\phi:C_0(X)\to C_b(Y)$ be a non-degenerate completely positive map. For each $y\in Y$, let $\mu_y\in \mathcal{M}_r(X, \text{Bor}(X))$ be the unique positive regular measure associated to positive linear functional
$$\text{ev}_y\circ \phi:C_0(X)\to \mathbb{C}\quad\quad (\text{ev}_y\circ\phi)(f)=\phi(f)(y)$$ from the Riesz-Markov-Kakutani Representation Theorem. Then the map
$$\kappa_\phi:Y\times\text{Bor}(X)\to [0, 1]\quad\quad \kappa_\phi(y, B)=\mu_y(B)$$
defines a weakly continuous Markov kernel which uniquely satisfies the following condition: for all $f\in C_0(X)$ and $y\in Y$,
$$\phi(f)(y)=\int_{X}f(x)\;\kappa_\phi(y, dx).$$
\item The following assignments are well-defined and inverses of each other:
$$\left\{\begin{matrix} 
  \text{weakly continuous}\\ 
  \text{Markov kernels}\\ 
  \kappa:Y\times \text{Bor}(X)\to [0, 1] 
\end{matrix}\right\}
\xrightleftharpoons[\displaystyle\kappa_\phi\mapsfrom \phi]{\quad\displaystyle\kappa\mapsto \phi_\kappa\quad}
\left\{\begin{matrix} 
  \text{non-degenerate completely}\\ 
  \text{positive maps}\\ 
  \phi:C_0(X)\to C_b(Y) 
\end{matrix}\right\}$$

\end{enumerate}
    
\end{proposition}

Now let us suppose the algebras are carrying symmetries from a group $G$; that is, we have $C^*$-dynamical systems $(C_0(X), G, \alpha)$ and $(C_0(Y), G, \beta)$. As before, we want to answer the question of how we model viewing $C_0(X)$ from the perspective of $C_0(Y)$; however, as we now have symmetries from $G$, we modify the question to ask how to model viewing $C_0(X)$ from the perspective of $C_0(Y)$ while preserving the symmetries of $\alpha$ and $\beta$. Translating the question in terms of the topological spaces via Gelfand Duality, we look to model viewing $X$ through $Y$ while preserving the symmetries from the induced actions $\tilde{\alpha}$ and $\tilde{\beta}$ on $X$ and $Y$, respectively. As our answer to the reformulated inquiry should reduce to our answer when $G$ is trivial, we look to use Markov kernels to condition $X$ by $Y$. To preserve the symmetries on $X$ and $Y$, we require the conditioning to be $G$-invariant; that is, we utilize $G$-invariant Markov kernels.

\begin{definition}\

\noindent Let $X$ and $Y$ be locally compact Hausdorff spaces, and let $G$ be a group acting on $X$ and $Y$ by $\tilde{\alpha}$ and $\tilde{\beta}$, respectively. A Markov kernel  $\kappa:Y\times \text{Bor}(X)\to [0, 1]$ is $G$-invariant if for all $g\in G$, $y\in Y$, and $B\in\text{Bor}(X)$, $\kappa(\tilde{\beta}_g(y), \tilde{\alpha}_g(B))=\kappa(y, B)$.

\end{definition}

\begin{example}\

\noindent Let $X=\{x_1, x_2\}$ and $Y=\{y_1, y_2\}$. Let $\tilde{\alpha}$ and $\tilde{\beta}$ be the unique transitive $\mathbb{Z}/2\mathbb{Z}$ actions on $X$ and $Y$, respectively. We claim, with respect to these actions, every $\mathbb{Z}/2\mathbb{Z}$-invariant Markov kernel from $Y$ to $X$ is of the form
$$\kappa_p:Y\times \text{Bor}(X)\to [0, 1]\quad\quad \kappa_p(y, B)=\begin{cases}
    (1-p)\delta_{x_1}(B)+p\delta_{x_2}(B) & y=y_1\\
    p\delta_{x_1}(B)+(1-p)\delta_{x_2}(B) & y=y_2
\end{cases}$$
for $p\in [0, 1]$. First, it is clear that $\kappa_p$ is a $\mathbb{Z}/2\mathbb{Z}$-invariant Markov kernel. To see that all such kernels are of this form, suppose $\kappa$ is a $\mathbb{Z}/2\mathbb{Z}$-invariant Markov kernel, then we can uniquely write $\kappa$ as
$$\kappa(y, B)=\begin{cases}
    a\delta_{x_1}(B)+b\delta_{x_2}(B) & y=y_1\\
    c\delta_{x_1}(B)+d\delta_{x_2}(B) & y=y_2
\end{cases}$$
where $a, b, c, d\in \mathbb{R}$. As $\kappa$ is a Markov kernel, we have $a+b=1=c+d$. As $\kappa$ is $\mathbb{Z}/2\mathbb{Z}$ invariant, we have $a=d$ and $c=b$. Therefore 
$$\kappa(y, B)=\begin{cases}
    (1-b)\delta_{x_1}(B)+b\delta_{x_2}(B) & y=y_1\\
    b\delta_{x_1}(B)+(1-b)\delta_{x_2}(B) & y=y_2
\end{cases}$$
for $b\in \mathbb{R}$. As we must have $1-b, b\geq 0$, it follows $b\in [0, 1]$. Thus $\kappa=\kappa_b$.

\end{example}

\begin{example}\

\noindent Let $X=\{x_0, \ldots, x_{n-1}\}$ and $Y=\{y_0, \ldots, y_{m-1}\}$ with $n, m>1$ and $\text{gcd}(n, m)=1$. Take $G=\mathbb{Z}/n\mathbb{Z}\times \mathbb{Z}/m\mathbb{Z}$. As $\text{gcd}(n,m)=1$, then $([1], [1])$ generates $G$. Define the action of $G$ on $X$ and $Y$ generated by
$$([1], [1])\cdot x_i=x_{i+1\text{ mod }n}\quad\quad ([1], [1])\cdot y_{j}=y_{j+1\text{ mod }m}$$

We claim the only $G$ invariant Markov kernel from $Y$ to $X$ with respect to the actions above is 
$$\kappa:Y\times \text{Bor}(X)\to [0, 1]\quad\quad \kappa(y_j, B)=\frac{1}{n}\sum\limits_{k=0}^{n-1}\delta_{x_k}(B)$$ First, it is clear the Markov kernel above is $G$-invariant. To see that this is the only $G$-invariant Markov kernel, suppose $\kappa:Y\times \text{Bor}(X)\to [0, 1]$ is a $G$-invariant Markov kernel with associated matrix $[a^i_j]$. In order for $\kappa$ to be $G$-invariant, we must have for each $0\leq i\leq n-1$ and $0\leq j\leq m-1$
\begin{align*}
    a^{1+i}_{1+j}&=\kappa(y_j, \{x_i\})=\kappa(([1], [1])\cdot y_j, ([1], [1])\cdot\{x_i\})\\
    &=\sum\limits_{k=0}^{n-1} a_{1+(j+1) \text{mod }m }^{1+k}\delta_{x_{k}}(\{x_{i+1\text{ mod }n}\})=a^{1+(i+1)\text{mod }n}_{1+(j+1)\text{mod }m }
\end{align*}
Using this as well as  $\text{gcd}(n, m)=1$, it follows $a^{i}_j=a^{k}_l$ for all $1\leq i, k\leq n$ and $1\leq j, l\leq m$. Since the columns of $[a^i_j]$ must sum to one, we have $a^i_j=\frac{1}{n}$ for each $1\leq i\leq n$ and $1\leq j\leq m$. Hence
$$\kappa(y_j, B)=\frac{1}{n}\sum\limits_{k=0}^{n-1} \delta_{x_k}(B).$$

\end{example}

\begin{example}\

\noindent Let $X=\mathbb{S}^1=Y$ where we identify $\mathbb{S}^1=[0, 2\pi]/_{0\sim 2\pi}$. Let $G=\mathbb{S}^1$, and let $\mathbb{S}^1$ act by multiplication on $\mathbb{S}^1$. Then the weakly continuous Markov kernel
$$\kappa:\mathbb{S}^1\times\text{Bor}(\mathbb{S}^1)\to [0, 1]\quad\quad \kappa(\varphi, B)=\frac{1}{2\pi}\int_B 1+\sin(\theta-\varphi)\;d\theta$$
is $\mathbb{S}^1$-invariant.

\end{example}

\begin{example}\

\noindent Let $X=\mathbb{R}=Y$. Take $G=\mathbb{R}$, and let $\mathbb{R}$ act on $\mathbb{R}$ via translation. For each $\sigma\in (0, \infty)$, the map
$$\kappa:\mathbb{R}\times\text{Bor}(\mathbb{R})\to [0, 1]\quad\quad \kappa(y, B)=\frac{1}{\sqrt{2\pi \sigma}}\int_{B} \exp\left(\frac{-(x-y)^2}{2\sigma}\right) dx$$
defines a $\mathbb{R}$-invariant weakly continuous Markov kernel.
    
\end{example}

Using Proposition \ref{Prop: bij cor btw MK and NCP maps}, we obtain a bijective correspondence between $G$-invariant weakly continuous Markov kernels and non-degenerate completely positive $G$-equivariant maps. Using this correspondence, we obtain an appropriate answer to the original question modeling $C_0(X)$ from the perspective of $C_0(Y)$ while preserving the symmetries from $\alpha $ and $\beta$.

\begin{proposition}\label{Prop: G-invariant Markov kernel stuff}\

\noindent Let $X$ and $Y$ be locally compact Hausdorff spaces, and let $(C_0(X), G, \alpha)$ and $(C_0(Y), G, \beta)$ be $C^*$-dynamical systems.
\begin{enumerate}
    \item If $\kappa:Y\times\text{Bor}(X)\to [0, 1]$ is a $G$-invariant weakly continuous Markov kernel, then the map $\phi_\kappa:C_0(X)\to C_b(Y)$ given in Proposition \ref{Prop: bij cor btw MK and NCP maps} is a non-degenerate completely positive $G$-equivariant map.
    \item If $\phi:C_0(X)\to C_b(Y)$ is a unital completely positive $G$-equivariant map, then the map $\kappa_\phi:Y\times \text{Bor}(X)\to [0, 1]$ given in Proposition \ref{Prop: bij cor btw MK and NCP maps} is a $G$-invariant weakly continuous Markov kernel.
    \item The following assignment are well-defined and inverses of each other:
\begin{align*}
    \left\{\begin{matrix}
   G \text{-invariant weakly continuous}\\
    \text{Markov kernels}\\
    \kappa:Y\times \text{Bor}(X)\to [0, 1]
\end{matrix}\right\}&\xrightleftharpoons[\displaystyle\kappa_\phi\mapsfrom \phi]{\quad\displaystyle\kappa\mapsto \phi_\kappa\quad} \left\{\begin{matrix}
    \text{non-degenerate completely positive}\\
   G\text{-equivariant maps}\\
    \phi:C_0(X)\to C_b(Y)
\end{matrix}\right\}
\end{align*}

\end{enumerate}

\end{proposition}

\begin{proof}\

\noindent The proof of $(3)$ follows from both $(1)$ and $(2)$ as well as from the assignments given in Proposition \ref{Prop: bij cor btw MK and NCP maps}. Therefore we only prove $(1)$ and $(2)$. Furthermore, using the induced maps from Proposition \ref{Prop: bij cor btw MK and NCP maps} and Proposition \ref{Prop: bij cor btw MK and NCP maps}, we only verify in $(1)$ that $\phi_\kappa$ is $G$-equivariant, and we only verify in $(2)$ that $\kappa_\phi$ is $G$-invariant.

We start by proving $(1)$. Let $\kappa:Y\times\text{Bor}(X)\to [0, 1]$ be a $G$-invariant weakly continuous Markov kernel. To verify $\tilde{\kappa}$ is $G$-equivariant, let $f\in C_0(X)$, $y\in Y$, and $g\in G$; then, we have
\begin{align*}
    \phi_\kappa(\alpha_g(f))(y)&=\int_{X} \alpha_g(f)(x)\;\kappa(y, dx)=\int_X f(\tilde{\alpha}_{g^{-1}}(x))\;\kappa(y, dx)\\
    &=\int_X f(x)\;d(\kappa(y, \cdot)\circ\tilde{\alpha}_{g})(x)=\int_X f(x)\;\kappa(\tilde{\beta}_{g^{-1}}(y), dx)\\
    &=\phi_\kappa(f)(\tilde{\beta}_{g^{-1}}(y))=\beta_g(\phi_\kappa(f))(y)
\end{align*}
Therefore $\phi_\kappa(\alpha_g(f))=\beta_g(\phi_\kappa(f))$ for all $f\in C_0(X)$ and $g\in G$; that is, $\phi_\kappa$ is $G$-equivariant.

We now prove $(2)$. Let $\phi:C_0(X)\to C_b(Y)$ be a non-degenerate completely positive $G$-equivariant map. For each $y\in Y$, let $\mu_y\in\text{Prob}_r(X, \text{Bor}(X))$ be the unique Radon probability measure such that for all $f\in C_0(X)$
$$\phi(f)(y)=\int_X f(x)\;d\mu_y(x).$$
Given $f\in C_0(X)$, $y\in Y$, and $g\in G$, we have
\begin{align*}
    \phi(\alpha_g(f))(y)&=\int_X \alpha_g(f)(x)\;d\mu_y(x)=\int_{X}f(\tilde{\alpha}_{g^{-1}}(x))\;d\mu_y(x)=\int_Xf(x) \;d(\mu_y\circ\tilde{\alpha}_{g})(x)
\end{align*}
As $\phi$ is $G$-equivariant, we also have
\begin{align*}
    \phi(\alpha_g(f))(y)&=\beta_g(\phi(f))(y)=\phi(f)(\tilde{\beta}_{g^{-1}}(y))=\int_{X}f(x)\;d\mu_{\tilde{\beta}_{g^{-1}}(y)}(x)
\end{align*}
Therefore
$$\phi(f)(\tilde{\beta}_{g^{-1}}(y))=\int_{X}f(x)\;d\mu_{\tilde{\beta}_{g^{-1}}(y)}(x)=\int_Xf(x) \;d(\mu_y\circ\tilde{\alpha}_g)(x)$$
for all $f\in C_0(X)$, $y\in Y$, and $g\in G$. From the uniqueness of $\mu_y$, we know $\mu_y\circ \tilde{\alpha}_{g}=\mu_{\tilde{\beta}_{g^{-1}}(y)}$ for all $y\in Y$ and $g\in G$. Therefore, given $y\in Y$, $B\in\text{Bor}(X)$, and $g\in G$, we have
\begin{align*}
    \kappa_\phi(\tilde{\beta}_g(y), \tilde{\alpha}_g(B))&=\mu_{\tilde{\beta}_g(y)}(\tilde{\alpha}_g(B))=\mu_y(\tilde{\alpha}_{g^{-1}}(\tilde{\alpha}_g(B)))=\mu_y(B)=\kappa_\phi(y, B)
\end{align*}
Hence $\kappa_\phi$ is a $G$-invariant weakly continuous Markov kernel.

\end{proof}
    
\subsection{Building relational invariant observables from a Markov kernel}

Let $C_0(X)$ and $C_0(Y)$ be $C^*$-algebras for the main system and reference system, respectively. Given a non-degenerate completely positive map $\phi:C_0(X)\to C_b(Y)$ encoding how $C_0(Y)$ is a background for $C_0(X)$, we want to build a $C^*$-algebra for which the observables represent relational observables. As remarked in the introduction, this algebra is given by a joint system which remembers that $C_0(Y)$ is the background. Using the weakly continuous Markov kernel associated to $\phi$, denoted by $\kappa_\phi$, we are able to construct such an algebra as the $C^*$-algebra of adjointable operators on a Hilbert module.

Let us begin with the usual joint system of $C_0(X)$ and $C_0(Y)$ given by the $C^*$-algebra $$C_0(X)\otimes_{*}C_0(Y)\cong C_0(X\times Y).$$
Under pointwise multiplication, we make $ C_0(X\times Y)$ into a $C_0(Y)$-module. Using $\kappa_\phi$, we define a $C_0(Y)$-valued pairing given by
$$\langle\cdot, \cdot\rangle: C_0(X\times Y)\times C_0(X\times Y)\to C_0(Y)\quad\quad \langle f, g\rangle(y)=\int_{X}\overline{f(x, y)}g(x, y)\;\kappa_\phi(y, dx).$$
Using this pairing, $C_0(X\times Y)$ becomes a pre-inner product $C_0(Y)$-module (see Lemma \ref{Lemma: building Gamma_0}). Note, we do not obtain an inner product module as the pairing is not necessarily non-degenerate; indeed, consider the case $X=[0, 1]=Y$, $\kappa(y, B)=\delta_{0}(B)$, and $f(x, y)=x$. 

To obtain an inner product $C_0(Y)$-module, we quotient $C_0(X\times Y)$ by the submodule generated by elements in $C_0(X\times Y)$ with norm zero:
$$N_{\kappa_\phi}=\left\{f\in C_0(X\times Y): \int_{X}|f(x, y)|^2\;\kappa_\phi(y, dx)=0\text{ for all }y\in Y\right\}.$$
The equivalences classes in $C_0(X\times Y)/N_{\kappa_\phi}$ represent observables in the usual joint system which are indistinguishable relative to the background provided by $C_0(Y)$ and $\phi$. To obtain a $C^*$-algebra representing the joint system relative to the background, we complete the module $C_0(X\times Y)/N_{\kappa_\phi}$ to a Hilbert $C_0(Y)$-module, denoted by $\Gamma_{\kappa_\phi}$, and take the unital $C^*$-algebra $\mathcal{L}(\Gamma_{\kappa_\phi})$ as the joint system. We interpret $\Gamma_{\kappa_\phi}$ as a conditional/relational joint space for $C_0(X)$ and $C_0(Y)$, and we interpret observables in $\mathcal{L}(\Gamma_{\kappa_\phi})$ as conditional/relational observables. We note that in the case we have $C^*$-dynamical systems $(C_0(X), G, \alpha)$ and $(C_0(Y), G, \beta)$ as well as $\kappa_\phi$ is $G$-invariant, the actions of $\alpha $ and $\beta$ descend to defining a unique Hilbert module dynamical system on $\Gamma_{\kappa}$ (see Lemma \ref{Lemma: G-action on the module}). Thus the invariant relational observables are in the fixed point subalgebra $\mathcal{L}(\Gamma_{\kappa_\phi})^G$. 

We now determine in what sense $\Gamma_{\kappa_\phi}$ is unique. By viewing $C_0(Y)\subset C_0(X\times Y)$ and viewing $C_0(Y)$ as a Hilbert module over itself, we obtain an adjointable map $$V_{\kappa_\phi}:C_0(Y)\to \Gamma_{\kappa_\phi}\quad\quad V_{\phi}(f)=f+N_{\kappa_\phi}\in \Gamma_{\kappa_\phi}.$$
(see Lemma \ref{Lemma: KSGNS construction in terms of kappa}). Furthermore, using pointwise multiplication, we obtain a $*$-algebra homomorphism
$$C_0(X)\to \mathcal{L}(C_0(X\times Y))\quad\quad f\mapsto M_f$$
which descends to a well-defined $*$-algebra homomorphism $\pi_{\kappa_\phi}:C_0(X)\to \mathcal{L}(\Gamma_{\kappa_\phi})$ (see Lemma \ref{Lemma: KSGNS construction in terms of kappa}). The triple $(\Gamma_{\kappa_\phi}, \pi_{\kappa_\phi}, V_{\kappa_\phi})$ satisfies the conditions of the KSGNS construction for the non-degenerate positive $C^*$-correspondences $(C_0(Y), \phi)$ (see Lemma \ref{Lemma: KSGNS construction in terms of kappa}). In the case we have $C^*$-dynamical systems, the triple along with the induced group homomorphism $G\to \mathcal{B}(\Gamma_\kappa)$ satisfy the conditions for the quadruple when applying the equivariant KSGNS construction to the non-degenerate positive equivariant $C^*$-correspondence $((C_0(Y), \phi), \beta)$ (see Lemma \ref{Lemma: G-action on the module}). Thus the algebra of relational observables is unitarily uniquely given by the (equivariant) KSGNS dilation of the correspondence determined by $\phi$ which can be explicitly constructed in terms of the associated Markov kernel $\kappa_\phi$.

We now verify the various claims we made in the construction above.

\begin{lemma}\label{Lemma: building Gamma_0}\

\noindent Let $X$ and $Y$ be locally compact Hausdorff spaces, and let $\kappa$ be a weakly continuous Markov kernel from $Y$ to $X$. Then $C_0(X\times Y)$ forms a pre-inner product $C_0(Y)$-module under pointwise multiplication from $C_0(Y)$ and with respect to the $C_0(Y)$-valued pairing
$$\langle\cdot, \cdot\rangle: C_0(X\times Y)\times C_0(X\times Y)\to C_0(Y)\quad\quad \langle f, g\rangle (y)=\int_{X}\overline{f(x, y)}g(x, y)\;\kappa(y, dx).$$

\end{lemma}

\begin{proof}\

\noindent It is clear $C_0(X\times Y)$ is a $C_0(Y)$-module under pointwise addition, pointwise scalar multiplication, and pointwise multiplication from $C_0(Y)$. Provided the pairing is well-defined, it is clear the pairing makes $C_0(X\times Y)$ into into a pre-inner product $C_0(Y)$-module. Thus, we only verify the pairing is well-defined.

Let $f\in C_0(X\times Y)$, and let 
$$\tilde{\kappa}(f):Y\to\mathbb{C}\quad\quad \tilde{\kappa}(f)(y)=\int_{X}f(x, y)\;\kappa(y, dx).$$
Note, for all $f, g\in C_0(X\times Y)$, $\langle f, g\rangle=\kappa(\overline{f}g)$. Thus, to know the pairing is well-defined, it suffices to show for each $f\in C_0(X\times Y)$ that $\tilde{\kappa}(f)$ is a well-defined, continuous function which vanishes at infinity. Fix $f\in C_0(X\times Y)$. First, for each $y\in Y$, the map 
$$X\to \mathbb{C}\quad\quad x\mapsto f(x, y)$$
is a continuous and vanishes at infinity. As $\kappa(y, \cdot)$ is a finite Borel measure, it follows
$$\int_{X}f(x, y)\;\kappa(y, dx)$$ is well-defined for each $y\in Y$. Thus, as a function, $\tilde{\kappa}(f)(y)$ is well-defined. 

To see $\tilde{\kappa}(f)$ vanishes at infinity, let $\epsilon>0$. As $f\in C_0(X\times Y)$, there exists a compact set $K\subset X\times Y$ such that $|f(x, y)|<\epsilon$ for all $(x, y)\in (X\times Y)\backslash K$. Let $K_Y=\text{pr}_Y(K)$ which is a compact subset of $Y$. Fix $y\in Y\backslash K_Y$, then for all $x\in X$, $(x, y)\not\in K$. Therefore $|f(x, y)|<\epsilon$. Hence for all $y\in Y\backslash K_Y$,
$$|\tilde{\kappa}(f)(y)|\leq \int_{X}|f(x, y)|\;\kappa(y, dx)\leq \int_{X}\epsilon\;\kappa(y, dx)=\epsilon$$
showing $\tilde{\kappa}(f)$ vanishes at infinity.

Now we verify $\tilde{\kappa}(f)$ is continuous. Fix $y_0\in Y$, and let $\epsilon>0$. Our goal is to determine an open subset $U\subset Y$ with $y_0\in U$ such that for all $y\in U$, $|\tilde{\kappa}(f)(y)-\tilde{\kappa}(f)(y_0)|<\epsilon$. Note, using the Triangle Inequality, we have the bound
\begin{align*}
    |\tilde{\kappa}(f)(y)-\tilde{\kappa}(f)(y_0)|&=\left|\int_{X}f(x, y)\;\kappa(y, dx)-\int_{X}f(x, y_0)\;\kappa(y_0, dx)\right|\\
    &\leq \int_{X}|f(x, y)-f(x, y_0)|\;\kappa(y, dx)\\
    &\quad\quad\quad\quad\quad+\left|\int_{X}f(x, y_0)\;\kappa(y, dx)-\int_{X}f(x, y_0)\;\kappa(y_0, dx)\right|
\end{align*}
Therefore it suffices to find an open set $U$ with $y_0\in U$ for which the last two terms are bounded by $\frac{\epsilon}{2}$ for all $y\in U$. We will bound the first term using the fact that $f\in C_0(X\times Y)$, while the bound for the second term will come from the fact that $\kappa$ is weakly continuous.

We start by bounding the first term. As $f\in C_0(X\times Y)$, there exists  a compact subset $K\subset X\times Y$ such that $|f(x, y)|<\frac{\epsilon}{8}$ for all $(x, y)\in K$. Let $K_X=\text{pr}_X(K)$ which is compact in $X$. Given $x\not\in K_X$, we have $(x, y)\not\in K$ for all $y\in Y$. Therefore, for all $y\in Y$,
$$ \int_{X\backslash K_X}|f(x, y)-f(x, y_0)|\;\kappa(y, dx)<\frac{\epsilon}{4}.$$

As $f$ is continuous, for each $x\in K_X$, there exists an open subset $U_x\subset X$ with $x\in U_x$ and open subset $W_x\subset Y$ with $y_0\in W_x$ such that for all $(x', y)\in U_x\times W_x$
$$|f(x', y)-f(x, y_0)|<\frac{\epsilon}{8}.$$
As $\{U_x\}_{x\in K_X}$ forms an open cover of $K_X$, there exists a finite subcover which we denote by $\{U_{x_1}, \ldots, U_{x_n}\}$. Let $W=\bigcap\limits_{k=1}^{n}W_{x_k}$ which is an open subset of $Y$ with $y_0\in W$. Given $x\in K_X$, there exists $1\leq k\leq n$ such that $x\in U_{x_k}$. Thus, for all $(x, y)\in K_X\times W$, we have
$$|f(x, y)-f(x, y_0)|\leq |f(x, y)-f(x_k, y_0)|+|f(x_k, y_0)-f(x, y_0)|<\frac{\epsilon}{8}+\frac{\epsilon}{8}=\frac{\epsilon}{4}.$$
Thus, if $y\in W$,
$$\int_{K_X} |f(x, y)-f(x, y_0)|\;\kappa(y, dx)<\frac{\epsilon}{4}.$$
Therefore, for all $y\in W$,
\begin{align*}
    \int_{X}|f(x, y)-f(x, y_0)|\;\kappa(y, dx)&=\int_{K_X}|f(x, y)-f(x, y_0)|\;\kappa(y, dx)\\
    &\quad\quad\quad\quad\quad+\int_{X\backslash K_X}|f(x, y)-f(x, y_0)|\;\kappa(y, dx)<\frac{\epsilon}{2}
\end{align*}

To obtain the second bound, we utilize $\kappa$ is a weakly continuous Markov kernel. As $f\in C_0(X\times Y)$, the map
$$F:X\to \mathbb{C}\quad\quad F(x)= f(x, y_0)$$ is a continuous function which vanishes at infinity. Let
$$\Phi_F:\text{Prob}_R(X, \text{Bor}(X))\to \mathbb{C}\quad\quad \Phi_F(\mu)=\int_{X}F(x)\;d\mu(x)=\int_{X}f(x, y_0)\;d\mu(x)$$
With respect to the weak topology, $\Phi_F$ is continuous. Thus, $\Phi_F^{-1}(B_{\frac{\epsilon}{2}}(\Phi_F(\kappa(y_0, -))))$ is an open subset of $\text{Prob}_R(X, \text{Bor}(X))$ containing $\kappa(y_0, -)$. As $\kappa$ is a weakly continuous Markov kernel, the map
$$Y\to \text{Prob}_R(X, \text{Bor}(X))\quad\quad y\mapsto \kappa(y, \cdot)$$
is a continuous with respect to the weak topology. Thus, the pre-image of the open set $\Phi_F^{-1}(B_{\frac{\epsilon}{2}}(\Phi_F(\kappa(y_0, -))))$ under this map yields an open set $W'\subset Y$ with $y_0\in W'$. Note, for all $y\in W'$, we have
$$\left|\int_{X}f(x, y_0)\;\kappa(y, dx)-\int_{X}f(x, y_0)\;\kappa(y_0, dx)\right|<\frac{\epsilon}{2}$$

Let $U=W\cap W'$ which is an open subset of $Y$ with $y_0\in U$. For all $y\in U$, we have
$$\int_{X}|f(x, y)-f(x, y_0)|\;\kappa(y, dx)+\left|\int_{X}f(x, y_0)\;\kappa(y, dx)-\int_{X}f(x, y_0)\;\kappa(y_0, dx)\right|<\epsilon$$
showing $|\tilde{\kappa}(f)(y)-\tilde{\kappa}(f)(y_0)|<\epsilon$. Therefore $\tilde{\kappa}(f)(U)\subset B_{\epsilon}(\tilde{\kappa}(f)(y_0))$. As $y_0\in Y$ was arbitrary, we conclude $\tilde{\kappa}(f)$ is continuous.

\end{proof}

\begin{lemma}\label{Lemma: KSGNS construction in terms of kappa}\

\noindent Let $X$ and $Y$ be locally compact Hausdorff topological spaces, and let $\kappa$ be a weakly continuous Markov kernel with associated non-degenerate completely positive map $\phi_\kappa: C_0(X)\to C_b(Y)$. 
\begin{enumerate}
    \item For each $f\in C_0(X)$, the map
$$C_0(X\times Y)/N_\kappa\to C_0(X\times Y)/N_\kappa\quad\quad g+N_\kappa\mapsto  fg+N_{\kappa}$$
defines a bounded, $C_0(Y)$-linear map which extends uniquely to an adjointable map $\pi_\kappa(f)$ on $\Gamma_\kappa$. Furthermore, the induced map
$$\pi_\kappa:C_0(X)\to \mathcal{L}(\Gamma_\kappa)\quad\quad f\mapsto \pi_\kappa(f)$$
defines a $*$-algebra homomorphism.
\item The map
$$V_{\kappa}:C_0(Y)\to C_0(X\times Y)/N_\kappa\subset \Gamma_\kappa\quad\quad V_\phi(f)=f+N_\kappa$$
defines an adjointable $C_0(Y)$-linear map.
\end{enumerate}
Additionally, the two maps satisfy the following conditions:
\begin{enumerate}
    \item For all $f\in C_0(X)$, $\phi_\kappa(f)=V_\kappa^*\pi_\kappa(f)V_\kappa$.
    \item $\pi_\kappa(C_0(X))V_\kappa C_0(Y)$ is a dense submodule in $\Gamma_\kappa$.
\end{enumerate}

\end{lemma}

\begin{proof}\

\noindent We begin by proving $(1)$. Fix $f\in C_0(X)$, and define
$$\pi_0(f):C_0(X\times Y)\to C_0(X\times Y)\quad\quad \pi_0(f)g=fg$$
As $f\in C_0(X)$, it follows $fg\in C_0(X\times Y)$. Therefore the map is well-defined. It is clear $\pi_0(f)$ is $C_0(Y)$-linear. We claim $\pi_0(f)$ is bounded with respect to the semi-norm on $C_0(X\times Y)$. Indeed, given $g\in C_0(X\times Y)$,
$$||\pi_0(f)g||^2=\sup\limits_{y\in Y}\int_{X}|f(x)|^2| g(x, y)|^2\kappa(y, dx)\leq ||f||_{\sup}^2\sup\limits_{y\in Y}\int_{X}| g(x, y)|^2\kappa(y, dx)=||f||_{\sup}^2||g||^2.$$
Therefore $\pi_0(f)$ induces a unique bounded $C_0(Y)$-linear map on $C_0(X\times Y)/N_\kappa$ which uniquely extends to a bounded $C_0(Y)$-linear map on all of $\Gamma_\kappa$. Denote the extension by $\pi_\kappa(f)$.

Now we claim $\pi_\kappa(f)$ is adjointable with adjoint $\pi_\kappa(f^*)$. As both maps are continuous, it suffices to show $\pi_\kappa(f^*)$ is the adjoint of $\pi_\kappa(f)$ on the dense subspace $C_0(X\times Y)/N_\kappa$. Given $[g], [h]\in C_0(X\times Y)/N_\kappa$ we have
\begin{align*}
    \langle \pi_\kappa(f)[g], [h]\rangle&=\int_{X}\overline{f(x)g(x, y)}h(x, y)\;\kappa(y, dx)\\
    &=\int_{X}\overline{g(x, y)}\left(\overline{f(x)}h(x, y)\right)\;\kappa(y, dx)=\langle [g], \pi_\kappa(\overline{f})[h]\rangle
\end{align*}
Hence $\pi_\kappa(f)$ is adjointable with $\pi_\kappa(f)^*=\pi_\kappa(\overline{f})$.

Define the map
$$\pi_\kappa:C_0(X)\to \mathcal{L}(\Gamma_{\kappa})\quad\quad f\mapsto \pi_\kappa(f)$$
We claim $\pi_\kappa$ is a $*$-algebra homomorphism. The argument above shows $\pi_\kappa$ is well-defined as well as preserves the adjoint. For all $f_1, f_2\in C_0(X)$ and $\lambda\in\mathbb{C}$, it is clear we have 
$$\pi_0(f_1f_2)=\pi_0(f_1)\pi_0(f_2)\quad\quad \pi_0(f_1+\lambda f_2)=\pi_0(f_1)+\lambda \pi_0(f_2).$$
By the uniqueness of the induced map on $C_0(X\times Y)/N_\kappa$ as well as the uniqueness of the extension to $\Gamma_\kappa$, we have
$$\pi_\kappa(f_1f_2)=\pi_{\kappa}(f_1)\pi_\kappa(f_2)\quad\quad \pi_\kappa(f_1+\lambda f_2)=\pi_\kappa(f_1)+\lambda \pi_\kappa(f_2).$$ showing $\pi_\kappa$ is multiplicative and $\mathbb{C}$-linear. Thus $\pi_\kappa$ forms a $*$-algebra homomorphism.

Now we establish $(2)$. It is clear the map $V_\kappa$ is $C_0(Y)$-linear as well as bounded. Define
$$W:C_0(X\times Y)\to C_0(Y)\quad\quad W(g)(y)=\int_{X}g(x, y)\;\kappa(y, dx)$$
From the continuity argument in Lemma \ref{Lemma: building Gamma_0}, $W$ is a well-defined map. It is clear $W$ is $C(Y)$-linear. Given $g\in C_0(X\times Y)$, we have
\begin{align*}
    ||Wg||&=\sup\limits_{y\in Y}|(Wg)(y)|=\sup\limits_{y\in Y}\left|\int_{X}g(x,y)\;\kappa(y, dx)\right|\\
    &\leq \sup\limits_{y\in Y}\int_{X}|g(x, y)|\;\kappa(y, dx)\leq \sup\limits_{y\in Y}\left(\int_{X}|g(x, y)|^2\;\kappa(y, dx)\right)^\frac{1}{2}\left(\int_{X}\;\kappa(y, dx)\right)^{\frac{1}{2}}\\
    &=\sup\limits_{y\in Y}\left(\int_{X}|g(x, y)|^2\;\kappa(y, dx)\right)^\frac{1}{2}=\left(\sup\limits_{y\in Y}\int_{X}|g(x, y)|^2\;\kappa(y, dx)\right)^\frac{1}{2}=||g||.
\end{align*}
Thus $W$ is a bounded map. Therefore $W$ induces a unique bounded $C_0(Y)$-linear map on $C_0(X\times Y)/N_\kappa$ which uniquely extends to a bounded $C_0(Y)$-linear map on $\Gamma_\kappa$. Denote the  induced map on $\Gamma _\kappa$ as $\widetilde{W}$.

We claim $V_\kappa$ is adjointable with adjoint $\widetilde{W}$. As both maps are continuous and linear, it suffices to show this on the dense submodule $C_0(X\times Y)/N_\kappa$ and $C_0(Y)$. Let $g\in C_0(Y)$ and $[h]\in C_0(X\times Y)/N_\kappa$, then
\begin{align*}
    \langle V_\kappa g, [h]\rangle(y)=\int_{X}\overline{g(y)}h(x, y)\;\kappa(y, dx)=\overline{g(y)} \widetilde{W}(h)(y)=\langle g, \widetilde{W}h\rangle(y)
\end{align*}
Hence $V_\kappa$ is adjointable with adjoint $\widetilde{W}$.

Now we verify $\pi_\kappa$ and $V_\kappa$ satisfy the two conditions. We start with condition $(1)$. Fix $f\in C_0(X)$ and fix $g\in C_0(Y)$. Then for each $y\in Y$,
$$(\phi_\kappa(f)g)(y)=g(y)\int_{X}f(x)\;\kappa(y, dx)$$
and
$$(V_\kappa^*\pi_\kappa(f)V_\kappa g)(y)=\int_{X}(\pi_\kappa(f)V_\kappa g)(x, y)\;\kappa(y, dx)=\int_{X} f(x)g(y)\;\kappa(y, dx)$$
As $y\in Y$ was arbitrary, we have $\phi_\kappa(f)g=
V_\kappa^*\pi_\phi(f)V_\kappa g$; hence, $\phi_\kappa(f)=V_\kappa^*\pi_\phi(f)V_\kappa$. 

Now we verify condition $(2)$. First, we claim
$$A=\text{Span}(\{fg: g\in C_0(Y), f\in C_0(X)\})\subset C_0(X\times Y)$$
is dense with respect to the sup-norm. To prove this, we utilize the Stone-Weierstrass Theorem. As it is clear that $A$ is a $*$-subalgebra of $C_0(X\times Y)$, we only verify $A$ vanishes no-where and $A$ separates points. To see that $A$ vanishes no-where, let $(x, y)\in X\times Y$. As $X$ is a locally compact space, there exists an open subset $U\subset X$ with $x\in U$ and $\overline{U}$ compact in $X$. Using Urysohn's Lemma, there exists $f\in C_0(X)$ such that $f(\overline{U})=1$. Similarly, there exists an open subset $V\subset Y$ with $y\in V$ and $\overline{V}$ compact in $Y$. Applying Urysohn's Lemma again, there exists $g\in C_0(Y)$ such that $g(\overline{V})=1$. Therefore, $f(x)g(y)=1$ with $fg\in A$ showing $A$ vanishes no-where. To see $A$ separates points, let $(x_1, y_1), (x_2, y_2), X\times Y$ with $(x_1, y_1)\neq (x_2, y_2)$. Without loss of generality, we assume $x_1\neq x_2$. As $X$ is locally compact and Hausdorff, Urysohn's Lemma implies there exists $f\in C_0(X)$ such that $f(x_1)=1$ and $f(x_2)=0$. As $Y$ is locally compact and Hausdorff, we can apply Urysohn's Lemma to know there exists $g\in C_0(X)$ such that $g(y_1)=1$. Taking $fg\in A$, we have $f(x_1)g(y_1)=1$ while $f(x_2)g(y_2)=0$. Thus $A$ is able to separate points in $X\times Y$. Therefore, by the Stone-Weierstrass Theorem, $A$ is dense in $C_0(X\times Y)$ with respect to the sup-norm.

Observe, the sup-norm on $C_0(X\times Y)$ is stronger than the semi-norm from the $C_0(Y)$-valued semi-inner product. Indeed, given $h\in C_0(X\times Y)$, we have for all $y\in Y$
$$|\langle h, h\rangle(y)|=\int_{X}|h(x, y)|^2\;\kappa(y, dx)\leq ||h||_{\sup}^2$$
so that $||h||\leq ||h||_{\sup}$. Therefore, $A$ is dense in $C_0(X\times Y)$ with respect to the semi-norm. As the map $C_0(X\times Y)\to C_0(X\times Y)/N_\kappa$ is a continuous surjection with respect to the semi-norm, the subspace $A/N_\kappa$ is dense in $C_0(X\times Y)/N_\kappa$. As $C_0(X\times Y)/N_\kappa$ is dense in $\Gamma_\kappa$, we have $A/N_\kappa$ is dense in $\Gamma_\kappa$. We observe
$$A/N_\kappa=\text{Span}\left(\{[f*g]\in C_0(X\times Y)/N_\kappa: f\in C_0(X), g\in C_0(Y)\}\right)=\pi_\kappa(C_0(Y))V_\kappa C_0(Y)$$
Therefore we conclude $(2)$ holds.

\end{proof}

\begin{lemma}\label{Lemma: G-action on the module}\

\noindent Let $(C_0(X), G, \alpha)$ and $(C_0(Y), G, \beta)$ be $C^*$-dynamical systems, and let $\phi:C_0(X)\to C_b(Y)$ be a non-degenerate completely positive $G$-equivariant map. Let $\kappa_\phi$ be the $G$-invariant weakly continuous Markov kernel associated to $\phi$.
\begin{enumerate}
    \item For each $g\in G$, the map
$$C_0(X\times Y)\to C_0(X\times Y)\quad\quad f\mapsto \left((x, y)\mapsto f(\tilde{\alpha}_{g^{-1}}(x), \tilde{\beta}_{g^{-1}}(y))\right)$$
defines a bounded $\beta_g$-linear map which uniquely descends to a well-defined $\beta_g$-adjointable unitary $W_g$ on $\Gamma_\kappa$.
    \item The map 
$$W_{\kappa}:G\to \mathcal{B}(\Gamma_k)\quad\quad g\mapsto W_g$$ defines a group homomorphism. Furthermore, if $\alpha$ and $\beta$ are continuous, then $W_\kappa$ is strongly continuous.
\item Let $\pi_\kappa:C_0(X)\to \mathcal{L}(\Gamma_k)$ be the $*$-algebra homomorphism from Lemma \ref{Lemma: KSGNS construction in terms of kappa}. Then for all $g\in G$ and $f\in C_0(X)$, $W_g\circ \pi_\kappa(f)=\pi_\kappa(\alpha_g(f))\circ W_g$.
\end{enumerate}

\end{lemma}

\begin{proof}\

\noindent We begin by proving $(1)$. Fix $g\in G$. First, define
$$W_g:C_0(X\times Y)\to C_0(X\times Y)\quad\quad W_g(f)(x, y)=f(\tilde{\alpha}_{g^{-1}}(x), \tilde{\beta}_{g^{-1}}(y))$$
As $\tilde{\alpha}_{g^{-1}}:X\to X$ and $\tilde{\beta}_{g^{-1}}:Y\to Y$ are homeomorphisms, it follows $W_g(f)\in C_0(X\times Y)$. It is clear $W_g$ is $\mathbb{C}$-linear. Given $f\in C_0(X\times Y)$ and $h\in C_0(Y)$, we have
\begin{align*}
    W_g(fh)(x, y)&=f(\tilde{\alpha}_{g^{-1}}(f), \tilde{\beta}_{g^{-1}}(y))*h(\tilde{\beta}_{g^{-1}}(y))\\
    &=f(\tilde{\alpha}_{g^{-1}}(f), \tilde{\beta}_{g^{-1}}(y))*\beta_g(h)(y)=(W_g(f)\beta_g(h))(x, y)
\end{align*}
Thus $W_g(fh)=W_g(f)\beta_g(h)$ showing $W_g$ is $\beta_g$-linear. We claim $W_g$ is bounded. Given $f\in C_0(X\times Y)$, $\tilde{\beta}_g$ being a bijection and $\kappa_\phi$ being $G$-invariant implies
\begin{align*}
    ||W_g(f)||^2&=\sup\limits_{y\in Y}\int_{X}|f(\tilde{\alpha}_{g^{-1}}(x), \tilde{\beta}_{g^{-1}}(y))|^2\;\kappa_\phi(y, dx)\\
    &=\sup\limits_{y'\in Y}\int_{X}|f(\tilde{\alpha}_{g^{-1}}(x), y')|^2\;\kappa(\tilde{\beta}_{g}(y'), dx)\\
    &=\sup\limits_{y'\in Y}\int_{X}|f(x, y')|^2\; d(\kappa(\tilde{\beta}_g(y'), -)\circ\tilde{\alpha}_g)(x)\\
    &=\sup\limits_{y'\in Y}\int_{X}|f(x, y')|^2\;\kappa(y', dx)=||f||^2
\end{align*}
Thus $W_g$ is an isometry. Therefore $W_g$ defines a unique $\beta_g$-linear isometry on $C_0(X\times Y)/N_\kappa$ which extends uniquely to a $\beta_g$-linear isometry on $\Gamma_\kappa$.  Denote extension to $\Gamma_\kappa$ also as $W_g$.

We claim $W_g$ is $\beta_g$-adjointable with adjoint given by $W_{g^{-1}}$. As $W_g$ and $W_{g^{-1}}$ are continuous and $\mathbb{C}$-linear, it suffices to shows this on the dense submodule $C_0(X\times Y)/N_\kappa$. Let us fix $[f], [h]\in C_0(X\times Y)/N_\kappa$, then for each $y\in Y$
\begin{align*}
    \langle W_g [f], [h]\rangle(y)&=\int_{X}\overline{f(\tilde{\alpha}_{g^{-1}}(x), \tilde{\beta}_{g^{-1}}(y))} h(x, y)\;\kappa(y, dx)\\
    &=\int_{X}\overline{f(x, \tilde{\beta}_{g^{-1}}(y)}h(\tilde{\alpha}_{g}(x), y)\;d(\kappa(y, -)\circ\tilde{\alpha}_g)(x)\\
    &=\int_{X}\overline{f(x, \tilde{\beta}_{g^{-1}}(y))}h(\tilde{\alpha}_{g}(x), y)\;\kappa(\tilde{\beta}_{g^{-1}}(y), x)\\
    &=\langle [f], W_{g^{-1}}[h]\rangle(\tilde{\beta}_{g^{-1}}(y))=\beta_g(\langle [f], W_{g^{-1}}[h]\rangle)(y)
\end{align*}
Therefore $\langle W_g [f], [h]\rangle=\beta_{g}(\langle [f], W_{g^{-1}}[h]\rangle)$ showing $W_g$ is $\beta_g$-adjointable with adjoint $W_{g^{-1}}$. As $W_gW_{g^{-1}}$ and $W_{g^{-1}}W_g$ are the identity on $C_0(X\times Y)$, uniqueness of the extensions to $C_0(X\times Y)/N_\kappa$ and $\Gamma_\kappa$ implies $W_gW_{g^{-1}}$ and $W_{g^{-1}}W_g$ are the identity map on $\Gamma_{\kappa}$. Therefore $W_g$ is a $\beta_g$-adjointable unitary.

Now we prove $(2)$. Give $g_1, g_2\in G$ and $f\in C_0(X\times Y)$, we have
\begin{align*}
    (W_{g_1}W_{g_2}f)(x, y)&=(W_{g_2}f)(\tilde{\alpha}_{g_1^{-1}}(x), \tilde{\beta}_{g_1^{-1}}(y))=f(\tilde{\alpha}_{g_2^{-1}}(\tilde{\alpha}_{g_1^{-1}}(x)), \tilde{\beta}_{g_2^{-1}}(\tilde{\beta}_{g_1^{-1}}(y)))\\
    &=f(\tilde{\alpha}_{(g_1g_2)^{-1}}(x), \tilde{\beta}_{(g_1g_2)^{-1}}(y))=W_{g_1g_2}(f)(x, y)
\end{align*}
Thus, by the uniqueness of the map on $C_0(X\times Y)/N_\kappa$ and extension to $\Gamma_\kappa$,  $W_{g_1}W_{g_2}=W_{g_1g_2}$ showing $W$ is a group homomorphism.

Now suppose $\alpha$ and $\beta$ are continuous. Define the group homomorphism
$$\tilde{\gamma}:G\to \text{Homeo}(X\times Y)\quad\quad \tilde{\gamma}_g(x, y)=(\tilde{\alpha}_{g^{-1}}(x), \tilde{\beta}_{g^{-1}}(y))$$
As the maps $\alpha$ and $\beta$ are continuous, the maps
$$G\xrightarrow{g\mapsto \tilde{\alpha}_{g^{-1}}} \text{Homeo}(X)\quad\quad G\xrightarrow{g\mapsto \tilde{\beta}_{g^{-1}}} \text{Homeo}(Y)$$
are continuous. Therefore $\tilde{\gamma}$ is continuous. Let $\gamma:G\to \text{Aut}_{*\text{-alg}}(C_0(X\times Y))$ be the group homomorphism associated to $\tilde{\gamma}$. As $\tilde{\gamma}$ is continuous, $\gamma$ is continuous with respect to the point-norm topology. Let $(g_\lambda)_{\lambda\in\Lambda}$ be a net in $G$ which converges to $g\in G$. We claim $(W_{g_\lambda})_{\lambda\in\Lambda}$ converges to $W_g$ in the strong operator topology on $\mathcal{B}(\Gamma_{\kappa})$. As $(W_{g_\lambda})_{\lambda\in\Lambda}$ is uniformly norm bounded, it suffices to prove strong convergence on  $C_0(X\times Y)/N_\kappa$. Fix $[f]\in C_0(X\times Y)$. For each $\lambda\in\Lambda$, we have
\begin{align*}
    \left\| W_{g_\lambda} [f]-W_{g}[f]\right\|^2&=\sup\limits_{y\in Y}\int_{X}|f(\tilde{\alpha}_{g_\lambda^{-1}}(x), \tilde{\beta}_{g_\lambda^{-1}}(y))-f(\tilde{\alpha}_{g^{-1}}(x), \tilde{\beta}_{g^{-1}}(y))|^2\;\kappa(y, dx)\\
    &\leq \sup\limits_{(x, y)\in X\times Y}|f(\tilde{\alpha}_{g_\lambda^{-1}}(x), \tilde{\beta}_{g_\lambda^{-1}}(y))-f(\tilde{\alpha}_{g^{-1}}(x), \tilde{\beta}_{g^{-1}}(y))|^2\\
    &=||\gamma_{g_\lambda}(f)-\gamma_g(f)||^2_{\sup}
\end{align*}
As $\gamma$ is continuous with respect to the point-norm topology, for any $\epsilon>0$, there exists $\lambda_0\in \Lambda$ such that for all $\lambda\in\Lambda$ with $\lambda\geq \lambda_0$, 
$$\left\| W_{g_\lambda} [f]-W_{g}[f]\right\|^2\leq ||\gamma_{g_\lambda}(f)-\gamma_g(f)||_{\sup}^2<\epsilon^2.$$
Thus $(W_{g_\lambda}[f])_{\lambda\in\Lambda}$ converges to $W_g[f]$ showing $W_\kappa$ is continuous with respect to the strong operator topology.

Finally, we prove $(3)$. Fix $f\in C_0(X)$ and $g\in G$. As $W_g\circ \pi_\kappa(f)$ and $\pi_\kappa(\alpha_g(f))\circ W$ are continuous and $\beta_g$-linear, it suffices to show the two maps agree on the dense submodule $C_0(X\times Y)/N_\kappa$. As $W_g\circ \pi_\kappa(f)$ and $\pi_\kappa(\alpha_g(f))\circ W_g$ are defined on $C_0(X\times Y)$, it suffices to show the two maps agree on $C_0(X\times Y)$.  Given $h\in C_0(X\times Y)$, we have
\begin{align*}
    (W_g\circ \pi_\kappa(f))(h)(x, y)&=(\pi_\kappa(f)h)(\tilde{\alpha}_{g^{-1}}(x), \tilde{\beta}_{g^{-1}}(y))=f(\tilde{\alpha}_{g^{-1}}(x))h(\tilde{\alpha}_{g^{-1}}(x), \tilde{\beta}_{g^{-1}}(y))\\
    &=f(\tilde{\alpha}_{g^{-1}}(x))*(W_gh)(x, y)=\alpha_g(f)(x)*(W_gh)(x, y)\\
    &=(\pi_{\kappa}(\alpha_g(f))(W_gh))(x, y)
\end{align*}
Thus $W_g\circ \pi_\kappa(f)=\pi_{\kappa}(\alpha_g(f))\circ W_g$ for all $g\in G$ and $f\in C_0(X)$.

\end{proof}

\begin{theorem}\label{Theorem: KSGNS data for MK}\

\noindent Let $(C_0(X), G, \alpha)$ and $(C_0(Y), G, \beta)$ be $C^*$-dynamical systems. If $\phi:C_0(X)\to C_b(Y)$ is a non-degenerate completely positive $G$-equivariant map with associated $G$-invariant weakly continuous Markov kernel $\kappa_\phi$, then the quadruple $(\Gamma_{\kappa_\phi}, \pi_{\kappa_\phi}, V_{\kappa_\phi}, W_{\kappa_\phi})$ satisfies the conditions of Theorem \ref{Thrm: Equiv Dil} for $((C_0(Y), \phi), \beta)$.

\end{theorem}

Theorem \ref{Theorem: KSGNS data for MK} suggest that we view $C_0(Y)$ not as a $C^*$-algebra but as a Hilbert module over itself; however, this perspective makes the starting data for the two systems nonequivalent: $C_0(X)$ is a viewed as a $C^*$-algebra while $C_0(Y)$ is viewed as a Hilbert module. To restore equivalence in the starting data, we view both $C_0(X)$ and $C_0(Y)$ as Hilbert modules over themselves. With this new perspective, we recognize the data needed for forming the relational joint system as a Hilbert module dynamical system $((C_0(X), G, \alpha), C_0(X), \alpha)$ together with a non-degenerate positive equivariant $C^*$-correspondence $((C_0(Y), \phi), \beta)$ which enables us to construct a Hilbert module dynamical system over $(C_0(Y), G, \beta)$. Combining this data together leads us to make the following definition.

\begin{definition}\

\noindent Let $C_0(X)$ be a commutative $C^*$-algebra, and let $G$ be a group. Define a classical $G$-equivariant coupling for $C_0(X)$ as the data
$$(((C_0(X), G, \alpha), (C_0(Y), G, \beta), ((C_0(Y), \phi), \beta)), ((C_0(X), G, \alpha), C_0(X), \alpha))$$
where
\begin{itemize}
    \item $(C_0(X), G, \alpha)$ is a $C^*$-dynamical system.
    \item $(C_0(Y), G, \beta)$ is a $C^*$-dynamical system.
    \item $((C_0(Y), \phi), \beta)$ is a non-degenerate positive equivariant $C^*$-correspondence from \\$(C_0(X), G, \alpha)$ to $(C_0(Y), G, \beta)$.
    \item $((C_0(X), G, \alpha), C_0(X), \alpha)$ is a Hilbert module dynamical system.
\end{itemize}

\end{definition}

\begin{definition}\

\noindent Let $C_0(X)$ be a commutative $C^*$-algebra, and let $G$ be a group, and let
$$(((C_0(X), G, \alpha), (C_0(Y), G, \beta), ((C_0(Y), \phi), \beta)), ((C_0(X), G, \alpha), C_0(X), \alpha))$$ be a classical $G$-equivariant coupling for $C_0(X)$. Let $(F_\phi, \pi_\phi, V_\phi, W)$ be the unitarily unique quadruple obtained from applying Theorem \ref{Thrm: Equiv Dil} to $((C_0(Y), \phi), \beta)$. 
\begin{itemize}
    \item  Define the relational joint space of the classical $G$-equivariant coupling as $F_\phi$.
    \item Define the relational joint system of the classical $G$-equivariant coupling as $\mathcal{L}(F_\phi)$.
    \item Define the reference frame system of the classical $G$-equivariant coupling as $\mathcal{L}(F_\phi)^G$.
\end{itemize}

\end{definition}

\subsection{Relativization maps}

Now that we have built the relational joint system and described the relational invariant observables, we turn our attention to describing relativization maps. To help setup our conversation, we fix a commutative $C^*$-algebra $C_0(X)$, and we fix a classical $G$-equivariant coupling for $C_0(X)$ given by
$$(((C_0(X), G, \alpha), (C_0(Y), G, \beta), ((C_0(Y), \phi), \beta)), ((C_0(X), G, \alpha), C_0(X), \alpha)).$$

As elements of $\mathcal{L}(\Gamma_{\kappa_\phi})$ represent relational observables, we interpret the non-degenerate $*$-algebra homomorphism $\pi_{\kappa_\phi}:C_0(X)\to \mathcal{L}(\Gamma_{\kappa_\phi})$ as relativizing observables in $C_0(X)$. As $\pi_{\kappa_\phi}$ is a non-degenerate $*$-algebra map, we can extend this interpretation to the unique extension $$\tilde{\pi}_{\kappa_\phi}:\mathcal{L}(C_0(X))=C_b(X)\to \mathcal{L}(\Gamma_{\kappa_\phi}).$$

From the introduction, the meaningful relational observables are those which are invariant, i.e. the observables in the reference frame system $\mathcal{L}(\Gamma_{\kappa_\phi})^G$; thus, for the relativization $\tilde{\pi}_{\kappa_\phi}$ to be physically meaningful, we need to map into $\mathcal{L}(\Gamma_{\kappa_\phi})^G$.  As the map $\tilde{\pi}_{\kappa_\phi}$ is $G$-equivariant, we obtain a unital $*$-algebra homomorphism
$$\mathcal{L}(C_0(X))^G\to \mathcal{L}(\Gamma_{\kappa_\phi})^{G} \quad\quad f\mapsto \tilde{\pi}_{\kappa_\phi}(f)$$
Therefore we just need to fix the relativization for observables in $\mathcal{L}(C_0(X))\backslash \mathcal{L}(C_0(X))^G$; that is, we want our relativization map to be a function $\pounds:\mathcal{L}(C_0(X))\to \mathcal{L}(\Gamma_{\kappa_\phi})^G$  making the following diagram commute: 
\begin{center}
    \begin{tikzcd}
	{\mathcal{L}(C_0(X))} &&& \\
	\\
	{\mathcal{L}(C_0(X))^G} &&& {\mathcal{L}(\Gamma_{\kappa_\phi})^G}
	\arrow["\pounds"{description}, dashed, from=1-1, to=3-4]
	\arrow[hook', from=3-1, to=1-1]
	\arrow["{T\mapsto \tilde{\pi}_\phi(T)}"{description}, from=3-1, to=3-4]
\end{tikzcd}
\end{center}
In addition to making the diagram commute, we require the map $\pounds$ to be completely positive as well as strictly continuous on the unit ball. Both of these conditions come from generalizing that $\tilde{\pi}_{\kappa_\phi}$ is a $*$-algebra homomorphism which is strictly continuous on the unit ball in $\mathcal{L}(C_0(X))$. Note, we do not require $\pounds$ to be unital as that is guaranteed by the commutativity of the diagram. 

\begin{definition}\

\noindent Let $C_0(X)$ be a commutative $C^*$-algebra, let $G$ be a group, and let
$$(((C_0(X), G, \alpha), (C_0(Y), G, \beta), ((C_0(Y), \phi), \beta)), ((C_0(X), G, \alpha), C_0(X), \alpha))$$ be a classical $G$-equivariant coupling for $C_0(X)$. Let $(F_\phi, \pi_\phi, V_\phi, W)$ be the unitarily unique quadruple obtained from applying Theorem \ref{Thrm: Equiv Dil} to $((C_0(Y), \phi), \beta)$. Define a relativization map for the classical $G$-equivariant coupling  as the data of a function $$\pounds:\mathcal{L}(C_0(X))\to \mathcal{L}( F_\phi)^{G}$$
which satisfies the following conditions:
\begin{enumerate}
    \item $\pounds$ is completely positive.
    \item $\pounds$ is strictly continuous on the unit ball.
    \item $\pounds$ makes the following diagram commute:
\begin{center}
    \begin{tikzcd}
	{\mathcal{L}(C_0(X))} &&& \\
	\\
	{\mathcal{L}(C_0(X))^G} &&& {\mathcal{L}(F_\phi)^G}
	\arrow["\pounds"{description}, dashed, from=1-1, to=3-4]
	\arrow[hook', from=3-1, to=1-1]
	\arrow["{T\mapsto \tilde{\pi}_\phi(T)}"{description}, from=3-1, to=3-4]
\end{tikzcd}
\end{center}

\end{enumerate}

\end{definition}

Before moving on, let us make a few comments. First, if $\pounds$ is a relativization map for the classical equivariant coupling, then, as $I\in \mathcal{L}(C_0(X))^{G}$ and the bottom map in the diagram is a unital $*$-algebra homomorphism, it follows $\pounds$ is a unital completely positive map. Therefore, the pullback $\pounds^*$ yields a well-defined map on states. Second, this formulation of the relativization map frames the existence and possible uniqueness of $\pounds$ as an extension problem. Third, as $\pounds$ is a unital completely positive map which is strictly continuous on the unit ball in $\mathcal{L}(C_0(X))$, the restriction $\pounds|_{\mathcal{K}(C_0(X))}$ defines a non-degenerate completely positive map. Finally, as all of the examples in the next subsection show, it is not true that a relativization map is unique. We do not have a commutative example in mind showing a relativization map can fail to exist; however, when we go to the non-commutative setting, relativization maps can fail to exist (see Example \ref{Ex: non-existence example} for details). 

We conclude the subsection with a result for building relativization maps for classical $G$-equivariant couplings when the actions are both continuous and transitive. 

\begin{lemma}\

\noindent Let $C_0(X)$ be a commutative $C^*$-algebra, and let
$$(((C_0(X), G, \alpha), (C_0(Y), G, \beta), ((C_0(Y), \phi), \beta)), ((C_0(X), G, \alpha), C_0(X), \alpha))$$
be a classical $G$-equivariant coupling for $C_0(X)$. Let $\kappa_\phi$ be the $G$-invariant weakly continuous Markov associated to $\phi$. 
\begin{enumerate}
    \item For each $f\in C_b(X\times Y)$, the map
$$\mathcal{M}_0(f):C_0(X\times Y)\to C_0(X\times Y)\quad\quad (\mathcal{M}_0(f)g)(x, y)=f(x, y)g(x, y)$$ descends to a unique adjointable $C_0(Y)$-linear map $\mathcal{M}_f:\Gamma_{\kappa}\to \Gamma_{\kappa}.$
\item The map
$$\mathcal{M}:C_b(X\times Y)\to \mathcal{L}(\Gamma_{\kappa})\quad\quad f\mapsto \mathcal{M}_f$$ defines a unital $*$-algebra homomorphism which is strictly continuous on the unit ball in $C_b(X\times Y)$.
\item Let $\tilde{\gamma}$ denote the diagonal action of $G$ on $X\times Y$, and let $\gamma$ denote the corresponding action on $C_b(X\times Y)$. For each $g\in G$ and $f\in C_b(X\times Y)$, $\mathcal{M}_{\gamma_g(f)}=\text{Ad}(W_g)\mathcal{M}_f$.
\end{enumerate}

\end{lemma}

\begin{proof}\

\noindent We start by proving $(1)$. It is clear $\mathcal{M}_0(f)$ is $C_0(Y)$-linear. We claim $\mathcal{M}_0(f)$ is bounded with respect to the semi-norm on $C_0(X\times Y)$. Let $g\in C_0(X\times Y)$, then
\begin{align*}
    \left\|\mathcal{M}_0(f)g\right\|^2&=\sup\limits_{y\in Y} \int_{X}|f(x, y)|^2|g(x, y)|^2\;\kappa_\phi(y, dx)\leq  ||f||_{\sup}^2||g||^2
\end{align*}
Thus $\mathcal{M}_0(f)$ is bounded. Let $\mathcal{M}_f$ denote the unique $C_0(Y)$-linear bounded map induced by $\mathcal{M}_0(f)$ on $\Gamma_{\kappa_\phi}$.  

We now claim $\mathcal{M}_f$ is adjointable with adjoint $\mathcal{M}_{f^*}$. As both maps are linear and continuous, it suffices to prove this claim on the dense submodule $C_0(X\times Y)/N_{\kappa_\phi}$. Given $[h]$ and $[g]$ in $C_0(X\times Y)/N_{\kappa_\phi}$, we have
\begin{align*}
    \langle \mathcal{M}_f[g], [h]\rangle(y)&=\int_{X}\overline{f(x, y)g(x, y)}h(x, y)\;\kappa_\phi(y, dx)=\langle [g], \mathcal{M}_{f^*}[h]\rangle(y).
\end{align*}
Thus $\mathcal{M}_f$ is adjointable with adjoint $\mathcal{M}_{f^*}$. Therefore $(1)$ holds.

Now we prove $(2)$. Given $f_1, f_2\in C_b(X\times Y)$ and $\lambda\in\mathbb{C}$, it is clear that
$$\mathcal{M}_0(f_1f_2)=\mathcal{M}_0(f_1)\mathcal{M}_0(f_2)\quad\quad \mathcal{M}_0(f_1+\lambda f_2)=\mathcal{M}_0(f_1)+\lambda \mathcal{M}_0(f_2).$$
Thus, by the uniqueness of the induced maps on $\Gamma_{\kappa_\phi}$, we have
$$\mathcal{M}_{f_1f_2}=\mathcal{M}_{f_1}\mathcal{M}_{f_2}\quad\quad \mathcal{M}_{f_1+\lambda f_2}=\mathcal{M}_{f_1}+\lambda \mathcal{M}_{f_2}$$
showing $\mathcal{M}$ is $\mathbb{C}$-linear and multiplicative. Furthermore, from $(1)$ we have $\mathcal{M}$ preserves the $*$-operation. Thus $\mathcal{M}$ is a $*$-algebra homomorphism. As $\mathcal{M}_0(1)$ is the identity map on $C_0(X\times Y)$, we have $\mathcal{M}_1=I$; thus, $\mathcal{M}$ is unital. 

Now we prove $\mathcal{M}$ is strictly continuous on the unit ball. Let $(g_n)_{n\in\mathbb{N}}$ be sequence in the unit ball of $C_b(X\times Y)$ which converges strictly to $g\in C_b(X\times Y)$. For each $[h]\in C_0(X\times Y)/N_{\kappa_\phi}$, we have
$$||\mathcal{M}_{g_n}[h]-\mathcal{M}_{g}[h]||\leq ||g_nh-gh||_{\sup}^{\frac{1}{2}}.$$
As $g_nh\to gh$ with respect to $||\cdot||_{\sup}$, it follows $\mathcal{M}_{g_n}[h]\to \mathcal{M}_{g}[h]$ with respect to the norm on $\Gamma_{\kappa_\phi}$. As $C_0(X\times Y)/N_{\kappa_\phi}$ is dense in $\Gamma_\kappa$ and $(\mathcal{M}_{g_n})_{n\in\mathbb{N}}$ is uniformly norm bounded, the standard approximation argument implies $\mathcal{M}_{g_n}x\to \mathcal{M}_{g}x$ for all $x\in \Gamma_{\kappa_\phi}$. Thus $\mathcal{M}$ is strictly continuous on the unit ball. 

Now we prove $(3)$. Fix $f\in C_b(X\times Y)$ and $s\in G$. As $\mathcal{M}_{\gamma_sf}$ and $\text{Ad}(W_s)\mathcal{M}_f$ are continuous, it suffices to show equality on the dense submodule $C_0(X\times Y)/N_{\kappa_\phi}$. As both maps are defined on $C_0(X\times Y)$, it suffices to show the maps agree on $C_0(X\times Y)$. Let $g\in C_0(X\times Y)$, then
\begin{align*}
    ((W_s\mathcal{M}_{f} W_{-s})g)(x, y)&=(\mathcal{M}_fW_{-s}g)(\tilde{\alpha}_{s^{-1}}x, \tilde{\beta}_{s^{-1}}y)\\
    &=f(\tilde{\alpha}_{s^{-1}}x, \tilde{\beta}_{s^{-1}}y)(W_{-s}g)(\tilde{\alpha}_{s^{-1}}x, \tilde{\beta}_{s^{-1}}y)\\
    &=f(\tilde{\alpha}_{s^{-1}}x, \tilde{\beta}_{s^{-1}}y)g(x, y)=(\gamma_sf)(x, y)g(x, y)\\
    &=(\mathcal{M}_{\gamma_s(f)}g)(x, y)
\end{align*}
Thus $\text{Ad}(W_s)\mathcal{M}_f=\mathcal{M}_{\gamma_s(f)}$ as claimed.

\end{proof}

\begin{proposition}\

\noindent Let $C_0(X)$ be a commutative $C^*$-algebra, and let
$$(((C_0(X), G, \alpha), (C_0(Y), G, \beta), ((C_0(Y), \phi), \beta)), ((C_0(X), G, \alpha), C_0(X), \alpha))$$
be a classical $G$-equivariant coupling. Let $\kappa_\phi$ be the $G$-invariant weakly continuous Markov associated to $\phi$. If $\tilde{\alpha}$ defines a transitive action of $G$ on $X$, then whenever $\iota: X\times Y\to X$ is a continuous map with 
$$\iota(\tilde{\alpha}_s(x), \tilde{\beta}_s(y))=\iota(x, y)$$
for all $x\in X$, $y\in Y$, $s\in G$, the unique extension of the map 
$$\mathcal{M}\circ\iota^*:C_0(X)\to \mathcal{L}(\Gamma_{\kappa_\phi})$$
to $C_b(X)$ defines a relativization map for the classical $G$-equivariant coupling.

\end{proposition}

\begin{proof}\

\noindent First, note the action $\tilde{\alpha}$ being transitive on $X$ implies $C_b(X)^G\cong \mathbb{C}$. Therefore, a relativization map for the equivariant coupling 
$$(((C_0(X), G, \alpha), (C_0(Y), G, \beta), ((C_0(Y), \phi), \beta)), ((C_0(X), G, \alpha), C_0(X), \alpha))$$
is equivalent to the data of a unital completely positive map $C_b(X)\to \mathcal{L}(\Gamma_{\kappa_\phi})^G$ which is strictly continuous on the unit ball. 

Suppose we have a continuous map  $\iota: X\times Y\to X$ with $\iota(\tilde{\alpha}_s(x), \tilde{\beta}_s(y))=\iota(x, y)$ for all $x\in X$, $y\in Y$, $s\in G$. As $\iota$ is a continuous map between locally compact Hausdorff spaces, the map
$$\iota^*:C_0(X)\to C_b(X\times Y)\quad \quad f\mapsto f\circ\iota$$
defines a non-degenerate $*$-algebra homomorphism. Therefore the unique extension 
$$\tilde{\iota}^*:C_b(X)\to C_b(X\times Y)\quad\quad \tilde{\iota}^*(f)=f\circ\iota$$
defines a unital $*$-algebra homomorphism which is strictly continuous on the unit ball. Therefore
$$\mathcal{M}\circ\tilde{\iota}^*:C_0(X)\to \mathcal{L}(\Gamma_{\kappa_\phi})$$
defines a unital $*$-algebra homomorphism which is strictly continuous on the unit ball. 

To conclude $\mathcal{M}\circ\tilde{\iota}^*$ is a relativization map, we just need to show the map is valued in $\mathcal{L}(\Gamma_{\kappa_\phi})^G$. Let $\tilde{\gamma}$ denote the diagonal action of $G$ on $X\times Y$, and let $\gamma$ denote the corresponding action on $C_b(X\times Y)$. For all $f\in C_0(X)$ and $s\in G$ we have
$$\gamma_s(\tilde{\iota}^*f)(x, y)=(f\circ\iota)(\tilde{\alpha}_{s^{-1}}x, \tilde{\beta}_{s^{-1}}y)=(f\circ\iota)(x, y)=(\tilde{\iota}^*f)(x, y)$$
Therefore $\tilde{\iota}^*(C_b(X))\subset C_b(X\times Y)^G$. As $\mathcal{M}$ is $G$-equivariant, it follows $\mathcal{M}\circ\tilde{\iota}^*$ is valued in $\mathcal{L}(\Gamma_{\kappa_\phi})^G$. Thus the map defines a relativization map for the classical $G$-equivariant coupling.

\end{proof}

\begin{lemma}\

\noindent Let $X$ and $Y$ be locally compact Hausdorff spaces, and let $G$ be a topological group acting on $X$ by $\tilde{\alpha}$. If $X/G$ is equipped with the quotient topology, then, as sets,
$$\{f\in C(X, Y): f\circ\tilde{\alpha}_s=f\text{ for all }s\in S\}\cong C(X/G, Y).$$

\end{lemma}

\begin{proof}\

\noindent The bijection follows from the universal property of the quotient topology.
    
\end{proof}

\begin{lemma}\

\noindent Let $X$ and $Y$ be locally compact Hausdorff spaces, and let $G$ be a topological group acting continuously on $X$ and $Y$ by $\tilde{\alpha}$ and $\tilde{\beta}$, respectively. Let $\tilde{\gamma}$ denote the diagonal action of $G$ on $X\times Y$. If $\tilde{\alpha}$ and $\tilde{\beta}$ are transitive actions, then $(X\times Y)/ G\cong K\backslash G/H$ where $K=\text{Stab}_G(x_0)$ and $H=\text{Stab}_G(y_0)$ for any $x_0\in X$ and $y_0\in Y$. 

\end{lemma}

\begin{proof}\

\noindent First, as $\tilde{\alpha}$ and $\tilde{\beta}$ are transitive actions, $\text{Stab}_G(x_0)\cong\text{Stab}_G(x_0')$ and  $\text{Stab}_G(y_0)\cong\text{Stab}_G(y_0')$ for all $x_0, x_0'\in X$ and $y_0, y_0'\in Y$. Therefore, fixing $x_0\in X$ and $y_0\in Y$, let $K=\text{Stab}_G(x_0)$ and $H=\text{Stab}_G(y_0)$. As $\tilde{\alpha}$ and $\tilde{\beta}$ are transitive actions, we have $X\cong G/K$ and $Y\cong G/H$. Therefore we have $X\times Y\cong G/K\times G/H$ so that $(X\times Y)/G\cong (G/K\times G/H)/G$. Let
$$\psi:(G/K\times G/H)/G\to K\backslash G/H\quad\quad \psi([g_1K, g_2H])=Kg_1^{-1}g_2H$$
We claim this is a well-defined homeomorphism. 

To $\psi$ is well-defined, suppose $(g_1K, g_2H)$ and $(g_3K, g_4H)$ are in the same orbit. Let $g\in G$ such that $(gg_1K, gg_2H)=(g_3K, g_4H)$, then $gg_1K=g_3K$ and $gg_2H=g_4H$. Therefore
$$Kg_1^{-1}g_2H=(Kg_1^{-1}g^{-1})(gg_2H)=(Kg_3^{-1})(g_4H)=Kg_3^{-1}g_4H$$
showing $\psi([g_1K, g_2H])=\psi([g_3K, g_4H])$. Thus $\psi$ is a well-defined function.

Now we prove $\psi$ is a bijection. Since $\psi([K, gH])=KgH$, the map $\psi$ is surjective. For injectivity, suppose $\psi([g_1K, g_2H])=\psi([g_3K, g_4H])$, then there exists $k\in K$ and $h\in H$ such that $g_1^{-1}g_2=kg_3^{-1}g_4h$.  Thus we have
$$(g_1K, g_2H)\sim (K, g_1^{-1}g_2H)\sim (K, kg_3^{-1}g_4hH)=(K, kg_3^{-1}g_4H)\sim (g_3k^{-1}K, g_4H)=(g_3K, g_4H)$$
showing $[g_1K, g_2H]=[g_3K, g_4H]$. Hence $\psi$ is injective.

Now we verify $\psi$ is continuous. Note, as $G$ is a topological group, we have the continuous map
$$f:G/K\times G/H\to K\backslash G/H\quad\quad f(g_1K, g_2H)=Kg_1^{-1}g_2H$$
Since $f$ is constant on the orbits of the diagonal action, then there exists a unique continuous map $\tilde{f}:(G/K\times G/H)/G\to K\backslash G/H$ such that $f=\tilde{f}\circ q$ where $q$ is the canonical quotient map from $G/K\times G/H$ to $(G/K\times G/H)/G$. Since $\tilde{f}=\psi$, then we know $\psi$ is continuous.

Now we prove $\psi^{-1}$ is continuous. Note, 
$$\psi^{-1}:K\backslash G/H\to (G/K\times G/H)/G\quad\quad \psi^{-1}(KgH)=[K, gH]$$
Since the map 
$$f:G\to (G/K\times G/H)/G\quad\quad f(g)=[K, gH]$$ 
is a continuous function with $f(kgh)=f(g)$ for all $k\in K$, $h\in H$ and $g\in G$, then, by the universal property of the quotient topology, there exists a unique continuous map $\tilde{f}:K\backslash G/H\to (G/K\times G/H)/G$ such that $f=\tilde{f}\circ q$ where $q$ is the canonical surjection from $G$ to $K\backslash G/H$. Since  $\tilde{f}=\psi^{-1}$, then we know $\psi^{-1}$ is continuous.

\end{proof}

\begin{theorem}\label{Theorem: rel map for transitive actions}\

\noindent Let $C_0(X)$ be a commutative $C^*$-algebra, and let
$$(((C_0(X), G, \alpha), (C_0(Y), G, \beta), ((C_0(Y), \phi), \beta)), ((C_0(X), G, \alpha), C_0(X), \alpha))$$
be a classical $G$-equivariant coupling for $C_0(X)$ with $G$ a topological group.  If $\tilde{\alpha}$ and $\tilde{\beta}$ define continuous transitive actions of $G$ on $X$ and $Y$, respectively, then, letting $K=\text{Stab}_G(x_0)$ and $H=\text{Stab}_G(y_0)$ for $x_0\in X$ and $y_0\in Y$,  every $\iota\in C(K\backslash G/ H, G/K)$ defines a relativization map for the classical $G$-equivariant coupling.

\end{theorem}

\subsection{Hilbert module classical reference frames and examples}

Combining together the data of a classical $G$-equivariant coupling and relativization map, we arrive at the definition of a Hilbert module classical reference frame.

\begin{definition}\

\noindent Let $C_0(X)$ be a commutative $C^*$-algebra. Define a Hilbert module classical reference frame for $C_0(X)$ as the data
$$((((C_0(X), G, \alpha), (C_0(Y), G, \beta), ((C_0(Y), \phi), \beta)), ((C_0(X), G, \alpha), C_0(X), \alpha)), \pounds)$$
where 
\begin{itemize}
    \item $(((C_0(X), G, \alpha), (C_0(Y), G, \beta), ((C_0(Y), \phi), \beta)), ((C_0(X), G, \alpha), C_0(X), \alpha))$ is a classical $G$-equivariant coupling for $C_0(X)$.
    \item $\pounds$ is a relativization for the classical $G$-equivariant coupling.
\end{itemize}
    
\end{definition}

\begin{example}\

\noindent A two state switch can be modeled by the $C^*$-algebra $C(X)$ where $X=\{x_1, x_2\}$. Assuming the switch is unmarked, we are unable to distinguish the absolute states $x_1$ and $x_2$ which is encoded by the unique transitive $\mathbb{Z}/2\mathbb{Z}$ action on $X$. Denote the action on $X$ by $\tilde{\alpha}$, and denote the corresponding action on $C(X)$ by $\alpha$. Without a reference system, the algebra of invariant observables is given by $C(X)^{\mathbb{Z}/2\mathbb{Z}}\cong\mathbb{C}$. Suppose we use another unmarked two state switch as a reference system. Let us model the system as $C(Y)$ with $Y=\{y_1, y_2\}$. As before, denote the unique transitive $\mathbb{Z}/2\mathbb{Z}$ action on $Y$ by $\tilde{\beta}$, and denote the corresponding action on $C(Y)$ as $\beta$. 

With respect to the actions on $X$ and $Y$, every $\mathbb{Z}/2\mathbb{Z}$-invariant Markov kernel from $Y$ to $X$ is of the form
$$\kappa_p:Y\times \text{Bor}(X)\to [0, 1]\quad\quad\kappa_p(y, B)=\begin{cases}
    (1-p)\delta_{x_1}(B)+p\delta_{x_2}(B) & y=y_1\\
    p\delta_{x_1}(B)+(1-p)\delta_{x_2}(B) & y=y_2
\end{cases}$$
for $p\in [0, 1]$. Physically, when $p=0$ or $p=1$ we have the deterministic situation of aligned and anti-aligned, respectively. When $p\in (0, 1)$, there is a probability of the states being aligned or anti-aligned given by $1-p$ and $p$, respectively. Using $\kappa_p$, we obtain a classical $\mathbb{Z}/2\mathbb{Z}$-equivariant coupling given by
$$(((C(X), \mathbb{Z}/2\mathbb{Z}, \alpha), (C(Y), \mathbb{Z}/2\mathbb{Z}, \beta), ((C(Y), \phi_{\kappa_p}), \beta)), ((C(X), \mathbb{Z}/2\mathbb{Z}, \alpha), C(X), \alpha)).$$

To build the relational joint space for the two switches, we begin with the pre-inner product $C(Y)$-module $C(X\times Y)$. Note, the $C(Y)$-valued pairing is given by
$$\langle f, g\rangle(y)=\begin{cases}
    (1-p)\overline{f(x_1, y_1)}g(x_1, y_1)+p\overline{f(x_2, y_1)}g(x_2, y_1) & y=y_1\\
    p\overline{f(x_1, y_2)}g(x_1, y_2)+(1-p)\overline{f(x_2, y_2)}g(x_2, y_2) & y=y_2
\end{cases}$$
Therefore, it follows
$$C(X\times Y)/N_{\kappa_p}\cong \begin{cases}
    C(X\times Y) & p\in (0, 1)\\
   C(\{(x_1, y_1), (x_2, y_2)\})  & p=0\\
  C(\{(x_1, y_2), (x_2, y_1)\})   & p=1
\end{cases}.$$
Note, in all three cases $C(X\times Y)/N_{\kappa_p}$ forms a Hilbert $C(Y)$-module since it a finite dimensional $C(Y)$-module.

Now let us compute $\mathcal{L}(C(X\times Y)/N_{\kappa_p})^{\mathbb{Z}/2\mathbb{Z}}$. In the case $p=0$ or $p=1$, we have 
$$C(X\times Y)/N_{\kappa_p}\cong C(Y)$$ so that
$$\mathcal{L}(C(X\times Y)/N_{\kappa_p})\cong M_1(C(Y))\cong M_1(\mathbb{C}\oplus \mathbb{C})\cong M_1(\mathbb{C})\oplus M_1(\mathbb{C})\cong \mathbb{C}\oplus \mathbb{C}$$
Under this identification, the $\mathbb{Z}/2\mathbb{Z}$ action on $\mathbb{C}\oplus \mathbb{C}$ is given by swapping the two factors. Therefore
$$\mathcal{L}(C(X\times Y)/N_{\kappa_p})^{\mathbb{Z}/2\mathbb{Z}}\cong (\mathbb{C}\oplus\mathbb{C})^{\mathbb{Z}/2\mathbb{Z}}\cong\mathbb{C}.$$
Thus a relativization map in the case $p=0$ or $p=1$ is equivalent to the choice of state on $C(X)$.

Now suppose $p\in (0, 1)$. As we have $C(X\times Y)\cong C(Y)\oplus C(Y)$, it follows
$$\mathcal{L}(C(X\times Y))\cong M_2(\mathbb{C}\oplus\mathbb{C})\cong M_2(\mathbb{C})\oplus M_2(\mathbb{C}).$$
Let $\tau\in M_2(\mathbb{C})$ given by
$$\tau=\begin{bmatrix}
    0 & 1\\
    1 & 0
\end{bmatrix},$$
then, under the identification of $\mathcal{L}(C(X\times Y))\cong M_2(\mathbb{C})\oplus M_2(\mathbb{C})$, the $\mathbb{Z}/2\mathbb{Z}$ action on $M_2(\mathbb{C})\oplus M_2(\mathbb{C})$ is generated by
$$[1](A, B)=(\tau B\tau, \tau A\tau).$$
Thus, the fix point algebra is given by
$$\mathcal{L}(C(X\times Y))^{\mathbb{Z}/2\mathbb{Z}}\cong\{(A, \tau A\tau)\in M_2(\mathbb{C})^{\oplus 2}: A\in M_2(\mathbb{C})\}\cong M_2(\mathbb{C}).$$
Therefore, when $p\in (0, 1)$, a relativization map for the classical $\mathbb{Z}/2\mathbb{Z}$ equivariant coupling is equivalent to a unital completely positive map $\pounds:C(X)\to M_2(\mathbb{C})$.

An example of a relativization map in the case $p\in (0, 1)$ is given by

$$\pounds:C(X)\to M_2(\mathbb{C})\quad\quad 
    \pounds(f)=\begin{bmatrix}
    \mathbb{E}[f(X) | Y=y_1] & 0\\
    0 & \mathbb{E}[f(X) | Y=y_2]
\end{bmatrix}$$
It is clear the map $\pounds$ is positive (thus completely positive) as well as unital. Therefore $\pounds$ defines relativization map. For $i=1, 2$, let
$$\chi_{i}:X\to \mathbb{C}\quad\quad \chi_i(x)=\begin{cases}
    1 & x=x_i\\
    0 & x\neq x_i
\end{cases},$$
then
$$\pounds(\chi_1)=\begin{bmatrix}
    1-p & 0\\
    0 & p
\end{bmatrix}\quad\quad \pounds(\chi_2)=\begin{bmatrix}
    p & 0\\
    0 & 1-p
\end{bmatrix}$$
The observables $\chi_i$ represent checking if the system $C(X)$ is in state $x_i$.  Thus $\pounds(\chi_i)$ represents the observable of checking if $X$ is in the state $x_i$ relative to the state of $Y$. Notice we obtain the probability from the conditioning as the eigenvalues. 

\end{example}

\begin{example}\

\noindent  We can model a three-state switch by the $C^*$-algebra $C(X)$ where $X=\{x_1, x_2, x_3\}$. Assuming the switch is unmarked, we are unable to distinguish the absolute states leading to a symmetry group of $D_6$:  
$$D_6=\langle R, S \;|\; R^3=I=S^2, SR=R^2S\rangle.$$
We encode this symmetry group via the action of $D_6$ on $X$ generated by
$$Rx_1=x_2\quad\quad Rx_2= x_3\quad\quad Rx_3=x_1\quad\quad Sx_1=x_3\quad\quad Sx_2= x_2\quad\quad S x_3=x_1$$
Denote action of $D_6$ on $X$ by $\tilde{\alpha}$, and denote the induced action on $C(X)$ by $\alpha$. Note, as the $D_6$ action is transitive, the algebra of invariant observables without a reference frame is given by $C(X)^{D_6}\cong\mathbb{C}$. Suppose we use a similar switch as a reference system. Denote the $C^*$-algebra for the reference switch as $C(Y)$ with $Y=\{y_1, y_2, y_3\}$, and denote the action from $D_6$ on $C(Y)$ and $Y$ by $\beta$ and $\tilde{\beta}$, respectively.

With respect to the actions on $X$ and $Y$, every $D_6$-invariant Markov kernel from $Y$ to $X$ is of the form
$$\kappa_p:Y\times\text{Bor}(X)\to [0, 1]\quad\quad \kappa_p(y, B)=\begin{cases}
    (1-p)\delta_{x_1}(B)+ \frac{p}{2}\delta_{x_2}(B)+\frac{p}{2}\delta_{x_3}(B) & y=y_1\\
   \frac{p}{2}\delta_{x_1}(B)+ (1-p)\delta_{x_2}(B)+\frac{p}{2}\delta_{x_3}(B)  & y=y_2\\
  \frac{p}{2}\delta_{x_1}(B)+ \frac{p}{2}\delta_{x_2}(B)+(1-p)\delta_{x_3}(B)   & y=y_3
\end{cases}$$
where $p\in [0, 1]$. Thus, for each $p\in [0, 1]$, we have the classical $D_6$-equivariant coupling
$$(((C(X), D_6, \alpha), (C(Y), D_6, \beta), ((C(Y), \phi_{\kappa_p}), \beta))((C(X), D_6, \alpha), C(X), \alpha)).$$

From the Markov kernel, the $C(Y)$-valued pairing on $C(X\times Y)$ is given by
$$\langle f, g\rangle(y)=\begin{cases}
    (1-p)\overline{f(x_1, y_1)}g(x_1, y_1)+ \frac{p}{2}\overline{f(x_2, y_1)}g(x_2, y_1)+\frac{p}{2}\overline{f(x_2, y_1)}g(x_2, y_1) & y=y_1\\
   \frac{p}{2}\overline{f(x_1, y_2)}g(x_1, y_2)+ (1-p)\overline{f(x_2, y_2)}g(x_2, y_2)+\frac{p}{2}\overline{f(x_3, y_2)}g(x_3, y_2)  & y=y_2\\
  \frac{p}{2}\overline{f(x_1, y_3)}g(x_1, y_3)+ \frac{p}{2}\overline{f(x_2, y_3)}g(x_2, y_3)+(1-p)\overline{f(x_3, y_3)}g(x_3, y_3)   & y=y_3
\end{cases}.$$
Therefore, it follows
$$C(X\times Y)/N_{\kappa_p}\cong \begin{cases}
    C(X\times Y) & 0<p<1\\
  C(\{(x_1, y_1), (x_2, y_2), (x_3, y_3)\})   & p=0\\
  C((X\times Y)\backslash \{(x_1, y_1), (x_2, y_2), (x_3, y_3)\})  & p=1
\end{cases}.$$
As $C(X\times Y)/N_{\kappa_p}$ is a finite dimensional inner product $C(Y)$-module, $C(X\times Y)/N_{\kappa_p}$ is a Hilbert $C(Y)$-module for all $p\in [0, 1]$.

In the case $p=0$, we have $C(X\times Y)/N_{\kappa_0}\cong C(Y)$ so that
$$\mathcal{L}(C(X\times Y)/N_{\kappa_0})\cong M_1(C(Y))\cong \mathbb{C}\oplus\mathbb{C}\oplus \mathbb{C}.$$
Under this identification, the $D_6$ action on $\mathbb{C}^{\oplus 3}$ is generated by
$$S(a, b, c)=(c, b, a)\quad\quad R(a, b, c)=(c, b, a)$$
Thus if follows $\mathcal{L}(C(X\times Y)/N_{\kappa_0})^{D_6}\cong \mathbb{C}$ making a relativization map equivalent to the choice of a state on $C(X)$.

In the cases $p=1$, we have $C(X\times Y)/N_{\kappa_1}\cong C(Y)\oplus C(Y)$ so that 
$$\mathcal{L}(C(X\times Y)/N_{\kappa_1})=M_2(C(Y))\cong M_2(\mathbb{C})\oplus M_2(\mathbb{C})\oplus M_2(\mathbb{C}).$$
For $i=1, 2, 3$, let
$$\chi_i:X\times Y\to \mathbb{C}\quad\quad \chi_i(x, y)=\begin{cases}
    1 & x=x_i\\
    0 & x\neq x_i
\end{cases}$$
and let $a, b:Y\to \mathbb{C}$ where
$$a(y)=\begin{cases}
    1 & y=y_2, y_3\\
    -1 & y=y_1
\end{cases}\quad\quad b(y)=\begin{cases}
    1 & y=y_1, y_2\\
    -1 & y=y_3
\end{cases}.$$
With respect to the basis $\chi_1+\chi_2$ and $\chi_2+\chi_3$ on $C(X\times Y)/N_{\kappa_1}$, the $\beta$-adjointable unitaries $W_S, W_R,$ and $W_{R^2}$ are given, respectively, by the elements $\tau_S, \rho_R, \rho_{R^2}\in M_2(C(Y))$ where
$$\tau_S=\begin{bmatrix}
    0 & 1\\
    1 & 0
\end{bmatrix}\quad\quad \rho_R=\begin{bmatrix}
    0 & a\\
    1 & b
\end{bmatrix}\quad\quad \rho_{R^2}=\begin{bmatrix}
    a & 1\\
    b & 0
\end{bmatrix}.$$
Under the identification of $\mathcal{L}(C(X\times Y)/N_{\kappa_1})$ with $M_2(\mathbb{C})^{\oplus 3}$, the $D_6$ action on $M_2(\mathbb{C})^{\oplus 3}$ is generated by
\begin{align*}
    S(A, B, C)&=(\tau(y_1) C\tau(y_3), \tau(y_2) B\tau(y_2), \tau(y_3) A\tau(y_1))\\
    R(A, B, C)&=(\rho_{R}(y_1) C \rho_{R^2}(y_3), \rho_R(y_2)A \rho_{R^2}(y_1), \rho_{R}(y_3) B\rho_{R^2}(y_2))
\end{align*}
Thus it follows
\begin{align*}
    \mathcal{L}(C(X\times Y)/N_{\kappa_1})^{D_6}&\cong \left\{(aI, aI, aI)\in M_2(\mathbb{C})^{\oplus 3}: a\in\mathbb{C}\right\}\cong \mathbb{C}
\end{align*}
Therefore a relativization map in this case is equivalent to the choice of a state on $C(X)$.

In the case $p\in (0, 1)$, we have $C(X\times Y)\cong C(Y)^{\oplus 3}$ so that
$$\mathcal{L}(C(X\times Y))\cong M_3(C(Y))\cong M_3(\mathbb{C})\oplus M_3(\mathbb{C})\oplus M_3(\mathbb{C}).$$
Under this identification, the $D_6$ action on $M_3(\mathbb{C})^{\oplus 3}$ is generated by
$$R(A, B, C)=(\sigma C\sigma^{-1}, \sigma A\sigma^{-1}, \sigma B\sigma^{-1})\quad\quad S(A, B, C)=(\phi C\phi^{-1}, \phi B\phi^{-1}, \phi A\phi^{-1})$$
where
$$\sigma=\begin{bmatrix}
    0 & 0 & 1\\
    1 & 0 & 0\\
    0 & 1 & 0
\end{bmatrix}\quad\quad \phi=\begin{bmatrix}
    0 & 0 & 1\\
    0 & 1 & 0\\
    1 & 0 & 0
\end{bmatrix}$$
Thus, it follows
\begin{align*}
    \mathcal{L}(C(X\times Y))^{D_6}&\cong \left\{(A, \sigma A\sigma^{-1}, \sigma^{-1} A\sigma)\in M_3(\mathbb{C})^{\oplus 3}: A=\begin{bmatrix}
        a & b & b\\
        b & a & b\\
        b & b & a
    \end{bmatrix}\text{ for }a, b\in\mathbb{C} \right\}\\
    &\cong \mathbb{C}\oplus \mathbb{C}
\end{align*}
Note, the last $*$-algebra isomorphism is given by $(A, \sigma A\sigma^{-1}, \sigma^{-1} A\sigma)\mapsto (a-b, a+2b)$. Therefore a relativization map in this case is equivalent to the data of two states on $C(X)$.

\end{example}

\begin{example}\

\noindent  Suppose we have both of our three-states switches $C(X)$ and $C(Y)$, but now suppose the symmetry group is from the rotation subgroup of $D_6$; that is, we have $\mathbb{Z}/3\mathbb{Z}$ symmetry group with actions on $X$ and $Y$ generated, respectively, by
$$[1]x_1=x_2\quad\quad [1]x_2=x_3\quad\quad [1]x_3=x_1$$
and
$$[1]y_1=y_2\quad\quad [1]y_2=y_3\quad\quad [1]y_3=y_1.$$
Denote the actions on $C(X)$ and $X$ by $\alpha$ and $\tilde{\alpha}$, respectively. Denote the actions on $C(Y)$ and $Y$ by $\beta$ and $\tilde{\beta}$, respectively.  Note, as the action of $\mathbb{Z}/3\mathbb{Z}$ is transitive, the algebra of invariant observables on $C(X)$ is given by $C(X)^{\mathbb{Z}/3\mathbb{Z}}\cong \mathbb{C}$.

It follows every $\mathbb{Z}/3\mathbb{Z}$-invariant Markov kernel from $Y$ to $X$ is of the form
$$\kappa_{p, q}:Y\times\text{Bor}(X)\to [0, 1]\quad\quad \kappa_{p, q}(y, B)=\begin{cases}
    (1-p-q)\delta_{x_1}(B)+ p\delta_{x_2}(B)+q\delta_{x_3}(B) & y=y_1\\
   q\delta_{x_1}(B)+ (1-p-q)\delta_{x_2}(B)+p\delta_{x_3}(B)  & y=y_2\\
  p\delta_{x_1}(B)+ q\delta_{x_2}(B)+(1-p-q)\delta_{x_3}(B)   & y=y_3
\end{cases}$$
where $p, q\in [0, 1]$ with $p+q\leq 1$. Thus, for each such $p$ and $q$, we have the classical $\mathbb{Z}/3\mathbb{Z}$-equivariant coupling
$$(((C(X),\mathbb{Z}/3\mathbb{Z}, \alpha), (C(Y),\mathbb{Z}/3\mathbb{Z}, \beta), ((C(Y), \phi_{\kappa_{p, q}}), \beta)), ((C(X), \mathbb{Z}/3\mathbb{Z}, \alpha), C(X), \alpha)).$$

From the Markov kernel, the $C(Y)$-valued pairing on $C(X\times Y)$ is given by
$$\langle f, g\rangle(y)=\begin{cases}
    (1-p-q)\overline{f(x_1, y_1)}g(x_1, y_1)+ p\overline{f(x_2, y_1)}g(x_2, y_1)+q\overline{f(x_2, y_1)}g(x_2, y_1) & y=y_1\\
  q\overline{f(x_1, y_2)}g(x_1, y_2)+ (1-p-q)\overline{f(x_2, y_2)}g(x_2, y_2)+p\overline{f(x_3, y_2)}g(x_3, y_2)  & y=y_2\\
  p\overline{f(x_1, y_3)}g(x_1, y_3)+ q\overline{f(x_2, y_3)}g(x_2, y_3)+(1-p-q)\overline{f(x_3, y_3)}g(x_3, y_3)   & y=y_3
\end{cases}.$$
Therefore, it follows
$$C(X\times Y)/N_{\kappa_{p, q}}\cong \begin{cases}
    C(X\times Y) &  p, q\in (0, 1), p+q<1\\
   C((X\times Y)\backslash \{(x_1, y_3), (x_2, y_1), (x_3, y_2)\})  & p=0, q\in (0, 1)\\
   C((X\times Y)\backslash \{(x_1, y_2), (x_2, y_3), (x_3, y_1)\}) & p\in (0, 1), q=0\\
C((X\times Y)\backslash \{(x_1, y_1), (x_2, y_2), (x_3, y_3)\})  & p, q\in (0, 1), p+q=1\\
C(\{(x_1, y_1), (x_2, y_2), (x_3, y_3)\})  & p= 0, q=0\\
C(\{(x_1, y_3), (x_2, y_1), (x_3, y_2)\}) & p=1, q=0\\
C(\{(x_1, y_2), (x_2, y_3), (x_3, y_1)\}) & p=0, q=1\\
\end{cases}.$$

We now compute the fixed-points of adjointable operators on $C(X\times Y)/N_{\kappa_{p, q}}$. First, consider the cases of $p, q\in (0, 1)$ and $p+q<1$.  As we have $C(X\times Y)\cong C(Y)^{\oplus 3}$, it follows
$$\mathcal{L}(C(X\times Y))\cong M_3(C(Y))\cong M_3(\mathbb{C})^{\oplus 3}.$$
Let $\rho\in M_3(\mathbb{C})$ where
$$\rho=\begin{bmatrix}
    0 & 0 & 1\\
    1 & 0 & 0\\
    0 & 1 & 0
\end{bmatrix},$$
then the action of $\mathbb{Z}/3\mathbb{Z}$ on $M_3(\mathbb{C})^{\oplus 3}$ is generated by
$$[1](A, B, C)=(\rho C\rho^{-1}, \rho A\rho^{-1}, \rho B \rho^{-1}).$$
Therefore it follows
$$\mathcal{L}(C(X\times Y))^{\mathbb{Z}/2\mathbb{Z}}\cong \left\{(A, \rho A\rho^{-1}, \rho A\rho^{-1}: A\in M_3(\mathbb{C})\right\}\cong  M_3(\mathbb{C}).$$
Thus a relativization map in this case is equivalent to a unital completely positive map $\pounds:C(X)\to M_3(\mathbb{C})$. Note, an example of such a map is given by
$$\pounds:C(X)\to M_3(\mathbb{C})\quad\quad \pounds(f)=\begin{bmatrix}
    \mathbb{E}[f(X)| Y=y_1] & 0 & 0\\
    0 & \mathbb{E}[f(X)| Y=y_2] & 0 \\
    0 & 0 & \mathbb{E}[f(X)| Y=y_3]
\end{bmatrix}.$$

Now let us consider the cases $p=0=q$, $p=1$ and $q=0$, and the case $p=0$ and $q=1$. In these three cases, we have $C(X\times Y)/N_{\kappa_{p, q}}\cong C(Y)$ so that
$$\mathcal{L}(C(X\times Y)/N_{\kappa_{p, q}})\cong M_1(C(Y))\cong \mathbb{C}\oplus \mathbb{C}\oplus \mathbb{C}.$$
Under this identification, the $\mathbb{Z}/3\mathbb{Z}$ action on $\mathbb{C}^{\oplus 3}$ is generated by 
$$[1](a, b, c)=(c, a, b)$$
Therefore $\mathcal{L}(C(X\times Y)/N_{\kappa_{p, q}})^{\mathbb{Z}/3\mathbb{Z}}\cong \mathbb{C}$ so that the data of relativization map in these cases is equivalent to the data of a state on $C(X)$.

Now let us consider the last three generic cases. As $C(X\times Y)/N_{\kappa_{p, q}}\cong C(Y)\oplus C(Y)$, it follows
$$\mathcal{L}(C(X\times Y)/N_{\kappa_{p, q}})=M_2(C(Y))\cong M_2(\mathbb{C})\oplus M_2(\mathbb{C})\oplus M_2(\mathbb{C}).$$
For $i=1, 2, 3$, let
$$\chi_i:X\times Y\to \mathbb{C}\quad\quad \chi_i(x, y)=\begin{cases}
    1 & x=x_i\\
    0 & x\neq x_i
\end{cases}$$
and let $a, b:Y\to \mathbb{C}$ where
$$a(y)=\begin{cases}
    1 & y=y_2, y_3\\
    -1 & y=y_1
\end{cases}\quad\quad b(y)=\begin{cases}
    1 & y=y_1, y_2\\
    -1 & y=y_3
\end{cases}.$$
With respect to the basis $\chi_1+\chi_2$ and $\chi_2+\chi_3$ on $C(X\times Y)/N_{\kappa_{p, q}}$, the $\beta$-adjointable unitaries $W_{[1]}$ and  $W_{[2]}$ are given, respectively, by the elements $\rho_{[1]}, \rho_{[2]}\in M_2(C(Y))$ where
$$\rho_R=\begin{bmatrix}
    0 & a\\
    1 & b
\end{bmatrix}\quad\quad \rho_{R^2}=\begin{bmatrix}
    a & 1\\
    b & 0
\end{bmatrix}.$$
Under the identification of $\mathcal{L}(C(X\times Y)/N_{\kappa_{p, q}})$ with $M_2(\mathbb{C})^{\oplus 3}$, the $\mathbb{Z}/2\mathbb{Z}$ action on $M_2(\mathbb{C})^{\oplus 3}$ is generated by
\begin{align*}
    [1](A, B, C)&=(\rho_{[1]}(y_1) C \rho_{[2]}(y_3), \rho_{1}(y_2)A \rho_{[2]}(y_1), \rho_{[1]}(y_3) B\rho_{[2]}(y_2))
\end{align*}
Thus it follows
\begin{align*}
    \mathcal{L}&(C(X\times Y)/N_{\kappa_{p, q}})^{\mathbb{Z}/3\mathbb{Z}}\\
    &\cong \left\{\left(A,\begin{bmatrix}
        0 & 1\\
        1 & 1
    \end{bmatrix} A\begin{bmatrix}
        -1 & 1\\
        1 & 0
    \end{bmatrix}, \begin{bmatrix}
        1 & -1\\
        1 & 0
    \end{bmatrix}A\begin{bmatrix}
        0 & -1\\
        1 & 1
    \end{bmatrix}\right)\in M_2(\mathbb{C})^{\oplus 3}: A\in M_2(\mathbb{C})\right\}\\
    &\cong M_2(\mathbb{C})
\end{align*}
Therefore a relativization map in this case is equivalent to the data of a unital completely positive map $\pounds:C(X)\to M_2(\mathbb{C})$.

\end{example}

\begin{example}\

\noindent We can model a classical particle on a line by the $C^*$-algebra $C_0(\mathbb{R})$. Assuming the line is not marked, we are unable to distinguish the absolute states (position on the line) which we encode by the translation action of $\mathbb{R}$ on $C_0(\mathbb{R})$. Denote this action by $\alpha$. Thus, without a reference frame, the algebra of invariant observables is given by $C_b(\mathbb{R})^{\mathbb{R}}\cong\mathbb{C}$. 

Suppose we use another particle on the line to act as reference system. Let
$$\kappa:\mathbb{R}\times\text{Bor}(\mathbb{R})\to [0, 1]\quad\quad \kappa(y, B)=\frac{1}{\sqrt{2\pi \sigma}}\int_{B}\exp\left(\frac{-(y-x)^2}{2\sigma}\right)\;dx$$
which defines an $\mathbb{R}$-invariant weakly continuous Markov kernel with respect to the action $\tilde{\alpha}$ on $\mathbb{R}$. Therefore we obtain the classical $\mathbb{R}$-equivariant coupling 
$$(((C_0(\mathbb{R}), \mathbb{R}, \alpha), (C_0(\mathbb{R}), \mathbb{R}, \alpha), ((C_0(\mathbb{R}), \phi_\kappa), \alpha)), ((C_0(\mathbb{R}), \mathbb{R}, \alpha), C_0(\mathbb{R}), \alpha)).$$

As the stabilizer for the translation action on $\mathbb{R}$ is trivial, Theorem \ref{Theorem: rel map for transitive actions} implies that every map in $C(\mathbb{R}, \mathbb{R})$ defines a relativization map for the equivariant coupling. For example, if we take the identity map, then the induced relativization map is given by
$$\pounds_{\text{id}}:C_b(\mathbb{R})\to \mathcal{L}(\Gamma_{\kappa_\phi})\quad\quad \pounds(f)=\mathcal{M}^{\text{id}}_{f}$$
where for each $g\in C_0(\mathbb{R}\times \mathbb{R})$
$$(\mathcal{M}^{\text{id}}_fg)(x, y)=f(x-y)g(x, y).$$
More generally, if $h\in C(\mathbb{R}, \mathbb{R})$, then the induced relativization map is given by
$$\pounds_h:C_b(\mathbb{R})\to \mathcal{L}(\Gamma_{\kappa_\phi})\quad\quad \pounds_h(f)=\mathcal{M}^h_{f}$$
where for each $g\in C_0(\mathbb{R}\times \mathbb{R}),$
$$(\mathcal{M}^h_fg)(x, y)=f(h(x-y))g(x, y).$$

\end{example}

\begin{example}\

\noindent Suppose we want to model a classical particle on $\mathbb{S}^2$. As a $C^*$-algebra, we take the unital $C^*$-algebra $C(\mathbb{S}^2)$. Assuming the sphere is not marked, we unable to distinguish the absolute states (positions on the sphere) which we encode by the standard action of $SO(3)$ on $\mathbb{S}^2$. Denote this action by $\tilde{\alpha}$, and denote the corresponding action on $C(\mathbb{S}^2)$ by $\alpha$. Note, as the action is transitive, we have $C(\mathbb{S}^2)^{SO(3)}\cong\mathbb{C}$. 

Suppose we use another particle on the sphere to acts as reference system. Let $\sigma$ denote surface measure on $\mathbb{S}^2$, and let $\gamma>0$. Consider the von Mises-Fischer kernel for $\mathbb{S}^2$: 
$$\kappa_\gamma:\mathbb{S}^2\times\text{Bor}(\mathbb{S}^2)\to [0, 1]\quad\quad \kappa_\gamma(\vec{y}, B)=\frac{\gamma}{4\pi \sinh(\gamma)}\int_{B}\exp(\gamma (\vec{x}\cdot \vec{y}))\;d\sigma(\vec{x})$$
Note $\kappa_\gamma$ defines a $SO(3)$-invariant weakly continuous Markov kernel with respect to the action $\tilde{\alpha}$ on $\mathbb{S}^2$. Therefore, we obtain the classical $SO(3)$-equivariant coupling
$$(((C(\mathbb{S}^2), SO(3), \alpha), (C(\mathbb{S}^2), SO(3), \alpha), ((C(\mathbb{S}^2), \phi_{\kappa_\gamma}), \alpha)), ((C(\mathbb{S}^2), SO(3), \alpha), C(\mathbb{S}^2), \alpha)).$$

The stabilizer of the $SO(3)$ action on $\mathbb{S}^2$ is isomorphic to $SO(2)$. Therefore, we have  $SO(2)\backslash SO(3)/SO(2)\cong [-1, 1]$ given by
$$SO(2)\backslash SO(3)/SO(2)\cong (\mathbb{S}^2\times \mathbb{S}^2)/SO(3)\to [-1, 1]\quad\quad [(\vec{x}, \vec{y})]\mapsto \vec{x}\cdot\vec{y}.$$
Furthermore, we have $SO(3)/SO(2)\cong \mathbb{S}^2$. By Theorem \ref{Theorem: rel map for transitive actions}, every continuous map in $C([-1, 1], \mathbb{S}^2)$ defines a relativization map for the equivariant coupling. In particular, given $\iota\in C([-1, 1], \mathbb{S}^2)$, we obtain the relativization map
$$\pounds_{\iota}:C(\mathbb{S}^2)\to \mathcal{L}(\Gamma_{\kappa_\gamma})\quad\quad \pounds_{\iota}(f)=\mathcal{M}_{f}^{\iota}$$
where for each $g\in C(\mathbb{S}^2\times\mathbb{S}^2)$
$$(\mathcal{M}^\iota_fg)(\vec{x}, \vec{y})=f(\iota(\vec{x}\cdot\vec{y}))g(\vec{x}, \vec{y}).$$

\end{example}

\subsection{Hilbert module classical reference frames and the frame bundle}

Let $M$ be a smooth $m$-dimensional manifold. Let $\pi:\text{Fr}(M)\to M$ denote the frame bundle for $M$; that is, for each $p\in M$, 
$$\pi^{-1}(p)=\left\{(v_{1, p}, \ldots, v_{m, p})\in (T_pM)^m: v_{1, p}, \ldots, v_{m, p}\text{ is a basis for }T_pM\right\}.$$
Let $\pi':TM\to M$ denote the tangent bundle. For each $1\leq i\leq m$, define
$$\text{pr}_i:\text{Fr}(M)\to TM\quad\quad \text{pr}_i(v_{1, p}, \ldots, v_{m, p})=v_{i, p}$$
which is a smooth map with the property $\pi'\circ\text{pr}_i=\pi$.

Using the frame bundle, a classical reference frame for $M$ can be interpreted as the data $(U, s)$ where $U\subset M$ is open and $s:U\to\text{Fr}(M)$ is a local smooth section. This interpretation derives from the one-to-one correspondence between the data $(U, s)$ and a local frame of the tangent bundle. As a result of this correspondence, the data $(U, s)$ enables one to locally describe geometric information of the tangent and cotangent bundles such as metrics and connections. 

As remarked in the introduction, specifying a local section represents a deterministic/sharp classical reference frame, but a comprehensive model of classical reference frames must include a description of both sharp and unsharp frames. To model unsharp frames using the frame bundle, we utilize weakly continuous Markov kernels to describe "unsharp sections". Unlike the sharp description, these unsharp sections characterize a probability distribution (or more generally a probability measure) for the basis of the tangent space at each point. To formally enforce the geometric requirement that these kernels behave like a section, we utilize push-forwards of Markov kernels. 

First, we recall that if $X, Y, Z$ are locally compact Hausdorff spaces and if we have a weakly continuous Markov kernel $\kappa:X\times \text{Bor}(Y)\to [0, 1]$, then the push-forward of $\kappa$ along a continuous map $f:Y\to Z$ is defined as the weakly continuous Markov kernel
$$f_*(\kappa):X\times\text{Bor}(Z)\to [0, 1]\quad\quad f_*(\kappa)(x, B)=\kappa(x, f^{-1}(B)).$$
From this definition, it immediately follows that if we have another locally compact Hausdorff space $W$ and continuous map $g:Z\to W$, then $g_*(f_*(\kappa))=(g\circ f)_*(\kappa)$. 

Using the push-forward, we define an \textit{unsharp section} of the frame bundle as the data $(U, \kappa)$ where $U\subset M$ is open and $\kappa:U\times\text{Bor}(\text{Fr}(M))\to[0, 1]$ is weakly continuous Markov kernel, and the data satisfies the condition $\pi_*\kappa=\delta_{\text{inc}}$ where $\delta_{\text{inc}}$ is the induced weakly continuous Markov kernel from the inclusion map $\text{inc}:U\to M$. Note, in the case $\kappa=\delta_s$ for $s$ a section of the frame bundle, the condition $\pi_*\kappa=\delta_{\text{inc}}$ is equivalent to $\pi\circ s=\text{inc}$. Similar to sections of the frame bundle, unsharp sections induce unsharp sections of the tangent bundle.

\begin{proposition}\

\noindent Let $M$ be a smooth $m$-dimensional manifold. If $(U, \kappa)$ is an unsharp section of the frame bundle, then, for all $1\leq i\leq n$,  $\kappa_i:=(\text{pr}_i)_*(\kappa)$ is a weakly continuous Markov kernel with the property that $\pi'_*(\kappa_i)=\delta_{\text{inc}}$.

\end{proposition}

\begin{proof}\

\noindent As $\pi'\circ\text{pr}_i=\pi$ for all $1\leq i\leq m$ and $\pi_*(\kappa)=\delta_{\text{inc}}$, it follows that
$$\pi'_*(\kappa_i)=\pi'_*((\text{pr}_i)_*(\kappa))=(\pi'\circ \text{pr}_i)_*(\kappa)=\pi_*(\kappa)=\delta_{\text{inc}}.$$

\end{proof}

We can generalize our setup of unsharp sections by replacing the data $(U, \kappa)$ with $(Y, \kappa, \theta)$ where $Y$ is a locally compact Hausdorff space, $\kappa:Y\times \text{Bor}(M)\to [0, 1]$ is a weakly continuous Markov kernel, and  $\theta:Y\times\text{Bor}(\text{Fr}(M))\to [0, 1]$ is a weakly continuous Markov kernel, and the data satisfies the condition $\pi_*(\theta)=\kappa$. Given the data of $Y$ and $\kappa$, we can ask whether there always exists such a $\theta$. Note, in the case of sections, the answer is generally no as sections of the frame bundle are in one-to-one correspondence with local trivializations. However, by equipping $M$ with a Riemannian metric to reduce the structure group to the compact group $O(m)$, we can utilize the normalized Haar measure on $O(m)$ to construct $\theta$. Conceptually, we construct $\theta$ by first sampling, for each $y\in Y$, a point $p\in M$ relative to $\kappa(y, \cdot)$, and then pick a random orthonormal basis in the fiber of $p$ using the normalized Haar measure on $O(m)$.

\begin{lemma}\

\noindent Let $Z$ be a locally compact Hausdorff second countable topological space. For each bounded Borel measurable function $f:Z\to \mathbb{C}$, the function
$$\Phi_f:\text{Prob}_R(Z, \text{Bor}(Z))\to \mathbb{R}\quad\quad \Phi_f(\mu)=\int_{Z}f(z)\;d\mu(z)$$ is Borel measurable with respect to the weak topology.

\end{lemma}

\begin{proof}\

\noindent It suffices to prove the statement holds for bounded Borel measurable functions $f:Z\to \mathbb{R}$. Our goal is to apply the Functional Monotone Class Theorem (Theorem 5.2.2 in \cite{durrett2019probability}). Let $\mathcal{A}$ be the set of all open subsets of $Z$; observe $\mathcal{A}$ forms a $\pi$-system for $Z$. Let $\mathcal{S}$ be the set of all bounded Borel measurable function $f:Z\to \mathbb{R}$ for which $\Phi_f$ is Borel measurable with respect to the weak topology. From linearity of the integral as well as that the sum of measurable functions is measurable, $\mathcal{S}$ forms a vector space. Additionally, $C_b(Z)_{SA}\subset \mathcal{S}$ by definition of the weak topology.

We claim that for each $U\in \mathcal{A}$, $\chi_U\in \mathcal{S}$. As $Z$ is a second countable locally compact Hausdorff space, there exists an increasing sequence $(f_n)_{n\in\mathbb{N}}$ of positive functions in $C_0(Z)$ such that $f_n\to \chi_U$ pointwise. Therefore, the Monotone Convergence Theorem implies
$$\Phi_{\chi_U}(\mu)=\int_{Z}\chi_U(z)\;d\mu(z)=\lim\limits_{n\to\infty}\int_Z f_n(z)\;d\mu(z)=\lim\limits_{n\to\infty}\Phi_{f_n}(\mu)$$
for all $\mu\in \text{Prob}_R(Z, \text{Bor}(Z))$. As the pointwise limit of Borel measurable functions is Borel measurable, we conclude $\Phi_{\chi_U}$ is Borel measurable so that $\chi_U\in \mathcal{S}$.

Now suppose $(f_n)_{n\in\mathbb{N}}$ is an increasing sequence of non-negative functions in $\mathcal{S}$ which converge pointwise to a function $f:Z\to \mathbb{R}$. We claim $f\in\mathcal{S}$.  As the pointwise limit of Borel measurable functions, $f$ is a Borel measurable function. Furthermore, by the Monotone Convergence Theorem
$$\Phi_{f}(\mu)=\int_{Z}f(z)\;d\mu(z)=\lim\limits_{n\to\infty}\int_Z f_n(z)\;d\mu(z)=\lim\limits_{n\to\infty}\Phi_{f_n}(\mu)$$ By the same argument as before, it follows $f\in \mathcal{S}$.

Having verified the hypotheses of the Functional Monotone Convergence Theorem, we conclude $\mathcal{S}$ contains all real valued bounded Borel measurable functions on $Z$. 
    
\end{proof}

\begin{lemma}\

\noindent Let $X$ be a locally compact Hausdorff topological space, and let $Y$ and $Z$ be second countable locally compact Hausdorff spaces. If $\kappa:X\times\text{Bor}(Y)\to [0, 1]$ and $\sigma:Y\times\text{Bor}(Z)\to [0, 1]$ are weakly continuous Markov kernels, then the map
$$\theta:X\times\text{Bor}(Z)\to [0, 1]\quad\quad \theta(x, B)=\int_{Y}\sigma(y, B)\;\kappa(x, dy)$$
defines a weakly continuous Markov kernel.

\end{lemma}

\begin{proof}\

\noindent First, we verify $\theta$ is a well-defined function. Fix $x\in X$ and $B\in \text{Bor}(Z)$. As $\sigma$ is a Markov kernel, the map
$$\sigma(\cdot, B):Y\to [0, 1]\quad\quad y\mapsto \sigma(y, B)$$
is a bounded Borel measurable function. As $\kappa$ is a Markov kernel, $\kappa(x, \cdot)$ is a probability measure on $(Y, \text{Bor}(Y))$. Therefore $\sigma(\cdot, B)$ is integrable with respect to $\kappa(y, \cdot)$ with
$$0\leq \int_{Y}\sigma(y, B)\kappa(x, dy)\leq 1$$
for all $x\in X$. Thus $\theta(x, B)$ is defined and $\theta(x, B)\in [0, 1]$. Hence $\theta$ is a well-defined function.

We now verify $\theta(x, \cdot)$ is a Radon probability measure. For each $x\in X$, we have the map
$$\theta(x, \cdot):\text{Bor}(Z)\to [0, 1]\quad\quad \theta(x, B)=\int_{Y}\sigma(y, B)\kappa(x, dy).$$
Since $\sigma(\cdot, Z)=1$ and $\sigma(\cdot, \emptyset)=0$, then $\theta(x, Z)=1$ and $\theta(x, \emptyset)=0$. Let $(B_n)_{n\in\mathbb{N}}$ be a disjoint sequence of Borel subsets of $Z$. As $\sigma(y, \cdot)$ is a measure for each $y\in Y$, we have
$$\sigma(y, \bigcup\limits_{n=1}^\infty B_n)=\sum\limits_{n=1}^\infty \sigma(y, B_n)$$
As $\left(\sum\limits_{k=1}^n \sigma(\cdot, B_k)\right)_{n\in\mathbb{N}}$ is a monotone increasing sequence of positive functions which converge pointwise to $\sigma(\cdot, \bigcup\limits_{n=1}^\infty B_n)$, we can apply the Monotone Convergence Theorem to know
\begin{align*}
    \theta(x, \bigcup\limits_{n=1}^\infty B_n)&=\int_{Y}\sigma(y, \bigcup\limits_{n=1}^\infty B_n)\;\kappa(x, dy)=\lim\limits_{n\to \infty}\sum\limits_{k=1}^n\int_{Y}\sigma(y, B_k)\;\kappa(x, dy)\\
    &=\lim\limits_{n\to\infty}\sum\limits_{k=1}^n\theta(x, B_k)=\sum\limits_{n=1}^\infty \theta(x, B_n)
\end{align*}
Therefore $\theta(x, \cdot)$ is a probability measure. As $Z$ is a second countable locally compact Hausdorff space and $\theta(x, \cdot)$ is a finite positive Borel measure, the measure is automatically a Radon measure. Thus $\theta(x, \cdot)$ is a Radon probability measure.

Now we verify the map
$$X\to \text{Prob}_R(X, \text{Bor}(X))\quad\quad x\mapsto \theta(x, \cdot)$$
is weakly continuous. Let $g\in C_b(Z)^+$, then we have
$$G:X\to \mathbb{C}\quad\quad G(x)=\int_{Z}g(z)\;\theta(x, dz)=\int_{Z}\int_Y g(z)\sigma(y, dz)\kappa(x, dy)$$
For a fixed $x\in X$, we apply Toneli's Theorem to obtain 
$$\int_{Z}\int_Y g(z)\sigma(y, dz)\kappa(x, dy)=\int_{Y}\int_Z g(z)\sigma(y, dz)\kappa(x, dy)$$
Let
$$G':Y\to \mathbb{C}\quad\quad G'(y)=\int_{Z} g(z)\sigma(y, dz)$$
Since $\sigma$ is a weakly continuous Markov kernel, $G'\in C_b(Y)$. As $\kappa$ is a weakly continuous Markov kernel, 
$$G(x)=\int_{Y}G'(y)\;\kappa(x, dy)$$
is continuous. As $g\in C_b(Z)^+$ was arbitrary, it follows the map $G$ is continuous for all $g\in C_b(Z)$. Hence
$$X\to \text{Prob}_R(X, \text{Bor}(X))\quad\quad x\mapsto \theta(x, \cdot)$$
is continuous with respect to the weak topology.

Finally, we verify $\theta(\cdot, B)$ is Borel measurable for each $B\in\text{Bor}(Z)$. Note, we can write $\theta(\cdot, B)$ as the composition
$$X\xrightarrow{x\mapsto \kappa(x, \cdot)} \text{Prob}(Y, \text{Bor}(Y))\xrightarrow{\mu\mapsto \int_{Y}\sigma(y, B)\;d\mu(y)} [0, 1]$$
As $\kappa$ is a weakly continuous Markov kernel, the map $x\mapsto \kappa(x, \cdot)$ is continuous, thus measurable. By the prior Lemma as well as that $\sigma$ is a Markov kernel, the second map is Borel measurable. Thus the composition is Borel measurable proving $\theta(\cdot, B)$ is Borel measurable.
    
\end{proof}

\begin{theorem}\label{Theorem: lifting WCMK}\

\noindent Let $M$ be a smooth $m$-dimensional  manifold, and let $Y$ be a locally compact Hausdorff space. If $\kappa:Y\times\text{Bor}(M)\to [0, 1]$ is a weakly continuous Markov kernel, then there exists a weakly continuous Markov kernel $\theta:Y\times \text{Bor}(\text{Fr}(M))\to [0, 1]$ such that $\pi_*(\theta)=\kappa$.

\end{theorem}

\begin{proof}\

\noindent Let $g$ be a Riemannian metric for $M$. Let $\rho:O(M, g)\to M$ be the orthonormal frame bundle. For each $p\in M$, pick an element $o_p\in \rho^{-1}(p)$. As $O(m)$ acts transitively and freely on $\rho^{-1}(p)$, we obtain a diffeomorphism
$$R_{o_p}:O(m)\to \rho^{-1}(p)\quad\quad R_{o_p}(A)=o_pA$$
Let $\lambda$ denote the unique normalized Haar measure on $O(m)$, and let $\sigma_{p}=(R_{o_p})_{*}(\lambda)$ which defines a Radon probability measure on $(\rho^{-1}(p), \text{Bor}(\rho^{-1}(p)))$. We claim the measure $\sigma_p$ is independent of $o_p$. Suppose we pick another element $o_p'\in \rho^{-1}(p)$ to obtain an alternative measure $(R_{o_p'})_{*}(\lambda)$. Let $g\in O(m)$ such that $o_pg=o_p'$, then $R_{o_p}\circ L_g=R_{o_p'}$. Thus
$$\sigma_p'=(R_{o_p'})_*(\lambda)=(R_{o_p}\circ L_g)_*(\lambda)=(R_{o_p})_*(\lambda)=\sigma_p.$$

Define the map
$$\sigma:M\times\text{Bor}(\text{Fr}(M))\to[0, 1]\quad\quad \sigma(p, B)=\sigma_p(B\cap \rho^{-1}(p))$$
We claim $\sigma$ is a weakly continuous Markov kernel. It is clear that for each $p\in M$, $\sigma(p, \cdot)$ is a probability measure. As $\sigma(p, \cdot)$ is a probability measure on a locally compact second countable topological space, it follows $\sigma(p, \cdot)$ is a Radon probability measure. 

We check the map
$$M\to \text{Prob}_R(\text{Fr}(M), \text{Bor}(\text{Fr}(M)))\quad\quad p\mapsto \sigma(p, \cdot)$$
is continuous with respect to the weak topology. Fix $f\in C_b(\text{Fr}(M))$, and let
$$F:M\to \mathbb{C}\quad\quad F(p)=\int_{\text{Fr}(M)}f(b)\;\sigma(p, db)$$
Let $(U, \phi)$ be a local trivialization of $\rho:O(M, g)\to M$, then, for each $p\in U$ we have 
\begin{align*}
    F(p)&=\int_{\text{Fr}(M)}f(b)\;\sigma(p, db)=\int_{\rho^{-1}(U)}f(b)\;\sigma(p, db)\\
    &=\int_{U\times O(m)}(f\circ \phi^{-1})(q, A)\; d(\delta_p\times \lambda)(q, A)=\int_{O(m)}(f\circ \phi^{-1})(p, A)\;d\lambda(A)
\end{align*}
Thus, using the Dominated Convergence Theorem, $F|_U$ is continuous. As we can cover $M$ in local trivializations, it follows $F$ is continuous. Hence the map
$$M\to \text{Prob}_R(\text{Fr}(M), \text{Bor}(\text{Fr}(M)))\quad\quad p\mapsto \sigma(p, \cdot)$$
is continuous with respect to the weak topology.

Finally, we verify $\sigma(\cdot, B)$ is measurable for each $B\in \text{Bor}(\text{Fr}(M))$. Since we have
$$\sigma(\cdot, B)=\sigma(\cdot, B\cap O(M, g)),$$
it suffices to verify measurability for $B\in \text{Bor}(O(M, g))$. First, let $(U, \phi)$ be a local trivialization of $\rho:O(M, g)\to M$. We claim that for each Borel measurable subset $B\subset \rho^{-1}(U)$, $\sigma(\cdot, B)$ is measurable. Note,
$$\sigma(\cdot, B):M\to [0, 1]\quad\quad \sigma(p, B)=\lambda(\{A\in O(m): \phi^{-1}(p, A)\in B \})$$
To prove the claim, we utilize Dykin's $\pi-\lambda$ Theorem. Let
$$\mathcal{A}=\{\phi^{-1}(W\times P): W\subset U\text{ and }P\subset O(m)\text{ are open}\}$$
which forms a $\pi$-system for $\rho^{-1}(U)$. Let 
$$\mathcal{S}=\{B\in \text{Bor}(\rho^{-1}(U)): \sigma(\cdot, B)\text{ is measurable}\}$$
We claim $\mathcal{S}$ forms a $\lambda$-system. As $\sigma(\cdot, \rho^{-1}(U))=\chi_{U}$  is measurable, $\rho^{-1}(U)\in \mathcal{S}$. Given $B\in \mathcal{S}$, we have
$$\sigma(\cdot, \rho^{-1}(U)\backslash B)=\chi_U-\sigma(\cdot, B)$$
As the difference of measurable functions is measurable, $\mathcal{S}$ is closed under complements. Suppose $(B_n)_{n\in\mathbb{N}}$ is an increasing sequence of subset in $\mathcal{S}$, then
$$\sigma(\cdot, \bigcup\limits_{n=1}^\infty B_n)=\lim\limits_{n\to \infty}\sigma(\cdot, B_n)$$
where the limit converges pointwise. As the pointwise limit measurable functions is measurable, we have $\bigcup\limits_{n=1}^\infty B_n\in \mathcal{S}$. Thus $\mathcal{S}$ is a $\lambda$ system. Lastly, we check $\mathcal{A}\subset \mathcal{S}$. Suppose $V\in \mathcal{A}$ such that $\phi(V)=W\times P$ for $W\subset U$ open and $P\subset O(m)$ open, then $\sigma(\cdot, V)=\lambda(P)\chi_W$ which is measurable. Therefore $V\in \mathcal{S}$.  Applying Dykin's $\pi-\lambda$ Theorem, we know $\text{Bor}(\rho^{-1}(U))\subset \mathcal{S}$.

Let $\{(U_n, \phi_n)\}_{n\in\mathbb{N}}$ be a cover for $M$ in terms of local trivialization for the orthonormal frame bundle. Let
$$E_1=\rho^{-1}(U_1)\quad\quad E_n=U_n\backslash \left(\bigcup\limits_{k=1}^{n-1} \rho^{-1}(U_k)\right),$$
then $(E_n)_{n\in\mathbb{N}}$ forms a cover of $O(M, g)$ in terms of disjoint Borel measurable subsets with $E_i\subset \rho^{-1}(U_i)$. Therefore, given $B\in \text{Bor}(O(M, g))$
$$\sigma(\cdot, B)=\sigma(\cdot, \bigcup\limits_{n=1}^\infty (B\cap E_n))=\sum\limits_{n=1}^\infty \sigma(\cdot, B\cap E_n)$$
As $B\cap E_n\in \text{Bor}(\rho^{-1}(U_n))$, each $\sigma(\cdot, B\cap E_n)$ is measurable. Therefore $\sigma(\cdot, B)$ is the pointwise limit of Borel measurable functions implying $\sigma(\cdot, B)$ is Borel measurable. 

Using $\sigma$, we now construct $\theta$. Define
$$\theta:Y\times \text{Bor}(\text{Fr}(M))\to [0, 1]\quad\quad \theta(y, B)=\int_{M} \sigma(p, B)\kappa(y, dp)$$
By the prior Lemma, $\theta$ is a weakly continuous Markov kernel. Furthermore, $\pi_*(\theta)=\kappa$. Indeed, let $y\in Y$ and $B\in\text{Bor}(M)$, then
\begin{align*}
    \pi_*(\theta)(y, B)&=\theta(y, \pi^{-1}(B))=\int_{M}\sigma(p, \pi^{-1}(B))\kappa(y, dp)\\
    &=\int_{M}\sigma_p(\pi^{-1}(B)\cap \rho^{-1}(p))\;\kappa(y, dp)=\int_{B}\sigma_p(\rho^{-1}(p))\;\kappa(y, dp)\\
    &=\int_{B}\kappa(y, dp)=\int_{M}\chi_B(p)\;\kappa(y, dp)=\kappa(y, B)
\end{align*}
    
\end{proof}

Now suppose $G$ is a Lie group which acts smoothly on $M$ by $\tilde{\alpha}$ and acts on $Y$ by $\tilde{\beta}$. For each $s\in G$, the smooth diffeomorphism $\tilde{\alpha}_s:M\to M$ induces a linear isomorphism $(\tilde{\alpha}_{s})_{*, p}:T_pM\to T_{\tilde{\alpha}_s(p)}M$ for each $p\in M$. Therefore we can define a smooth action of $G$ on $\text{Fr}(M)$ given by
$$\tilde{\gamma}:G\times \text{Fr}(M)  \to \text{Fr}(M)\quad\quad \tilde{\gamma}(s, (v_{1, p}, \ldots, v_{n, p}))=((\tilde{\alpha}_s)_{*, p}(v_{1, p}),\ldots,  (\tilde{\alpha}_s)_{*, p}(v_{n, p}))$$
With respect to this action, $\pi:\text{Fr}(M)\to M$ is $G$-equivariant. Given a $G$-invariant Markov kernel $\kappa:Y\times\text{Bor(M)}\to [0, 1]$, we can ask if the Markov kernel lifts to a $G$-invariant Markov kernel on $\text{Fr}(M)$; that is, does there exists a $G$-invariant weakly continuous Markov kernel $\theta:Y\times\text{Bor}(\text{Fr}(M))\to [0, 1]$ such that $\pi_*(\theta)=\kappa$. If $M$ admits a $G$-invariant Riemannian metric, then there always exists such a Markov kernel $\theta$. Recall a Riemannian metric $g$ on $M$ is $G$-invariant if $\tilde{\alpha}_s^*g=g$ for all $s\in G$; that is, for each $p\in M$ and $s\in G$, $(\tilde{\alpha}_{s})_{*, p}$ maps orthonormal bases to orthonormal bases.  Note, this implies $\tilde{\gamma}_s(O(M, g))=O(M, g)$ for all $s\in G$.

\begin{theorem}\

\noindent Let $M$ be a smooth $m$-dimensional manifold, and let $Y$ be a locally compact Hausdorff space. Let $G$ be a Lie group acting smoothly on $M$ via $\tilde{\alpha}$ and acting on $Y$ via $\tilde{\beta}$. If there exists a $G$-invariant Riemannian metric on $M$, then for every $G$-invariant weakly continuous Markov kernel $\kappa:Y\times \text{Bor}(M)\to [0, 1]$, there exists a $G$-invariant weakly continuous Markov kernel $\theta:Y\times \text{Bor}(\text{Fr}(M))\to [0, 1]$  with respect to the induced action of $G$ on $\text{Fr}(M)$ such that $\pi_*\theta=\kappa$.

\end{theorem}

\begin{proof}\

\noindent Assume $g$ is a $G$-invariant Riemannian metric on $M$. Let $\rho:O(M, g)\to M$ be the orthonormal frame bundle. Let $\tilde{\gamma}$ denote the induced action of $G$ on $O(M, g)$, and let $\tilde{\zeta}$ denote the canonical $O(m)$ action on $O(M, g)$. We claim $\tilde{\gamma}$ commutes with $\tilde{\zeta}$ on $O(M, g)$. Fix $s\in G$, $A\in O(m)$, and $p\in M$. Write $A=[a^i_j]$. For each $(e_{1, p}, \ldots, e_{m, p})\in O(T_pM, g_p)\subset O(M, g)$, we have
\begin{align*}
    \tilde{\gamma}_s(\tilde{\zeta}_A(e_{1, p}, \ldots, e_{m, p}))&=\tilde{\gamma}_s\left(\sum\limits_{k=1}^m a^k_1 e_{k, p}, \ldots, \sum\limits_{k=1}^m a^k_m e_{k, p}\right)\\
    &=\left((\tilde{\alpha}_s)_{*, p}\left(\sum\limits_{k=1}^m a^k_1 e_{k, p}\right), \ldots, (\tilde{\alpha}_s)_{*, p}\left(\sum\limits_{k=1}^m a^k_m e_{k, p}\right)\right)\\
    &=\left(\sum\limits_{k=1}^m a^k_1 (\tilde{\alpha}_s)_{*, p}(e_{k, p}), \ldots, \sum\limits_{k=1}^m a^k_m (\tilde{\alpha}_s)_{*, p}(e_{k, p})\right)\\
    &=\tilde{\zeta}_A((\tilde{\alpha}_s)_{*, p}(e_{1, p}), \ldots, (\tilde{\alpha}_s)_{*, p}(e_{m, p}))\\
    &=\tilde{\zeta}_A(\tilde{\gamma}_s(e_{1, p}, \ldots, e_{m, p})).
\end{align*}
Thus $\tilde{\gamma}_s\circ \tilde{\zeta}_A=\tilde{\zeta}_A\circ\tilde{\gamma}_s$ for all $s\in G$ and $A\in O(m)$.

For each $p\in M$, pick an element $o_p\in \rho^{-1}(p)$. As $O(m)$ acts transitively and freely on $\rho^{-1}(p)$, we obtain a diffeomorphism
$$R_{o_p}:O(m)\to \pi^{-1}(p)\quad\quad R_{o_p}(A)=\tilde{\zeta}_A(o_p).$$
Let $\lambda$ denote the unique normalized Haar measure on $O(m)$, and let $\sigma_{p}=(R_{o_p})_{*}(\lambda)$ which defines a Radon probability measure on $(\rho^{-1}(p), \text{Bor}(\rho^{-1}(p)))$. As shown in Theorem \ref{Theorem: lifting WCMK}, $\sigma_p$ is independent of $o_p$ as well as the map
$$\sigma:M\times\text{Bor}(\text{Fr}(M))\to [0, 1]\quad\quad \sigma(p, B)=\sigma_p(B\cap \rho^{-1}(p))$$
defines a weakly continuous Markov kernel. 

We claim $\sigma$ is $G$-invariant. Let $p\in M$, then, for all $s\in G$, we have
$$\tilde{\gamma}_s(R_{o_p}(A))=\tilde{\gamma}_s\tilde{\zeta}_A(o_p)=\tilde{\zeta}_A(\tilde{\gamma}_s(o_p))=R_{\tilde{\gamma}_s(o_p)}(A)$$
Thus $\tilde{\gamma}_s\circ R_{o_p}=R_{\tilde{\gamma}_s(o_p)}$ so that
$$(\tilde{\gamma}_s)_*(\sigma_p)=(\tilde{\gamma}_s)_{*}((R_{o_p})_{*}(\lambda))=(\tilde{\gamma}_s\circ R_{o_p})_*(\lambda)=(R_{\tilde{\gamma}_s(o_p)})_{*}(\lambda)=\sigma_{\tilde{\alpha}_s(p)}$$
Therefore, given $p\in M$, $B\in \text{Bor}(\text{Fr}(M))$, and $s\in G$, we have
\begin{align*}
    \sigma(\tilde{\alpha}_s(p), \tilde{\gamma}_s(B))&=\sigma(\tilde{\alpha}_s(p), \tilde{\gamma}_s(B)\cap O(M, g))=\sigma(\tilde{\alpha}_s(p), \tilde{\gamma}_s(B\cap O(M, g)))\\
    &=\sigma(p, B\cap O(M, g))=\sigma(p, B)
\end{align*}
proving $\sigma$ is $G$-invariant.

Now define
$$\theta:Y\times \text{Bor}(\text{Fr}(M))\to [0, 1]\quad\quad \theta(y, B)=\int_{M} \sigma(p, B)\kappa(y, dp)$$
which is a weakly continuous Markov kernel with $\pi_*(\theta)=\kappa$. We claim $\theta$ is $G$-invariant. Let $g\in G$, $y\in Y$, $B\in \text{Bor}(\text{Fr}(M))$, then, using that $\sigma$ and $\kappa$ are $G$-invariant, we have
\begin{align*}
    \theta(\tilde{\beta}_g(y), \tilde{\gamma}_g(B))&=\int_{M}\sigma(p, \tilde{\gamma}_g(B))\;\kappa(\tilde{\beta}_g(y), dp)=\int_{M}\sigma(\tilde{\alpha}_g(p), \tilde{\gamma}_g(B))\;\kappa(y, dp)\\
    &=\int_{M}\sigma(p, B)\;\kappa(y, dp)=\theta(y, B).
\end{align*}

\end{proof}

As before, let $G$ be a Lie group acting smoothly on $M$ via $\tilde{\alpha}$, and let $\tilde{\gamma}$ denote the induced action of $G$ on $\text{Fr}(M)$. Given the data $(Y, \kappa, \theta)$ along with an action of $G$ on $Y$, denoted by $\tilde{\beta}$, such that $\kappa$ and $\theta$ are $G$-invariant weakly continuous Markov kernels, we can construct two classical $G$-equivariant couplings:
$$\mathcal{C}_1=(((C_0(\text{Fr}(M)), G, \gamma), (C_0(Y), G, \beta), ((C_0(Y), \phi_\theta), \beta)), ((C_0(\text{Fr}(M)), G, \gamma), C_0(\text{Fr}(M)), \gamma))$$
and
$$\mathcal{C}_2=(((C_0(M), G, \alpha), (C_0(Y), G, \beta), ((C_0(Y), \phi_\kappa), \beta)), ((C_0(M), G, \alpha),C_0(M), \alpha ))$$
As $\pi_*(\theta)=\kappa$, we ask if there exists a relationship between relativization maps for $\mathcal{C}_1$ and $\mathcal{C}_2$. Using the relationship $\pi_*(\theta)=\kappa$, we obtain a canonical isometric embedding of $\Gamma_\kappa$ into $\Gamma_\theta$. When this isometric embedding has orthogonally complementable range, we can utilize the embedding to build relativization maps for $\mathcal{C}_2$ from relativization maps for $\mathcal{C}_1$.

\begin{lemma}\

\noindent Let $M$ be a smooth $m$-dimensional  manifold, and let $Y$ be a locally compact Hausdorff space. Let $G$ be a Lie group which acts smoothly on $M$ by $\tilde{\alpha}$ and acts on $Y$ by $\tilde{\beta}$.  Denote the induced action from $G$ on $\text{Fr}(M)$ by $\tilde{\gamma}$. Let
$$\mathcal{C}_1=(((C_0(\text{Fr}(M)), G, \gamma), (C_0(Y), G, \beta), ((C_0(Y), \phi_\theta), \beta)), ((C_0(\text{Fr}(M)), G, \gamma), C_0(\text{Fr}(M)), \gamma))$$
and
$$\mathcal{C}_2=(((C_0(M), G, \alpha), (C_0(Y), G, \beta), ((C_0(Y), \phi_\kappa), \beta)), ((C_0(M), G, \alpha),C_0(M), \alpha ))$$
be classical $G$-equivariant couplings. Let 
$$\theta:Y\times\text{Bor}(\text{Fr}(M))\to [0, 1]\quad\quad \kappa:Y\times\text{Bor}(M)\to [0, 1]$$
be the $G$-invariant weakly continuous Markov kernels associated to $\phi_1$ and $\phi_2$, respectively. Let $(\Gamma_{\kappa}, \pi_\kappa, V_\kappa, W_\kappa)$ and $(\Gamma_\theta, \pi_\theta, V_\theta, W_\theta)$ be the quadruples associated to $\kappa$ and $\theta$, respectively. If $\pi_*(\theta)=\kappa$,  then the map 
$$U:\Gamma_\kappa\to \Gamma_\theta\quad\quad U(\pi_\kappa(f)V_\kappa g)=\tilde{\pi}_\theta(\pi^*(f))V_\theta g$$
is a well-defined $C_0(Y)$-linear isometry with the property $W_{\theta, s}U=UW_{\kappa, s}$ for all $s\in G$.
 
\end{lemma}

\begin{proof}\

\noindent We define the map
$$U:\pi_\kappa(C_0(M))V_\kappa C_0(Y)\to \Gamma_\theta\quad\quad U\left(\sum\limits_{k=1}^n \pi_\kappa(f_k)V_\kappa g_k\right)=\sum\limits_{k=1}^n \tilde{\pi}_\theta(\pi^*(f_k))V_\theta g_k$$
As
\begin{align*}
    \left\|U\left(\sum\limits_{k=1}^n \pi_\kappa(f_k)V_\kappa g_k\right)\right\|^2&=\sup\limits_{y\in Y}\sum\limits_{i, j=1}^n\int_{\text{Fr}(M)}\overline{f_i(\pi(b))g_i(y)}f_j(\pi(b))g_j(y)\;\theta(y, db)\\
    &=\sup\limits_{y\in Y}\sum\limits_{i, j=1}^n\int_{M}\overline{f_i(p)g_i(y)}f_j(p)g_j(y)\; d(\pi_*\theta(y, \cdot))(p)\\
    &=\sup\limits_{y\in Y}\sum\limits_{i, j=1}^n\int_{M}\overline{f_i(p)g_i(y)}f_j(p)g_j(y)\; \kappa(y, dp)=\left\|\sum\limits_{k=1}^n \pi_\kappa(f_k)V_\kappa g_k\right\|
\end{align*}
it follows that $U$ is a well-defined map. Furthermore, it is clear $U$ is $C_0(Y)$-linear. Thus $U$ is $C_0(Y)$-linear isometry. As $\pi_\kappa(C_0(M))V_\kappa C_0(Y)$ is dense in $\Gamma_\kappa$, $U$ extends to a $C_0(Y)$-linear isometry on $\Gamma_\kappa$ which we also denote by $U$.

To prove the last claim, fix $s\in G$. As $UW_{\kappa, s}$ and $W_{\theta, s}U$ are both continuous and $\beta_s$-linear, it suffices to prove equality on elements of the form $\pi_\theta(f)V_\theta g\in \Gamma_\theta$. Given such an element, we have
\begin{align*}
    (W_{\theta, s}U)(\pi_\kappa(f)V_\kappa g)&=W_{\theta, s}(\tilde{\pi}_\theta(\pi^*(f))V_\theta g)=\tilde{\pi}_\theta(\gamma_s(\pi^*(f)))V_\theta (\beta_s(g))\\
    &=\tilde{\pi}_\theta(\pi^*(\alpha_s(f)))V_\theta (\beta_s(g))=U(\pi_\kappa(\alpha_s(f))V_\kappa \beta_s(g))=(UW_{\kappa, s})(\pi_\kappa(f)V_\kappa(g))
\end{align*}
Hence $UW_{\theta, s}=W_{\kappa, s}U$ for all $s\in G$.

\end{proof}

\begin{proposition}\

\noindent Let $M$ be a smooth $m$-dimensional  manifold, and let $Y$ be a locally compact Hausdorff space. Let $G$ be a Lie group which acts smoothly on $M$ by $\tilde{\alpha}$ and acts on $Y$ be $\tilde{\beta}$.  Denote the induced action from $G$ on $\text{Fr}(M)$ by $\tilde{\gamma}$. Let
$$\mathcal{C}_1=(((C_0(\text{Fr}(M)), G, \gamma), (C_0(Y), G, \beta), ((C_0(Y), \phi_\theta), \beta)), ((C_0(\text{Fr}(M)), G, \gamma), C_0(\text{Fr}(M)), \gamma))$$
and
$$\mathcal{C}_2=(((C_0(M), G, \alpha), (C_0(Y), G, \beta), ((C_0(Y), \phi_\kappa), \beta)), ((C_0(M), G, \alpha),C_0(M), \alpha ))$$
be classical $G$-equivariant couplings. Let 
$$\theta:Y\times\text{Bor}(\text{Fr}(M))\to [0, 1]\quad\quad \kappa:Y\times\text{Bor}(M)\to [0, 1]$$
be the $G$-invariant weakly continuous Markov kernels associated to $\phi_1$ and $\phi_2$, respectively. Let $(\Gamma_{\kappa}, \pi_\kappa, V_\kappa, W_\kappa)$ and $(\Gamma_\theta, \pi_\theta, V_\theta, W_\theta)$ be the quadruples associated to $\kappa$ and $\theta$, respectively. If $\pi_*(\theta)=\kappa$ and the induced $C_0(Y)$-linear isometry 
$$U:\Gamma_\kappa\to \Gamma_\theta\quad\quad U(\pi_\kappa(f)V_\kappa g)=\tilde{\pi}_\theta(\pi^*(f))V_\theta g$$
has orthogonally complementable range, then whenever $\pounds$ is a relativization map for $\mathcal{C}_1$,  the map
$$\pounds':C_b(M)\to \mathcal{L}(\Gamma_{\kappa})^G\quad\quad \pounds'(f)=U^*\pounds(\tilde{\pi}^*(f))U$$
defines a relativization map for $\mathcal{C}_2$.

\end{proposition}

\begin{proof}\

\noindent  As $U$ is an isometry with orthogonally complementable range, it follows $U$ is adjointable with $U^*U=1_{\Gamma_\kappa}$. Note, that $U^*$ is the composition of the orthogonal projection map $\Gamma_\theta\to U(\Gamma_\kappa)$ followed by the inverse of $U$ on $U(\Gamma_\kappa)$.

By the prior lemma, we have that for all $s\in G$ that $W_{\theta, s}U=UW_{\kappa, s}$ for all $s\in G$. Thus, as $U$ is adjointable, we have $W_{\kappa, s}U^*=U^*W_{\theta, s}$ for all $s\in G$. Therefore the induced map
$$C:\mathcal{L}(\Gamma_\theta)\to \mathcal{L}(\Gamma_\kappa)\quad\quad C(T)=U^*TU$$
defines a $G$-equivariant unital $*$-algebra homomorphism. As $C$ is unital, it follows $C$ is automatically non-degenerate and thus strictly continuous on the unit ball. 

Now, consider the following diagram:
\begin{center}
   \begin{tikzcd}
	{C_b(M)} && {C_b(Fr(M))} \\
	\\
	{C_b(M)^G} && {C_b(Fr(M))^G} \\
	\\
	{\mathcal{L}(\Gamma_\kappa)^G} && {\mathcal{L}(\Gamma_\theta)^G}
	\arrow["{\tilde{\pi}^*}"{description}, from=1-1, to=1-3]
	\arrow[hook', from=3-1, to=1-1]
	\arrow["{\tilde{\pi}^*}"{description}, from=3-1, to=3-3]
	\arrow["{\tilde{\pi}_\kappa}"{description}, from=3-1, to=5-1]
	\arrow[hook, from=3-3, to=1-3]
	\arrow["{\tilde{\pi}_\theta}"{description}, from=3-3, to=5-3]
	\arrow["C"{description}, from=5-3, to=5-1]
\end{tikzcd}
\end{center}
We claim both squares commute. As $\pi$ is $G$-equivariant, it follows the top square commutes. To see that the bottom square commutes, fix $f\in C_b(M)^G$. As $\tilde{\pi}_\kappa(f)$ and $U^*\tilde{\pi}_\theta(\tilde{\pi}^*(f))U$ are continuous and linear, it suffices to show the two maps agree on elements of the form $\pi_\kappa(g)V_\kappa(h)\in \Gamma_{\kappa}$. Given such an element, we have
\begin{align*}
    (U^*\tilde{\pi}_\theta(\tilde{\pi}^*(f))U)(\pi_\kappa(g)V_\kappa(h))&=U^*\tilde{\pi}_\theta(\tilde{\pi}^*f)\tilde{\pi}_\theta(\tilde{\pi}^*g)V_\theta h=U^*(\tilde{\pi}_\theta(\tilde{\pi}^*(fg))V_\theta h)\\
    &=\pi_\kappa(fg)V_\kappa h=\tilde{\pi}_\kappa(f)(\pi_\kappa(f)V_\kappa h)
\end{align*}
Therefore $\tilde{\pi}_\kappa(f)=\pi_\kappa(g)V_\kappa(h)$ proving the bottom square commutes.

Now suppose we have a relativization map $\pounds$ for $\mathcal{C}_1$. Define
$$\pounds':C_b(M)\to \mathcal{L}(\Gamma_\kappa)^G\quad\quad \pounds'(f)=V\mathcal{L}(\tilde{\pi}^*(f))U.$$
As $\pounds'$ is the composition of completely positive maps which are strictly continuous the unit ball, it follows $\pounds'$ is a completely positive map which is strictly continuous on the unit ball. Now consider the diagram
\begin{center}
    \begin{tikzcd}
	{C_b(M)} && {C_b(\text{Fr}(M))} \\
	\\
	{C_b(M)^G} && {C_b(\text{Fr}(M))^G} \\
	\\
	{\mathcal{L}(\Gamma_\kappa)^G} && {\mathcal{L}(\Gamma_\theta)^G}
	\arrow["{\tilde{\pi}^*}"{description}, from=1-1, to=1-3]
	\arrow["{\pounds'}"', curve={height=30pt}, from=1-1, to=5-1]
	\arrow["\pounds", shift left=4, curve={height=-30pt}, from=1-3, to=5-3]
	\arrow[hook', from=3-1, to=1-1]
	\arrow["{\tilde{\pi}^*}"{description}, from=3-1, to=3-3]
	\arrow["{\tilde{\pi}_\kappa}"{description}, from=3-1, to=5-1]
	\arrow[hook, from=3-3, to=1-3]
	\arrow["{\tilde{\pi}_\theta}"{description}, from=3-3, to=5-3]
	\arrow["C"{description}, from=5-3, to=5-1]
\end{tikzcd}
\end{center}
From the prior paragraph the two middle squares commute. By assumption, the far right diagram commutes. Using commutativity of these three diagrams, it follows the far left diagram commutes. Hence $\pounds'$ defines a relativization map for $\mathcal{C}_2$.

\end{proof}

\begin{lemma}\

\noindent  Let $M$ be a smooth $m$-dimensional  manifold, and let $Y$ be a locally compact Hausdorff space. Let 
$$\theta:Y\times\text{Bor}(\text{Fr}(M))\to [0, 1]\quad\quad \kappa:Y\times\text{Bor}(M)\to [0, 1]$$
be weakly continuous Markov kernels such that $\pi_*(\theta)=\kappa$. If there exists a weakly continuous Markov 
$$\lambda:Y\times M\times\text{Bor}(\text{Fr}(M))\to [0, 1]$$
such that for all $y\in Y$, $p\in M$, $B\in\text{Bor}(\text{Fr}(M))$
$$\theta(y, B)=\int_{M} \lambda(y, p, B)\;\kappa(y, dp),$$
then for all $f\in C_0(Y\times M)$, $g\in C_0(Y\times \text{Fr}(M))$, $y\in Y$, and $p\in M$
$$\int_{\text{Fr}(M)}g(y, b)f(y, \pi(b))\lambda(y, p, db)=f(y, p)\int_{\text{Fr}(M)}g(y, b)\lambda(y, p, db).$$

\end{lemma}

\begin{proof}\

\noindent As $\pi_*\theta=\kappa$, it follows that for all $B\in\text{Bor}(M)$ and $y\in Y$
$$\theta(y, \pi^{-1}(B))=\kappa(y, B)=\int_{M}\chi_B(p)\;\kappa(y, dp).$$
Due to our assumption on $\lambda$, we have
$$\int_{M}\lambda(y, p, \pi^{-1}(B))\;\kappa(y, dp)=\int_M\chi_B(p)\;\kappa(y, dp)$$
so that $\lambda(y, \cdot, \pi^{-1}(B))=\chi_B$  almost everywhere (with respect to $\kappa(y, \cdot)$). Therefore, taking $B=\{p\}$ for $p\in M$, we have
$$\lambda(y, p, \pi^{-1}(p))=\chi_{\{p\}}(p)=1$$
As $\lambda(y, p, \cdot)$ is a probability measure, the equality implies $\text{supp}(\lambda(y, p, \cdot))\subset \pi^{-1}(p)$. Therefore
\begin{align*}
    \int_{\text{Fr}(M)}g(y, b)f(y, \pi(b))\lambda(y, p, db)&=\int_{\pi^{-1}(p)}g(y, b)f(y, \pi(b))\lambda(y, p, db)\\
    &=\int_{\pi^{-1}(p)}g(y, b)f(y, p)\lambda(y, p, db)\\
    &=f(y, p)\int_{\pi^{-1}(p)}g(y, b)\;\lambda(y, p, db)\\
    &=f(y, p)\int_{\text{Fr}(M)}g(y, b)\;\lambda(y, p, db)
\end{align*}

\end{proof}

\begin{theorem}\

\noindent  Let $M$ be a smooth $m$-dimensional  manifold, and let $Y$ be a locally compact Hausdorff space. Let $G$ be a Lie group which acts smoothly on $M$ by $\tilde{\alpha}$ and acts on $Y$ be $\tilde{\beta}$.  Denote the induced action from $G$ on $\text{Fr}(M)$ by $\tilde{\gamma}$. Let
$$\mathcal{C}_1=(((C_0(\text{Fr}(M)), G, \gamma), (C_0(Y), G, \beta), ((C_0(Y), \phi_\theta), \beta)), ((C_0(\text{Fr}(M)), G, \gamma), C_0(\text{Fr}(M)), \gamma))$$
and
$$\mathcal{C}_2=(((C_0(M), G, \alpha), (C_0(Y), G, \beta), ((C_0(Y), \phi_\kappa), \beta)), ((C_0(M), G, \alpha),C_0(M), \alpha ))$$
be classical $G$-equivariant couplings. Let 
$$\theta:Y\times\text{Bor}(\text{Fr}(M))\to [0, 1]\quad\quad \kappa:Y\times\text{Bor}(M)\to [0, 1]$$
be the $G$-invariant weakly continuous Markov kernels associated to $\phi_1$ and $\phi_2$, respectively. Let $(\Gamma_{\kappa}, \pi_\kappa, V_\kappa, W_\kappa)$ and $(\Gamma_\theta, \pi_\theta, V_\theta, W_\theta)$ be the quadruples associated to $\kappa$ and $\theta$, respectively. If $\pi_*(\theta)=\kappa$ and if there exists a weakly continuous Markov 
$$\lambda:Y\times M\times\text{Bor}(\text{Fr}(M))\to [0, 1]$$
such that for all $y\in Y$, $p\in M$, $B\in\text{Bor}(\text{Fr}(M))$
$$\theta(y, B)=\int_{M} \lambda(y, p, B)\;\kappa(y, dp),$$
then the induced $C_0(Y)$-linear isometry 
$$U:\Gamma_\kappa\to \Gamma_\theta\quad\quad U(\pi_\kappa(f)V_\kappa g)=\tilde{\pi}_\theta(\pi^*(f))V_\theta g$$
has orthogonally complementable range. 
   
\end{theorem}

\begin{proof}\

\noindent Define the map
$$V:C_0(Y\times \text{Fr}(M))\to C_0(Y\times M)\quad\quad V(f)(y, p)=\int_{\text{Fr}(M)} f(y, b)\;\lambda(y, p, db)$$
As $f(y, \cdot)$ is a continuous function which is bounded and $\lambda(y, p, \cdot)$ is a probability measure, it follows $V(f)$ is a well-defined function. Applying the same argument utilized in proving the map $\tilde{\kappa}(f)$ was continuous in Lemma \ref{Lemma: building Gamma_0}, it follows $V(f)$ is continuous. Similarly, applying the same argument showing $\tilde{\kappa}(f)$ vanishes at infinity in Lemma \ref{Lemma: building Gamma_0}, it follows $V(f)$ vanishes at infinity. Thus $V$ is a well-defined map. Furthermore, it is clear that $V$ is $C_0(Y)$-linear.

We now claim $V$ is bounded with respect to the semi-norms on $C_0(Y\times \text{Fr}(M))$ and $C_0(Y\times M)$. Indeed, we have
\begin{align*}
    ||V(f)||^2&=\sup\limits_{y\in Y}\int_{M}\overline{V(f)(y, p)}V(f)(y, p)\kappa(y, dp)\\
    &=\sup\limits_{y\in Y}\int_{M}\left|\int_{\text{Fr}(M)}f(y, b)\;\lambda(y, p, db)\right|^2\kappa(y, dp)\\
    &\leq \sup\limits_{y\in Y}\int_{M}\int_{\text{Fr}(M)}|f(y, b)|^2\;\lambda(y, p, db)\kappa(y, dp)\\
    &= \sup\limits_{y\in Y}\int_{\text{Fr}(M)}\int_{M}|f(y, b)|^2\;\lambda(y, p, db)\kappa(y, dp)\\
    &=\sup\limits_{y\in Y}\int_{\text{Fr}(M)}|f(y, b)|^2\;\theta(y, db)=||f||^2
\end{align*}
Therefore $||V(f)||\leq ||f||$. Hence the map descends to a well-defined $C_0(Y)$-linear bounded map $V:\Gamma_\theta\to \Gamma_\kappa$.

We claim $V$ is the adjoint of $U$. As both maps are continuous and $C_0(Y)$-linear, it suffices to prove the claim on the dense subspaces $C_0(Y\times M)/N_\kappa$ and $C_0(Y\times \text{Fr}(M))/N_\theta$. Given $[f]\in C_0(Y\times M)/N_\kappa$ and $[g]\in C_0(Y\times \text{Fr}(M))/N_\kappa$, we have for all $y\in Y$ that
\begin{align*}
    \langle V([g]), [f]\rangle (y)&=\int_{M}\overline{V(g)(y, p)} f(y, p)\;\kappa(y, dp)\\
    &=\int_{M}\left(\int_{\text{Fr}(M)}\overline{g(y, b)}\;\lambda(y, p,  db)\right)f(y, p)\;\kappa(y, dp)\\
    &=\int_{M}\int_{\text{Fr}(M)}\overline{g(y, b)}f(y, \pi(b))\lambda(y, p, db)\kappa(y, dp)\\
    &=\int_{\text{Fr}(M)}\int_{M}\overline{g(y, b)}f(y, \pi(b))\lambda(y, p, db)\kappa(y, dp)\\
    &=\int_{\text{Fr}(M)}\overline{g(y, b)}f(y, \pi(b))\theta(y, db)=\langle [g], U([f])\rangle(y)
\end{align*}
Thus $V$ is the adjoint of $U$.

Now we claim $VU=1_{\Gamma_\kappa}$. To prove the claim, it suffices show the equality holds on the dense submodule $C_0(Y\times M)/N_\kappa$. As $VU$ is defined on $C_0(Y\times M)$, it suffices to show $VU$ is the identity map on $C_0(Y\times M)$. Given $f\in C_0(Y\times M)$, we have
\begin{align*}
    (VUf)(y, p)&=\int_{\text{Fr}(M)}U(f)(y, b)\;\lambda(y, p, db)=\int_{\text{Fr}(M)}f(y, \pi(b))\;\lambda(y, p, db)\\
    &=f(y, p)\int_{\text{Fr}(M)}\lambda(y, p, db)=f(y, p)
\end{align*}
Hence $VU=1_{\Gamma_\kappa}$ as claimed. As $U$ is an adjointable isometry with $U^*U=1_{\Gamma_\kappa}$, it follows that the range of $U$ is orthogonally complementable.

\end{proof}

\section{Hilbert Module Quantum Reference Frames}

In this section we introduce our model of quantum reference frames using Hilbert modules. We begin by generalizing classical $G$-equivariant couplings and the construction of the relational joint system. Using the generalized relational joint system, we generalize the notion of a relativization maps to arrive at our definition of a Hilbert module quantum reference frame. We then consider how our framework models quantum reference frames for Hilbert spaces and $C^*$-algebras.

\subsection{Generalizing classical Hilbert module reference frames}

In our setup for Hilbert module classical reference frames, we start with a classical $G$-equivariant coupling for the Hilbert $C_0(X)$-module $C_0(X)$ given by the data
$$(((C_0(X), G, \alpha), (C_0(Y), G, \beta), ((C_0(Y), \phi), \beta)), ((C_0(X), G, \alpha), C_0(X), \alpha)).$$
Generalizing the $C^*$-algebras and Hilbert modules to be arbitrary, we obtain the following generalization of a classical $G$-equivariant coupling.

\begin{definition}\

	\noindent Let $E_A$ be a Hilbert $A$-module, and let $G$ be a group. Define a $G$-equivariant coupling for $E_A$ as the data
	$$(((A, G, \alpha), (B, G, \beta), ((E_B, \phi), W)), ((A, G, \alpha), E_A, U))$$
	where
	\begin{itemize}
		\item $(A, G, \alpha)$ is a $C^*$-dynamical system.
		\item $(B, G, \beta)$ is a $C^*$-dynamical system.
		\item $((E_B, \phi), W)$ is a non-degenerate positive equivariant $C^*$-correspondence from $(A, G, \alpha)$ to $(B, G, \beta)$.
		\item $((A, G, \alpha), E_A, U)$ is a Hilbert module dynamical system.
	\end{itemize}

\end{definition}

\begin{remark}\

\noindent As the definition of a $G$-equivariant coupling comes from generalizing the definition of a classical $G$-equivariant coupling, we interpret the equivariant correspondence $((E_B, \phi), W)$ as a generalized $G$-invariant conditioning of $B$ on $A$.

\end{remark}

In the classical setting, we constructed the relational joint space and relational joint system by using a $G$-invariant weakly continuous Markov kernel. From Theorem \ref{Theorem: KSGNS data for MK}, the relational joint space and relational joint system can be unitarily equivalently constructed by applying the equivariant KSGNS construction to the non-degenerate positive equivariant $C^*$-correspondence determined by the Markov kernel. Therefore, given a $G$-equivariant coupling
$$(((A, G, \alpha), (B, G, \beta), ((E_B, \phi), W)), ((A, G, \alpha), E_A, U))$$
for a Hilbert $A$-module $E_A$, we apply the equivariant KSGNS construction to $((E_B, \phi), W)$ to obtain a unitarily unique quadruple $(F_\phi, \pi_\phi, V_\phi, \widetilde{W})$. 

From the classical setting, we would be led to define the relational joint space as $F_\phi$ and define the relational joint system as $\mathcal{L}(F_\phi)$; however, this is not quite right as we haven't include the original Hilbert module $E_A$. Instead, we first need to apply the equivariant interior tensor product with respect to the equivariant $C^*$-correspondence $((F_\phi, \pi_\phi), \widetilde{W})$ to convert the Hilbert module dynamical system $((A, G, \alpha), E_A, U)$ to the Hilbert module dynamical system
$$((B, G, \beta), E_A\otimes_{\pi_\phi} F_\phi, U\otimes_{\pi_\phi}\widetilde{W}).$$
From this Hilbert module dynamical system, we take the relational joint space as $E_A\otimes_{\pi_\phi}F_\phi$, and we take the relational joint system as $\mathcal{L}(E_A\otimes_{\pi_\phi}F_\phi)$. Using the action $\text{Ad}(U\otimes_{\pi_\phi}\widetilde{W})$ on $\mathcal{L}(E_A\otimes_{\pi_\phi}F_\phi)$, we identify the invariant relational observables as those contained in the subalgebra $\mathcal{L}(E_A\otimes_{\pi_\phi}F_\phi)^G$. 

We note that while it may appear that this construction is different than the classical construction, the two are in fact the same. To see this, let us fix a classical $G$-equivariant coupling for $C_0(X)$:
$$(((C_0(X), G, \alpha), (C_0(Y), G, \beta), ((C_0(Y), \phi), \beta)), ((C_0(X), G, \alpha), C_0(X), \alpha)).$$
Applying the construction described above, we obtain the Hilbert module dynamical system
$$((C_0(Y), G, \beta), C_0(X)\otimes_{\pi_\phi} F_\phi, \alpha\otimes_{\pi_\phi} W)$$
As we are viewing $C_0(X)$ as a Hilbert module over itself and $\pi_\phi$ is a non-degenerate $*$-algebra map, the assignment
$$C_0(X)\otimes_{\pi_\phi} F_\phi\to F_\phi \quad\quad f\dot{\otimes}y\mapsto \pi_\phi(f)y$$
defines a unitary map of the Hilbert $C_0(Y)$-modules. Under this identification, the resulting Hilbert module dynamical system on $F_\phi$ is given by
$$((C_0(Y), G, \beta),  F_\phi,  W)$$
so that the relational joint space is given by $F_\phi$ and the relational joint system is given by $\mathcal{L}(F_\phi)$. Therefore, when applying the construction above to classical $G$-equivariant couplings, we canonically recover the description of the relational joint space and relational joint system as described in Section 3.

\begin{definition}\

\noindent Let $E_A$ be a Hilbert $A$-module, let $G$ be a group, and let 
	$$(((A, G, \alpha), (B, G, \beta), ((E_B, \phi), W)), ((A, G, \alpha), E_A, U))$$
	be a $G$-equivariant coupling for $E_A$. Let $(F_\phi, \pi_\phi, V_\phi, \tilde{W})$ be the unitarily unique quadruple obtain from applying Theorem \ref{Thrm: Equiv Dil} to $((E_B, \phi), W)$. 
    \begin{itemize}
        \item Define the relational joint space for the $G$-equivariant coupling as $E_A\otimes_{\pi_\phi}F_\phi$.
        \item Define the relational joint system for the $G$-equivariant coupling as $\mathcal{L}(E_A\otimes_{\pi_\phi}F_\phi)$. 
        \item Define the reference frame for the $G$-equivariant coupling as $\mathcal{L}(E_A\otimes_{\pi_\phi}F_\phi)^{G}$.
    \end{itemize}

\end{definition}

Now we generalize the data of a relativization map. In the classical setting,  relativization maps were used to fix the image of the unital $G$-equivariant $*$-algebra map
$$\widetilde{\pi}_\phi:\mathcal{L}(C_0(X))\to \mathcal{L}(F_\phi).$$
Under the identification of $C_0(X)\otimes_{\pi_\phi}F_\phi$ and $F_\phi$, the map $\widetilde{\pi}_\phi$ corresponds to the induced unital $G$-equivariant $*$-algebra map
$$\mathcal{L}(C_0(X))\to \mathcal{L}(C_0(X)\otimes_{\pi_\phi} F_\phi)\quad\quad T\mapsto T\otimes_{\pi_\phi}I$$
from the equivariant interior tensor product. Using this induced map in place of $\tilde{\pi}_\phi$, we obtain the following definition of relativization map for $G$-equivariant couplings.

\begin{definition}\

\noindent Let $E_A$ be a Hilbert $A$-module, let $G$ be a group, and let 
	$$(((A, G, \alpha), (B, G, \beta), ((E_B, \phi), W)), ((A, G, \alpha), E_A, U))$$
	be a $G$-equivariant coupling for $E_A$. Let $(F_\phi, \pi_\phi, V_\phi, \tilde{W})$ be the unitarily unique quadruple obtain from applying Theorem \ref{Thrm: Equiv Dil} to $((E_B, \phi), W)$. Define a relativization map for the equivariant coupling as the data of a function
	$$\pounds:\mathcal{L}(E_A)\to \mathcal{L}(E_A\otimes_{\pi_\phi}F_\phi)^{G}$$
    which satisfies the following conditions:
    \begin{enumerate}
        \item $\pounds$ is completely positive.
        \item $\pounds$ is strictly continuous on the unit ball in $\mathcal{L}(E_A)$.
        \item $\pounds$ makes the following diagram commute:
		\begin{center}
		\begin{tikzcd}
			{\mathcal{L}(E_A)} &&& \\
			\\
			{\mathcal{L}(E_A)^{G}} &&& {\mathcal{L}(E_A\otimes_{\pi_\phi} F_\phi)^{G}}
			\arrow["\pounds"{description}, dashed, from=1-1, to=3-4]
			\arrow[{description}, hook', from=3-1, to=1-1]
			\arrow["{T\mapsto T\otimes_{\pi_\phi}I}"{description}, from=3-1, to=3-4]
		\end{tikzcd}
		\end{center}
        \end{enumerate}
    
	\end{definition}

Similar to the relativization maps for classical $G$-equivariant couplings, relativization maps for $G$-equivariant couplings enjoy a few properties. First, if $\pounds$ is a relativization map, then, as $I\in \mathcal{L}(E_A)^{G}$ and the bottom map in the diagram is a unital $*$-algebra homomorphism, it follows $\pounds$ is a unital completely positive map. Therefore, the pullback $\pounds^*$ yields a well-defined map on states. Second, this formulation of the relativization map frames the existence and possible uniqueness of $\pounds$ as an extension problem. Third, as $\pounds$ is a unital completely positive map which is strictly continuous on the unit ball in $\mathcal{L}(E_A)$, the restriction $\pounds|_{\mathcal{K}(E_A)}$ defines a non-degenerate completely positive map. Finally, as the next example shows, it is not true that every $G$-equivariant coupling admits a relativization map.
    
\begin{example}\label{Ex: non-existence example}\

\noindent The following example demonstrates the possibility for there to not exist a relativization map for an equivariant coupling. Let $\mathbb{F}_2$ be the free group on two generators equipped with the discrete topology. Take $A=\mathbb{C}$, $E_A=\ell^2(\mathbb{F}_2)$, and $G=\mathbb{F}_2$. Let $\lambda:\mathbb{F}_2\to \mathcal{B}(\ell^2(\mathbb{F}_2))$ be the left regular representation of $\mathbb{F}_2$, and let $\text{trv}$ denote the trivial action of $\mathbb{F}_2$ on $\mathbb{C}$. Then
	$$(((\mathbb{C}, \mathbb{F}_2, \text{trv}), (\mathbb{C}, \mathbb{F}_2, \text{trv}), ((\mathbb{C}, 1_{\mathbb{C}}), \text{trv})), ((\mathbb{C}, \mathbb{F}_2, \text{trv}), \ell^2(\mathbb{F}_2), \lambda)) $$
	defines an equivariant coupling. As $1_{\mathbb{C}}$ is a unital $*$-algebra homomorphism, we can take $(\mathbb{C}, 1_{\mathbb{C}}, 1_{\mathbb{C}}, \text{trv})$ as the quadruple from the equivariant KSGNS construction. Applying the interior tensor product, we obtain the Hilbert space
	$$\ell^2(\mathbb{F}_2)\otimes_{1_\mathbb{C}}\mathbb{C}=\ell^2(\mathbb{F}_2)\hat{\otimes}\mathbb{C}\cong \ell^2(\mathbb{F}_2),$$
	so that the reference frame system is given by $\mathcal{B}(\ell^2(\mathbb{F}_2))^{\text{Ad}(\lambda)}=\lambda(\mathbb{F}_2)'$. Therefore, the data of a relativization map for the equivariant coupling is equivalent to a conditional expectation $\pounds:\mathcal{B}(\ell^2(\mathbb{F}_2))\to\lambda(\mathbb{F}_2)'$. As $\mathbb{F}_2$ is not amenable (see example 4.4.5 in \cite{takesaki2003theory3}), $\lambda(\mathbb{F}_2)'$ is not  injective as a von Neumann algebra (Proposition 6.4 and Corollary 7.2 in \cite{Connes1976}). Thus, there does not exists a conditional expectation $\pounds$ (Corrollary 1.9 in \cite{eleftherakis2023similarity}) which shows there does not exist a relativization map for the equivariant coupling.
	
\end{example}

Combining together the data of a $G$-equivariant coupling and relativization map, we obtain a definition for Hilbert module quantum reference frame.
    
\begin{definition}\

		\noindent Let $E_A$ be a Hilbert $A$-module. Define a Hilbert module quantum reference frame for $E_A$ as the data
		$$((((A, G, \alpha), (B, G, \beta), ((E_B, \phi), W)), ((A, G, \alpha), E_A, U)), \pounds)$$
		where
		\begin{itemize}
			\item $(((A, G, \alpha), (B, G, \beta), ((E_B, \phi), W)), ((A, G, \alpha), E_A, U))$ is a $G$-equivariant coupling for $E_A$.
			\item $\pounds$ is a relativization map for the $G$-equivariant coupling.
			
		\end{itemize}
		
	\end{definition}

    \subsection{Hilbert module quantum reference frames and Hilbert spaces}

 The simplest possible $G$-equivariant couplings to consider are when $A=\mathbb{C}=B$ so that the Hilbert modules over $A$ and $B$ become Hilbert spaces. When applying the equivariant KSGNS and interior tensor product in this setting, we canonically recover the standard joint space and joint system of Hilbert spaces. Using this identification, we are able to utilize the operational Hilbert space formulation of quantum reference frames to generate examples of Hilbert module quantum reference frames.

\begin{proposition}\label{Prop: Hilbert space case}\

\noindent Let $H$ be a Hilbert space. Every $G$-equivariant coupling for $H$ of the form
$$((\mathbb{C}, G, \alpha), (\mathbb{C}, G, \beta), ((K, \phi), W), ((\mathbb{C}, G, \alpha), H, U))$$
determines and is determined by the data $(G, (K, W), (H, U))$ where $G$ is a group and both $(K, W)$ and $(H, U)$ are unitary representations of $G$ on $K$ and $H$, respectively. Furthermore, given the data $(G, (K, W), (H, U))$, 
\begin{enumerate}
    \item the relational joint space  for the $G$-equivariant coupling associated to $(G, (K, W), (H, U))$  is canonically unitarily isomorphic to $H\hat{\otimes}K$,
    \item the relational joint system  for the $G$-equivariant coupling associated to $(G, (K, W), (H, U))$  is canonically isomorphic to $\mathcal{B}(H\hat{\otimes} K)$, 
    \item the reference frame system for the $G$-equivariant coupling associated to $(G, (K, W), (H, U))$   is canonically isomorphic to $\mathcal{B}(H\hat{\otimes}K)^{\text{Ad}(U\hat{\otimes} W)}$,
    \item and a relativization map for the $G$-equivariant coupling associated to $(G, (K, W), (H, U))$ determines and is determined by a completely positive map $\pounds:\mathcal{B}(H)\to \mathcal{B}(H\hat{\otimes}K)^{\text{Ad}(U\hat{\otimes}W)}$ which is strictly continuous on the unit ball in $\mathcal{B}(H)$ and makes the following diagram commute:
\begin{center}
    \begin{tikzcd}
	{\mathcal{B}(H)} &&& \\
	\\
	{\mathcal{B}(H)^{\text{Ad}(U)}} &&& {\mathcal{B}(H\hat{\otimes}K)^{\text{Ad}(U\hat{\otimes}W)}}
	\arrow["\pounds"{description}, dashed, from=1-1, to=3-4]
	\arrow[hook', from=3-1, to=1-1]
	\arrow["{T\mapsto T\hat{\otimes}I}"{description}, from=3-1, to=3-4]
\end{tikzcd}
\end{center}
\end{enumerate}
\end{proposition}

\begin{proof}\

\noindent First, suppose we have $G$-equivariant coupling
$$((\mathbb{C}, G, \alpha), (\mathbb{C}, G, \beta), ((K, \phi), W), ((\mathbb{C}, G, \alpha), H, U)).$$
As for each $g\in G$, $\alpha_g=1_{\mathbb{C}}=\beta_g$, the maps $W_g$ and $U_g$ are unitary maps on $K$ and $H$, respectively, for all $g\in G$. Thus $(K, W)$ and $(H, U)$ are unitary representations. Hence, we obtain the data $(G, (K, W), (H, U))$.

Now suppose we have the data $(G, (K, W), (H, U))$. Let $\alpha:G\to \text{Aut}_{*\text{-alg}}(\mathbb{C})$ be the unique group homomorphism, then for each $g\in G$, $W_g$ and $U_g$ are both $\alpha_g=1_{\mathbb{C}}$ adjointable unitaries. Therefore $((\mathbb{C}, G, \alpha), H, U)$ defines a Hilbert module dynamical system on $H$. Let 
$$\phi:\mathbb{C}\to \mathcal{B}(K)\quad\quad \phi(\lambda)=\lambda I$$
which is the unique non-degenerate completely positive map. As 
$$\phi(\alpha_g(\lambda))\circ W_g=\lambda W_g=W_g\circ\phi(\lambda)$$
for all $g\in G$ and $\lambda\in\mathbb{C}$, the data $((K, \phi), W)$ forms a non-degenerate positive equivariant $C^*$-correspondence from $(\mathbb{C}, G, \alpha)$ to $(\mathbb{C}, G, \alpha)$. Thus
$$((\mathbb{C}, G, \alpha), (\mathbb{C}, G, \alpha), ((K, \phi), W), ((\mathbb{C}, G, \alpha), H, U))$$
defines a $G$-equivariant coupling for $H$. It is clear these two constructions are inverse of each other.

Now suppose we have the data $(G, (K, W), (H, U))$, and let $(F_\phi, \pi_\phi, V_\phi, \tilde{W})$ be the unitarily unique quadruple obtained from applying the equivariant KSGNS construction to the non-degenerate positive $C^*$-correspondence $((K, \phi), W)$. As 
$$\phi:\mathbb{C}\to \mathcal{B}(K)\quad\quad \phi(\lambda)=\lambda I$$
is a unital $*$-algebra homomorphism, the data quadruple $(K, \phi, I, W)$ satisfies the conditions for the quadruple from the equivariant KSGNS construction for the correspondence $((K, \phi), W)$. Therefore, from the unitary uniqueness in the equivariant KSGNS construction, there exists a canonical unitary $V:F_\phi\to K$ intertwining $\phi$ and $\pi_\phi$ as well as $W$ and $\tilde{W}$. Applying the interior tensor product of $H$ and $K$ with respect to $\phi$ yields the Hilbert space 
$$H\otimes_\phi K=H\hat{\otimes}K$$
Thus, the map 
$$I\otimes V:H\otimes_{\pi_\phi}F_\phi\to H\hat{\otimes} K\quad\quad (I\otimes V)(x\dot{\otimes}y)=x\otimes V(y)$$
defines a unitary isomorphism of the Hilbert spaces. Furthermore, the intertwining property of $\phi$ for the representations implies
$$(U_g\hat{\otimes} W_g)\circ (I\otimes V)=(I\otimes V)\circ (U_g\hat{\otimes}\tilde{W}_g)$$
for all $g\in G$. Thus we obtain a $*$-algebra isomorphisms
$$\mathcal{B}(H\otimes_{\pi_\phi}F_\phi)\to\mathcal{B}(H\hat{\otimes}K)\quad\quad T\mapsto (I\otimes V)T(I\otimes V^*)$$
for which preserves the fixed points. Under this isomorphism, we have that for all $T\in \mathcal{B}(H)$,
$$(I\otimes V)(T\otimes_{\pi_\phi}I)(I\otimes V^*)=T\hat{\otimes}I$$
Therefore we have established $(1)-(3)$ from the canonical unitary isomorphism $V:F_\phi\to K$.

Now we prove $(4)$. Suppose we have a completely positive map $\pounds:\mathcal{B}(H)\to \mathcal{B}(H\hat{\otimes}K)^{\text{Ad}(U\hat{\otimes}W)}$ which is strictly continuous on the unit ball in $\mathcal{B}(H)$ and makes the following diagram commute:
\begin{center}
    \begin{tikzcd}
	{\mathcal{B}(H)} &&& \\
	\\
	{\mathcal{B}(H)^{\text{Ad}(U)}} &&& {\mathcal{B}(H\hat{\otimes}K)^{\text{Ad}(U\hat{\otimes}W)}}
	\arrow["\pounds"{description}, dashed, from=1-1, to=3-4]
	\arrow[hook', from=3-1, to=1-1]
	\arrow["{T\mapsto T\hat{\otimes}I}"{description}, from=3-1, to=3-4]
\end{tikzcd}
\end{center}
Define
$$\pounds':\mathcal{B}(H)\to \mathcal{B}(H\otimes_{\pi_\phi} F_\phi)^{\text{Ad}(U\otimes_{\pi_\phi} \tilde{W})}\quad\quad \pounds'(T)=(I\otimes V^*)\pounds(T)(I\otimes V)$$
As $\pounds'$ is the composition of a strictly continuous completely positive map and $*$-algebra isomorphism, $\pounds'$ is also completely positive and strictly continuous on the unit ball. Furthermore, given $T\in \mathcal{B}(H)^{\text{Ad}(U)}$, we have
\begin{align*}
    \pounds'(T)=(I\otimes V^*)\pounds(T)(I\otimes V)=(I\otimes V^*)(T\hat{\otimes}I)(I\otimes V)=T\otimes_{\pi_\phi}I
\end{align*}
Thus $\pounds'$ is a relativization map for the $G$-equivariant coupling associated to $(G, (K, W), (H, U))$.  Similarly, given $\pounds'$ we can define $\pounds(T)=(I\otimes V)\pounds'(T)(I\otimes V^*)$. By the same argument, $\pounds$ is a completely positive map which is strictly continuous on the unit ball in $\mathcal{B}(H)$ and makes the following diagram commute:
\begin{center}
    \begin{tikzcd}
	{\mathcal{B}(H)} &&& \\
	\\
	{\mathcal{B}(H)^{\text{Ad}(U)}} &&& {\mathcal{B}(H\hat{\otimes}K)^{\text{Ad}(U\hat{\otimes}W)}}
	\arrow["\pounds"{description}, dashed, from=1-1, to=3-4]
	\arrow[hook', from=3-1, to=1-1]
	\arrow["{T\mapsto T\hat{\otimes}I}"{description}, from=3-1, to=3-4]
\end{tikzcd}
\end{center}
It is clear the construction between $\mathcal{L}$ and $\mathcal{L}'$ are inverses of each other. Thus $(4)$ holds.

\end{proof}

\begin{lemma}\label{Lemma: equivalence of normal and strict continuity}
\

\noindent Let $H$ and $K$ be Hilbert spaces. Let $\phi:\mathcal{B}(H)\to \mathcal{B}(K)$ be a unital completely positive map. Then $\phi$ is normal if and only if $\phi$ is strictly continuous on the unit ball.

\end{lemma}

\begin{proof}\

\noindent First, suppose $\phi$ is strictly continuous on the unit ball, then $\phi$ is continuous in the strong operator topology on the unit ball. As the strong operator topology is finer than the weak operator topology, $\phi$ is continuous with respect to the weak operator topology on the unit ball. As $\phi(B_1(\mathcal{B}(H)))\subset B_1(\mathcal{B}(K))$ and the weak operator topology agrees with the ultra-weak operator topology on norm bounded sets, the map $\phi$ is ultra-weak continuous on the unit ball. As every ultra-weak convergent sequence is norm bounded by an application of the Uniform Boundedness Principle, it follows $\phi$ is ultra-weak continuous. Thus $\phi$ is normal.

Now suppose $\phi$ is normal. We claim $\phi$ maps strictly convergent nets in the unit ball to strictly convergent nets in $\mathcal{B}(K)$. As the strict topology forms a topological vector space and $\phi$ is linear, it suffices to work with strictly convergent nets in $B_1(\mathcal{B}(H))$ which converge to $0\in \mathcal{B}(H)$. Let $(T_\lambda)_{\lambda\in\Lambda}$ be a net of operators in $B_1(\mathcal{B}(H))$ which converge in the strict topology to $0\in B_1(\mathcal{B}(H))$. As $(T_\lambda)_{\lambda\in\Lambda}$ is uniformly norm bounded, the net $(T_\lambda^*T_\lambda)_{\lambda\in\Lambda}$ converges to $0$ in the strict topology. As the strict topology is finer than the ultra-weak topology on the unit ball, then $(T_\lambda^*T_\lambda)_{\Lambda\in\Lambda}$ converges to $0$ in the ultra-weak topology. Using the continuity of $\phi$, we have $(\phi(T_\lambda^*T_\lambda))_{\lambda\in\Lambda}$ converges to $\phi(0)=0$ in the ultra-weak topology on $\mathcal{B}(K)$.

Fix $x\in K$. The rank-1 projection $\langle x, \cdot\rangle x$ defines a trace-class operator on $K$. Therefore, by definition of the ultra-weak topology, $(\langle x,\phi(T_\lambda^*T_\lambda)x\rangle)_{\lambda\in\Lambda}$ converges to $\langle x, 0x\rangle=0$. As $\phi$ is a unital completely positive map, the map $\phi$ satisfies the Kadison-Schwartz Inequality (Proposition 3.3 in \cite{paulsen2002completely}): for all $S\in \mathcal{B}(H)$, $\phi(S)^*\phi(S)\leq \phi(S^*S)$. Therefore
\begin{align*}
    ||\phi(T_\lambda)x||^2=\langle \phi(T_\lambda)x, \phi(T_\lambda)x\rangle=\langle x, \phi(T_\lambda)^*\phi(T_\lambda)x\rangle\leq \langle x, \phi(T_\lambda^*T_\lambda)x\rangle
\end{align*}
Thus $\langle x, \phi(T_\lambda^*T_\lambda)x\rangle\to 0$ implies $\phi(T_\lambda)x\to 0$. A similar argument applied to $(T_\lambda T_\lambda^*)_{\lambda\in\Lambda}$ shows $\phi(T_\lambda)^*x\to 0$. Hence $(\phi(T_\lambda))_{\lambda\in\Lambda}$ converges to $0\in\mathcal{B}(K)$ with respect to the strict topology. Thus $\phi$ is continuous with respect to the strict topology on the unit ball in $\mathcal{B}(H)$. 

\end{proof}

\begin{theorem}\

\noindent Let $H_S$ be a Hilbert space, let $G$ be a locally compact Hausdorff, second countable topological group, and let $(H_S, U_S)$ be a unitary representation of $G$ on $H_S$. If  $(H_R, E, U_R)$ is a covariant system as defined in \cite{Loveridge_2018}, then the data
$((G, (H_S, U_S), (H_R, U_R)), \yen_E)$ defines a Hilbert module quantum reference frame for $H_S$ where $\yen_E$ is the relativization map associated to the triple $(H_R, E, U_R)$ in the framework of \cite{Loveridge_2018}.

\end{theorem}

\begin{proof}\

\noindent As established in Proposition $3$ of \cite{Loveridge_2018}, $\yen_E$ is a unital completely positive normal map making the following diagram commute:
\begin{center}
    \begin{tikzcd}
	{\mathcal{B}(H_S)} &&& \\
	\\
	{\mathcal{B}(H_S)^{\text{Ad}(U_S)}} &&& {\mathcal{B}(H_S\hat{\otimes}H_R)^{\text{Ad}(U_S\hat{\otimes}U_R)}}
	\arrow["\yen_E"{description}, dashed, from=1-1, to=3-4]
	\arrow[hook', from=3-1, to=1-1]
	\arrow["{T\mapsto T\hat{\otimes}I}"{description}, from=3-1, to=3-4]
\end{tikzcd}
\end{center}
Using Lemma \ref{Lemma: equivalence of normal and strict continuity}, $\yen_E$ is a unital completely positive map which is strictly continuous on the unit ball. Therefore the data
$((G, (H_S, U_S), (H_R, U_R)), \yen_E)$ defines a Hilbert module quantum reference frame for $H_S$ by Proposition \ref{Prop: Hilbert space case}.

\end{proof}

\subsection{Hilbert module quantum reference frames and $C^*$-algebras}

Another simplified setting of Hilbert module quantum reference frames occurs when the $C^*$-algebras are viewed as Hilbert modules over themselves. In this case, the construction of the relational joint space and relation joint system reduces to only having to apply the equivariant KSGNS construction. From this reduction, we recover our model of classical reference frames when the algebras are commutative due to the correspondence between non-degenerate completely positive $G$-equivariant maps and $G$-invariant weakly continuous Markov kernels. 

\begin{proposition}\label{Prop: C^*-algebra case}\

\noindent Let $(A, G, \alpha)$ and $(B, G, \beta)$ be a $C^*$-dynamical systems. Every non-degenerate completely positive $G$-equivariant map $\phi:A\to M(B)$ determines and is determined by a $G$ -equivariant coupling of the form
$$(((A, G, \alpha), (B, G, \beta), ((B, \phi), \beta)), ((A, G, \alpha), A, \alpha)).$$
Furthermore, if $(F_\phi, \pi_\phi, V_\phi, W)$ is the unitarily unique quadruple obtained from apply the equivariant KSGNS construction to the non-degenerate positive equivariant $C^*$-correspondence $((B, \phi), \beta)$, then
\begin{enumerate}
    \item the relational joint space is canonically unitarily isomorphic to $F_\phi$,
    \item  the relational joint system is canonically isomorphic to $\mathcal{L}(F_\phi)$,
    \item the reference frame system is canonically isomorphic to $\mathcal{L}(F_\phi)^{\text{Ad}(W)}$,
    \item and a relativization map for the equivariant coupling determines and is determined by a completely positive map $\pounds:M(A)\to \mathcal{L}(F_\phi)^{\text{Ad}(W)}$ which is strictly continuous on the unit ball in $M(A)$ and make the following diagram commute
\begin{center}
    \begin{tikzcd}
	{M(A)} &&& \\
	\\
	{M(A)^{\text{Ad}(\alpha)}} &&& {\mathcal{L}(F_\phi)^{\text{Ad}(W)}}
	\arrow["\pounds"{description}, dashed, from=1-1, to=3-4]
	\arrow[hook', from=3-1, to=1-1]
	\arrow["{T\mapsto \tilde{\pi}_\phi(T)}"{description}, from=3-1, to=3-4]
\end{tikzcd}
\end{center}
where $\tilde{\pi}_\phi$ is the unique unital $*$-algebra homomorphism which extends $\pi_\phi$ to $M(A)$.
\end{enumerate}

\end{proposition}

\begin{proof}\

\noindent First, suppose we have a $G$-equivariant coupling for $A$ given by
$$(((A, G, \alpha), (B, G, \beta), ((B, \phi), \beta)), ((A, G, \alpha), A, \alpha))$$
As $((B, \phi), \beta)$ is a non-degenerate positive equivariant $C^*$-correspondence from $(A, G, \alpha)$ to $(B, G, \beta)$, the map $\phi:A\to M(B)$ is a non-degenerate completely positive map which satisfies the condition
$$\beta_g\circ \phi(a)=\phi(\alpha_g(a))\circ \beta_g$$
for all $g\in G$ and $a\in A$. Therefore, if $\tilde{\beta}_g$ is the unique $*$-algebra isomorphism on $M(B)$ induced by $\beta_g$, then $\tilde{\beta}_g(\phi(a))=\phi(\alpha_g(a))$ for all $a\in A$ and $g\in G$. Thus $\phi:A\to M(B)$ is a non-degenerate completely positive $G$-equivariant map.

Now suppose we have a non-degenerate completely positive $G$-equivariant map $\phi$, then for all $g\in G$ and $a\in A$, we have
$$\tilde{\beta}_g(\phi(a))=\phi(\alpha_g(a))$$
which implies $\beta_g\circ \phi(a)=\phi(\alpha_g(a))\circ\beta_g$. As for each $g\in G$, $\beta_g$ is a $\beta_g$-adjointable unitary on the Hilbert $B$-module $B$, the pair $((B, \phi), \beta)$ forms a non-degenerate positive equivariant $C^*$-correspondence from $(A, G, \alpha)$ to $(B, G, \beta)$. As for each $g\in G$, $\alpha_g$ is an $\alpha_g$-adjointable unitary on the Hilbert $A$-module $A$, the data 
$$(((A, G, \alpha), (B, G, \beta), ((B, \phi), \beta)), ((A, G, \alpha), A, \alpha))$$
forms a $G$-equivariant coupling for $A$. It is clear these two constructions are inverses of each other.

Fix a $G$-equivariant coupling 
$$(((A, G, \alpha), (B, G, \beta), ((B, \phi), \beta)), ((A, G, \alpha), A, \alpha))$$
for $A$, and let $(F_\phi, \pi_\phi, V_\phi, W)$ be the quadruple obtains from the equivariant KSGNS construction. As $\pi_\phi:A\to \mathcal{L}(F_\phi)$ is a non-degenerate $*$-algebra homomorphism, the map
$$U:A\otimes_{\pi_\phi}F_\phi\to F_\phi\quad\quad U(a\dot{\otimes}x)=\pi_\phi(a)x$$
defines a unitary isomorphism of the Hilbert $B$-modules. Using $U$, we obtain a $*$-algebra isomorphism
$$U(-)U^*:\mathcal{L}(A\otimes_{\pi_\phi}F_\phi)\to \mathcal{L}(F_\phi)\quad\quad T\mapsto U TU^*$$
Furthermore, we have for all $g\in G$ that $U(\alpha_g\otimes_{\pi_\phi}W_g)U^*=W_g$ which implies the isomorphism $U(-)U^*$ descends to a $*$-algebra isomorphism 
$$\mathcal{L}(A\otimes_{\pi_\phi}F_\phi)^{\text{Ad}(\alpha\otimes_{\pi_\phi}W)}\to \mathcal{L}(F_\phi)^{\text{Ad}(W)}.$$
Thus, from the canonical unitary $U$, we obtain $(1)-(3)$.

We observe that for all $T\in M(A)$,
$$U\tilde{\pi}_\phi(T)U^*=U(T\otimes_{\pi_\phi}I)U^*=\tilde{\pi}_\phi(T).$$
Suppose we have a completely positive map 
 $\pounds:M(A)\to \mathcal{L}(F_\phi)^{\text{Ad}(W)}$ which is strictly continuous on the unit ball in $M(A)$ and makes the following diagram commute
\begin{center}
    \begin{tikzcd}
	{M(A)} &&& \\
	\\
	{M(A)^{\text{Ad}(\alpha)}} &&& {\mathcal{L}(F_\phi)^{\text{Ad}(W)}}
	\arrow["\pounds"{description}, dashed, from=1-1, to=3-4]
	\arrow[hook', from=3-1, to=1-1]
	\arrow["{T\mapsto \tilde{\pi}_\phi(T)}"{description}, from=3-1, to=3-4]
\end{tikzcd}
\end{center}
Define
$$\pounds':M(A)\to \mathcal{L}(A\otimes_{\pi_\phi}F_\phi)^{\text{Ad}(\alpha\otimes_{\pi_\phi}W)}\quad\quad \pounds'(T)=U^*\pounds(T)U$$
As the composition of a unital $*$-algebra homomorphism and completely positive map which is strictly continuous on the unit ball, $\pounds'$ is a completely positive map which is strictly continuous on the unit ball. Furthermore, for all $T\in M(A)^{\text{Ad}(\alpha)}$, 
$$\pounds'(T)=U^*\pounds(T)U=U^*\tilde{\pi}_\phi(T)U=T\otimes_{\pi_\phi}I.$$
Thus $\pounds'$ defines a relativization map for the $G$-equivariant coupling. Similarly, given $\pounds'$ we can define $\pounds(T)=U\pounds'(T)U^*$. By the same argument, $\pounds$ is a completely positive map which is strictly continuous on the unit ball in $M(A)$ and makes the following diagram commute:
\begin{center}
    \begin{tikzcd}
	{M(A)} &&& \\
	\\
	{M(A)^{\text{Ad}(\alpha)}} &&& {\mathcal{L}(F_\phi)^{\text{Ad}(W)}}
	\arrow["\pounds"{description}, dashed, from=1-1, to=3-4]
	\arrow[hook', from=3-1, to=1-1]
	\arrow["{T\mapsto \tilde{\pi}_\phi(T)}"{description}, from=3-1, to=3-4]
\end{tikzcd}
\end{center}
It is clear the construction between $\mathcal{L}$ and $\mathcal{L}'$ are inverses of each other. Thus $(4)$ holds.

\end{proof}

\begin{example}\

\noindent Let $(A, G, \alpha)$ be a $C^*$-dynamical system, and let $(\mathbb{C}, G, \text{trv})$ denote the trivial $C^*$-dynamical system on $\mathbb{C}$. Observe, non-degenerate completely positive $G$-equivariant maps $\phi:A\to \mathbb{C}$ are in one-to-one correspondence with $G$-invariant states on $A$. Given a $G$-invariant state $\phi\in S(A)$, we can apply the GNS construction to obtain a quadruple $(H_\phi, \pi_\phi, \Omega_\phi, U)$ where $(H_\phi, \pi_\phi, \Omega_\phi)$ is the usual triple from the GNS construction and $U:G\to \mathcal{B}(H_\phi)$ is a unitary representation with the property that for all $g\in G$, $U_g\circ \pi_\phi(a)=\pi_\phi(\alpha_g(a))\circ U_g$. Note, if we define
$$V_\phi:\mathbb{C}\to H_\phi\quad\quad V_\phi(\lambda)=\lambda\Omega_\phi,$$
then the quadruple $(H_\phi, \pi_\phi, V_\phi, U)$ satisfies the conditions for the quadruple when applying the KSGNS construction to the correspondence $((\mathbb{C}, \phi), \text{trv})$. Therefore, it follows the relational joint space is given by $H_\phi$, the relational joint system is given by $\mathcal{B}(H_\phi)$, and the reference frame system is given by $\mathcal{B}(H_\phi)^{G}$. Furthermore, a relativization map for the equivariant coupling determined by the state $\phi$ is given by a completely positive map $\pounds:M(A)\to \mathcal{B}(H_\phi)^{G}$ which is strictly continuous on the unit ball in $M(A)$ and makes the following diagram commute:
\begin{center}
    \begin{tikzcd}
	{M(A)} &&& \\
	\\
	{M(A)^{G}} &&& {\mathcal{B}(H_\phi)^{G}}
	\arrow["\pounds"{description}, dashed, from=1-1, to=3-4]
	\arrow[hook', from=3-1, to=1-1]
	\arrow["{T\mapsto \tilde{\pi}_\phi(T)}"{description}, from=3-1, to=3-4]
\end{tikzcd}
\end{center}
   
\end{example}

\begin{example}\

\noindent We can model a qubit using the $C^*$-algebra $M_2(\mathbb{C})$. Suppose we have the $\mathbb{Z}/2\mathbb{Z}$ symmetry given by 
$$\alpha:\mathbb{Z}/2\mathbb{Z}\times M_2(\mathbb{C})\to M_2(\mathbb{C})\quad\quad \alpha([1], A)=\sigma_zA\sigma_z.$$
Observe the fixed points of $M_2(\mathbb{C})$ with respect to this action are given by
$$M_2(\mathbb{C})^{\mathbb{Z}/2\mathbb{Z}}=\left\{\begin{bmatrix}
    a & 0\\
    0 & b
\end{bmatrix}\in M_2(\mathbb{C}): a, b\in\mathbb{C}\right\}\cong\mathbb{C}\oplus\mathbb{C}.$$
Suppose we utilize another qubit as a reference system, and we condition the original qubit by the reference qubit via the phase flip quantum channel
$$\phi_p:M_2(\mathbb{C})\to M_2(\mathbb{C})\quad\quad \phi_p(A)=(1-p)A+p\sigma_zA\sigma_z$$
for a fixed $p\in [0, 1]$. Therefore, we obtain a $\mathbb{Z}/2\mathbb{Z}$-equivariant coupling
$$(((M_2(\mathbb{C}), \mathbb{Z}/2\mathbb{Z}, \alpha), (M_2(\mathbb{C}), \mathbb{Z}/2\mathbb{Z}, \alpha), ((M_2(\mathbb{C}), \phi_p), \alpha)), ((M_2(\mathbb{C}), \mathbb{Z}/2\mathbb{Z}, \alpha), M_2(\mathbb{C}), \alpha)).$$

We observe that in the case $p=0$ or $p=1$, $\phi_p$ is a unital $*$-algebra homomorphisms. In both cases, the resulting relational joint space is isomorphic to $M_2(\mathbb{C})$, and the resulting relational joint system is isomorphic to $M_2(\mathbb{C})$. Furthermore, under these identifications,  relativization map for the equivariant coupling is equivalent to a completely positive map $\pounds:M_2(\mathbb{C})\to M_2(\mathbb{C})^{\mathbb{Z}/2\mathbb{Z}}$ making the following diagram commute:
\begin{center}
    \begin{tikzcd}
	{M_2(\mathbb{C})} &&& \\
	\\
	{M_2(\mathbb{C})^{\mathbb{Z}/2\mathbb{Z}}} &&& {M_2(\mathbb{C})^{\mathbb{Z}/2\mathbb{Z}}}
	\arrow["\pounds"{description}, dashed, from=1-1, to=3-4]
	\arrow[hook', from=3-1, to=1-1]
	\arrow["{\phi_p}"{description}, from=3-1, to=3-4]
\end{tikzcd}
\end{center}
For example, we can take $\pounds$ as the map $\pounds(A)=\frac{1}{2}\left(A+\sigma_zA\sigma_z\right)$ which is a conditional expectation from $M_2(\mathbb{C})$ to the subalgebra $M_2(\mathbb{C})^{\mathbb{Z}/2\mathbb{Z}}$.

Now suppose $p\in (0, 1)$. Let $(F_{\phi_p}, \pi_{\phi_p}, V_{\pi_p}, W)$ be the unitarily  unique quadruple obtained from applying Theorem \ref{Thrm: Equiv Dil} to $((M_2(\mathbb{C}), \phi_p), \alpha)$. We claim $F_{\phi_p}\cong M_2(\mathbb{C})\oplus M_2(\mathbb{C})$. Define the map 
\begin{align*}
\Phi_{p}&:M_2(\mathbb{C})\otimes_{\text{alg}}M_2(\mathbb{C})\to M_2(\mathbb{C})\oplus M_2(\mathbb{C})\\
\Phi_p\left(\sum\limits_{i=1}^n A_i\otimes B_i\right)&=\left(\sum\limits_{i=1}^n \sqrt{1-p}A_iB_i, \sum\limits_{i=1}^n\sqrt{p}A_i\sigma_zB_i\right)
\end{align*}
It is clear $\Phi_p$ is $M_2(\mathbb{C})$-linear. Furthermore, using the standard basis of $M_2(\mathbb{C})\otimes_{\text{alg}}M_2(\mathbb{C})$, it follows $\Phi_p$ is surjective. As
\begin{align*}
    \langle\sum\limits_{i=1}^n A_i\otimes B_i, \sum\limits_{i=1}^n A_i\otimes B_i \rangle&=\sum\limits_{i, j=1}^n B_i^*\phi_p(A_i^*A_j)B_j=\sum\limits_{i, j}^n (1-p)B_i^*A_i^*A_jB_j+\sum\limits_{i, j=1}^n pB_i^*\sigma_zA_i^*A_j\sigma_zB_j\\
    &=\left(\sum\limits_{i=1}^n \sqrt{1-p}A_iB_i\right)^*\left(\sum\limits_{i=1}^n \sqrt{1-p}A_iB_i\right)\\
    &\quad\quad\quad+\left(\sum\limits_{i=1}^n \sqrt{p}A_i\sigma_zB_i\right)^*\left(\sum\limits_{i=1}^n \sqrt{p}A_i\sigma_zB_i\right)\\
    &=\langle\Phi_p(\sum\limits_{i=1}^n A_i\otimes B_i ), \Phi_p(\sum\limits_{i=1}^n A_i\otimes B_i) \rangle
\end{align*}
it follows the kernel of $\Phi_p$ is precisely the submodule of elements of $M_2(\mathbb{C})\otimes_{\text{alg}}M_2(\mathbb{C})$ with norm zero. Thus $\Phi_p$ descends to a well-defined $M_2(\mathbb{C})$-linear bijective isometry 
$$F_{\phi_p}\to M_2(\mathbb{C})\oplus M_2(\mathbb{C})$$
which we also denote by $\Phi_p$. Hence $\Phi_p$ defines a unitary operator establishing the isomorphism $F_{\phi_p}\cong M_2(\mathbb{C})\oplus M_2(\mathbb{C})$. Note, under this isomorphism, the induced action of $\mathbb{Z}/2\mathbb{Z}$ on $M_2(\mathbb{C})\oplus M_2(\mathbb{C})$ is given by
$$(A, B)\mapsto (\sigma_zA\sigma_z, \sigma_zB\sigma_z)$$

Using the isomorphism $F_{\phi_p}\cong M_2(\mathbb{C})\oplus M_2(\mathbb{C})$, we obtain the $*$-algebra isomorphism 
$$\mathcal{L}(F_{\phi_p})\cong M_2(M_2(\mathbb{C}))\cong M_4(\mathbb{C}).$$
Under this identification, the induced $\mathbb{Z}/2\mathbb{Z}$ action on $M_4(\mathbb{C})$ is given by
$$\begin{bmatrix}
    A & B\\
    C & D
\end{bmatrix}\mapsto \begin{bmatrix}
    \sigma_z A\sigma_z & \sigma_z B\sigma_z\\
    \sigma_z C\sigma_z & \sigma_z D\sigma_z
\end{bmatrix}$$
so that 
$$M_4(\mathbb{C})^{\mathbb{Z}/2\mathbb{Z}}=\left\{\begin{bmatrix}
    A & B\\
    C & D
\end{bmatrix}\in M_4(\mathbb{C}): A, B, C, D\text{ are diagonal matrices}\right\}\cong M_2(\mathbb{C})\oplus M_2(\mathbb{C})$$
where the last isomorphism is given by
$$\begin{bmatrix}
    a & 0 & c & 0\\
    0 & b & 0 & d\\
    e & 0 & g & 0\\
    0 & f & 0 & h
\end{bmatrix}\mapsto \left(\begin{bmatrix}
    a & c\\
    e & g
\end{bmatrix},\begin{bmatrix}
    b & d\\
    f & h
\end{bmatrix} \right)$$
Thus, a relativization map for the equivariant coupling when $p\in (0, 1)$ is equivalent to the data of a completely positive map $\pounds:M_2(\mathbb{C})\to M_2(\mathbb{C})\oplus M_2(\mathbb{C})$ which is strictly continuous on the unit ball and makes the following diagram commute:
\begin{center}
    \begin{tikzcd}
	{M_2(\mathbb{C})} &&&& \\
	\\
	{M_2(\mathbb{C})^{\mathbb{Z}/2\mathbb{Z}}} &&&& {M_2(\mathbb{C})\oplus M_2(\mathbb{C})}
	\arrow["\pounds"{description}, dashed, from=1-1, to=3-5]
	\arrow[hook', from=3-1, to=1-1]
	\arrow["{(aE_{11}+bE_{22})\mapsto (aI, bI)}"{description}, from=3-1, to=3-5]
\end{tikzcd}
\end{center}
In this example, there only exists one relativization map $\pounds$ which is given by
$$\pounds':M_2(\mathbb{C})\to M_2(\mathbb{C})\oplus M_2(\mathbb{C})\quad\quad \pounds'\left(\begin{bmatrix}
    a & b\\
    c & d
\end{bmatrix}\right)=(aI, dI).$$
Indeed, give $\pounds$, we can write $\pounds=(\pounds_1, \pounds_2)$ where $\pounds_i:M_2(\mathbb{C})\to M_2(\mathbb{C})$ is a unital completely positive satisfying the property  $\pounds_i(aE_{jj})=a\delta_{ij}I$. As $\pounds_i$ is a completely positive, the map
$$\pounds_{i, 2}:M_2(M_2(\mathbb{C}))\to M_2(M_2(\mathbb{C}))\quad\quad \begin{bmatrix}
    A & B\\
    C & D
\end{bmatrix}\mapsto \begin{bmatrix}
    \pounds_i(A) & \pounds_i(B)\\
    \pounds_i(C) & \pounds_i(D)
\end{bmatrix}$$
is positive. In the case $i=1$, we have
$$\pounds_{1, 2}\left(\begin{bmatrix}
    E_{11} & E_{12}\\
    E_{21} & E_{22}
\end{bmatrix}\right)=\begin{bmatrix}
    I & \pounds_{1}(E_{12})\\
    \pounds_1(E_{21}) & 0
\end{bmatrix}.$$
As 
$$\begin{bmatrix}
    E_{11} & E_{12}\\
    E_{21} & E_{22}
\end{bmatrix}\in M_2(M_2(\mathbb{C}))$$ is positive, the resulting image needs to be positive; however, it is only positive when $\pounds_1(E_{12})=0=\pounds_1(E_{21})$. A similar argument shows 
$$\pounds_2(E_{11})=\pounds_2(E_{12})=\pounds_2(E_{21})=0\quad\quad \pounds_2(E_{22})=I$$
Hence $\pounds=\pounds'$ proving uniqueness.

\end{example}

    \appendix

    \section{Correspondence Between Weakly Continuous Markov Kernels and Non-degenerate Completely Positive Maps}\label{App; A}

    In this appendix we prove a bijective correspondence between 
    weakly continuous Markov kernels from $Y$ to $X$ and non-degenerate completely positive maps $C_0(X)\to C_b(Y)$. Given a weakly continuous Markov kernel $\kappa:Y\times\text{Bor}(X)\to [0, 1]$, we build  a non-degenerate completely positive map by integrating against $\kappa$; in particular, we define
    $$\phi_\kappa:C_0(X)\to C_b(Y)\quad\quad \phi_{\kappa}(f)(y)=\int_{X}f(x)\;\kappa(y, dx)$$
    Showing this map is well-defined, $\mathbb{C}$-linear, and completely positive is rather straight forward. The challenging part is in establishing $\phi_{\kappa}$ is non-degenerate. To establish non-degeneracy, we utilize the two following lemmas.
    
    \begin{lemma}\label{Lemma: Strict topology and uni con on compact}\

    \noindent Let $X$ be a locally compact Hausdorff space, and let $B_1(C_b(X))$ denote the unit ball in $C_b(X)$.  A net $(f_\lambda)_{\lambda\in\Lambda}$ in $B_1(C_b(X))$ converges to $f\in B_1(C_b(X))$ strictly if and only if $(f_\lambda)_{\lambda\in\Lambda}$  converges to $f$  uniformly on every compact subset of $X$.
    
    \end{lemma}
    
    \begin{proof}\
    
    \noindent First suppose the net converges strictly. Let $K\subset X$ be a compact set. Using Urysohn's Lemma, there exists $g\in C_0(X)$ such that $g(x)=1$ for all $x\in K$. Let $\epsilon>0$. By assumption, there exists $\lambda_0\in \Lambda$ such that for all $\lambda\in \Lambda$ with $\lambda_0\leq \lambda$, $||f_\lambda g-fg||_{\sup}<\epsilon$. Thus, 
    $$\sup\limits_{x\in K}|f_\lambda(x)-f(x)|=\sup\limits_{x\in K}|f_\lambda(x)g(x)-f(x)g(x)|\leq ||f_\lambda g-fg||_{\sup}<\epsilon$$
    showing $f_\lambda\to f$ uniformly on $K$.
    
    Now suppose the net converges uniformly on every compact subset of $X$.  Let $g\in C_0(X)$, and let $\epsilon>0$. As $g$ vanishes at infinity, there exists a compact subset $K\subset X$ such that $|g(x)|<\frac{\epsilon}{2}$ for all $x\in X\backslash K$. Let $M>0$ such that $|g(x)|\leq M$ for all $x\in K$.  By assumption, there exists $\lambda_0\in\Lambda$ such that for all $\lambda\in\Lambda$ with $\lambda_0\leq \lambda$,
    $$\sup\limits_{x\in K}|f_{\lambda_0}(x)-f(x)|<\frac{\epsilon}{M}.$$
    Thus given $\lambda\in \Lambda$ with $\lambda_0\leq \lambda$, we have for all $x\in X\backslash K$
    $$|f_\lambda(x)g(x)-f(x)g(x)|\leq |f_\lambda (x)-f(x)|*|g(x)|< 2*\frac{\epsilon}{2}=\epsilon$$
    and
    $$|\overline{f}_\lambda(x)g(x)-\overline{f}(x)g(x)|\leq |\overline{f}_\lambda (x)-\overline{f}(x)|*|g(x)|< 2*\frac{\epsilon}{2}=\epsilon,$$
    while for all $x\in K$,
    $$|f_\lambda(x)g(x)-f(x)g(x)|\leq |f_\lambda(x)-f(x)|*M<\epsilon$$
    and
    $$|\overline{f}_\lambda(x)g(x)-\overline{f}(x)g(x)|\leq |\overline{f}_\lambda(x)-\overline{f}(x)|*M<\epsilon.$$
    Thus $||f_\lambda g-fg||_{\sup}<\epsilon$ and $||\overline{f}_\lambda g-\overline{f}g||_{\sup}<\epsilon$ showing $(f_\lambda)_{\lambda\in\Lambda}$ converges to $f$  in the strict topology.
        
    \end{proof}
    
    \begin{lemma}\label{Lemma: Prokhorov Result}\

    \noindent Let $X$ be a locally compact Hausdorff topological space. If $K\subset \text{Prob}_R(X, \text{Bor}(X))$ is a non-empty, weak compact subset , then for each $\epsilon>0$, there exists a non-empty compact subset $K_\epsilon\subset X$ such that for all $\mu\in K$, $\mu(X\backslash K_\epsilon)<\epsilon$.
    
    \end{lemma}
    
    \begin{proof}\
    
    \noindent This immediately follows from Corollary 8.10.11 in \cite{bogachev2007measure}.
    
    \end{proof}

    \begin{proposition}\label{Prop: weak Markov Kernel to NCP map}\

    \noindent Let $X$ and $Y$ be locally compact Hausdorff topological spaces, and let $\kappa:Y\times\text{Bor}(X)\to [0, 1]$ be a weakly continuous Markov kernel. Then the map
    $$\phi_{\kappa}:C_0(X)\to C_b(Y)\quad\quad \phi_{\kappa}(f)(y)=\int_{X}f(x)\;\kappa(y, dx)$$
    is a well-defined, non-degenerate, completely positive map.
        
    \end{proposition}
    
    \begin{proof}\
    
    \noindent We begin by verifying the map $\phi_\kappa$ is well-defined. Fix $f\in C_0(X)$. As $f$ is a continuous function which vanishes at infinity, the function $f$ is a continuous and bounded. Therefore, as $\kappa(y, \cdot)$ is a finite measure on $(X, \text{Bor}(X))$ for all $y\in Y$,  the map
    $$Y\to \mathbb{C}\quad\quad y\mapsto \int_{X}f(x)\;\kappa(y, dx)$$
    forms a well-defined function. Furthermore, as $|f(x)|\leq ||f||_{\sup}<\infty$, we have for all $y\in Y$
    \begin{align*}
        |\phi_\kappa(f)(y)|=\left|\int_{X}f(x)\;\kappa(y, dx)\right|\leq \int_{X}|f(x)|\;\kappa(y, dx)\leq ||f||_{\sup}\kappa(y, X)=||f||_{\sup}
    \end{align*}
     showing $\phi_\kappa(f)$ is a bounded function with $||\phi_\kappa(f)||_{\sup}\leq ||f||_{\sup}$. Thus, the only remaining thing to check in order to know $\phi$ is well-defined is that $\phi_\kappa(f)$ is continuous. Let $(y_\lambda)_{\lambda\in\Lambda}$ be a net in $Y$ which converges to $y\in Y$. As $\kappa$ is weakly continuous, the net $(\kappa(y_\lambda, \cdot))_{\lambda\in\Lambda}$ converges to $\kappa(y, \cdot)$ in the weak topology on $\text{Prob}_r(X, \text{Bor}(X))$; that is, for all $g\in C_b(X)$
     $$\phi_\kappa(g)(y)=\int_{X}g(x)\;\kappa(y, dx)=\lim\limits_{\lambda\to\infty}\int_Xg(x)\;\kappa(y_\lambda, dx)=\lim\limits_{\lambda\to\infty}\phi(g)(y_\lambda)$$
     Taking $g=f$ shows $\phi_\kappa(f)$ is continuous. Hence $\phi_\kappa$ is a well defined map.
     
    Using $\mathbb{C}$-linearity of the integral, it follows $\phi_\kappa$ is $\mathbb{C}$-linear.  To see $\phi_\kappa$ is completely positive, it suffices to show $\phi_\kappa$ is positive (Theorem 3.9 in \cite{paulsen2002completely}). Let $f\in C_0(X)$ be positive; that is, $f(X)\subset [0, \infty)$. Since $\kappa(y, \cdot)$ is a probability measure (thus a positive measure), then 
    $$ \phi_{\kappa}(f)(y)=\int_{X}\;f(x)\;\kappa(y, dx)\in [0, \infty)$$
    for all $y\in Y$. Therefore $\phi_\kappa(f)$ is a positive element in $C_b(Y)$
    showing $\phi_\kappa$ is positive. Hence $\phi_\kappa$ is completely positive.
    
    Finally, we show $\phi_\kappa$ is non-degenerate. Define
    $$\tilde{\phi}_\kappa:C_b(X)\to C_b(Y)\quad\quad\tilde{\phi}_\kappa(f)(y)=\int_{X}f(x)\;\kappa(y, dx).$$
    Using the argument above, $\tilde{\phi}_{\kappa}$ is a well-defined, completely positive map satisfying the bound $||\tilde{\phi}_\kappa(f)||_{\sup}\leq ||f||_{\sup}$ for all $f\in C_b(X)$. Furthermore, as $\kappa(y, \cdot)$ is probability measure for each $y\in Y$, the $\tilde{\phi}$ is unital. Observe, $\tilde{\phi}_\kappa|_{C_0(X)}=\phi_\kappa$. Therefore, to know $\phi_\kappa$ is non-degenerate, it suffices to prove $\tilde{\phi}_\kappa$ is strictly continuous on the unit ball in $C_b(X)$ (Corollary 5.7 in \cite{lance1995hilbert}). 
    
    Let $(f_\lambda)_{\lambda\in\Lambda}$ be a net in $B_1(C_b(X))$ which converges strictly to $f\in B_1(C_b(X))$. As $(\tilde{\phi}_\kappa(f_\lambda))_{\lambda\in\Lambda}$ is a net in $B_1(C_b(Y))$ and $\tilde{\phi}_\kappa(f)\in B_1(C_b(Y))$, it suffices to show $(\tilde{\phi}_\kappa(f_\lambda))_{\lambda\in\Lambda}$ converges to $\tilde{\phi}_\kappa(f)$ uniformly on every compact subset of $Y$ (Lemma \ref{Lemma: Strict topology and uni con on compact}). Let $K\subset Y$ be a non-empty compact subset, and let $\epsilon>0$. As the map
    $$F:Y\to \text{Prob}_R(X, \text{Bor}(X))\quad\quad F(y)=\kappa(y, \cdot)$$
    is a continuous with respect to the weak topology, $F(K)$ is a non-empty compact subset of $\text{Prob}_R(X, \text{Bor}(X))$. Using Lemma \ref{Lemma: Prokhorov Result}, there exists a non-empty compact subset $K_\epsilon\subset X$ such that $\kappa(y, X\backslash K_\epsilon)<\frac{\epsilon}{4}$ for all $y\in K$. Using Lemma \ref{Lemma: Strict topology and uni con on compact}, the net $(f_\lambda)_{\lambda\in\Lambda}$ converges to $f$ uniformly on $K_{\epsilon}$. Thus, there exists $\lambda_0\in \Lambda$ such that for all $\lambda\in\Lambda$ with $\lambda\geq\lambda_0$, $$\sup\limits_{x\in K_{\epsilon}}|f_\lambda(x)-f(x)|<\frac{\epsilon}{2}.$$
    
    Fix $\lambda\in\Lambda$ with $\lambda\geq \lambda_0$.  For all $y\in K$, we have
    \begin{align*}
        |\tilde{\phi}_{\kappa}(f_\lambda)(y)-\tilde{\phi}_{\kappa}(f)(y)|&\leq \int_{X}|f_\lambda(x)-f(x)|\;\kappa(y, dx)\\
        &=\int_{K_{\epsilon}}|f_\lambda(x)-f(x)|\;\kappa(y, dx)+\int_{X\backslash K_\epsilon}|f_\lambda(x)-f(x)|\;\kappa(y, dx)\\
        &\leq \kappa(y, K_\epsilon)\left(\sup\limits_{x\in K_\epsilon}|f_\lambda(x)-f(x)|\right)+\kappa(y, X\backslash K_\epsilon) ||f_\lambda-f||_{\sup}\\
        &\leq \kappa(y, X)\left(\sup\limits_{x\in K_\epsilon}|f_\lambda(x)-f(x)|\right)+2\kappa(y, X\backslash K_\epsilon)\\
        &=\left(\sup\limits_{x\in K_\epsilon}|f_\lambda(x)-f(x)|\right)+2\kappa(y, X\backslash K_\epsilon)<\epsilon
    \end{align*}
    Therefore $\sup\limits_{y\in K}|\tilde{\phi}_\kappa(f_\lambda)(y)-\tilde{\phi}_\kappa(f)(y)|<\epsilon$. Hence $(\tilde{\phi}_\kappa(f_\lambda))_{\lambda\in\Lambda}$ converges to $\tilde{\phi}_\kappa(f)$ on every compact subset of $Y$ proving $\tilde{\phi}_{\kappa}$ is continuous on the unit ball with respect to the strict topology. Therefore we conclude $\phi_\kappa$ is non-degenerate.
    
    \end{proof}
    
    We now show how one can build a weakly continuous Markov kernel from a non-degenerate completely positive map $\phi:C_0(X)\to C_b(Y)$. For each $y\in Y$, we obtain the positive linear functional 
    $$\text{ev}_y\circ \phi:C_0(X)\to \mathbb{C}\quad\quad (\text{ev}_y\circ \phi)(f)=\phi(f)(y).$$
    Using the Riesz-Markov-Kakutani Representation Theorem, we obtain a regular probability measure $\kappa(y, \cdot)$ on $(X, \text{Bor}(X))$. As we vary over $Y$, we obtain a $Y$-family of regular probability measures. From the construction of this $Y$-family of measures, it follows each measure is a Radon probability measure as well as the induced map
    $$Y\to \text{Prob}_{R}(X, \text{Bor}(X))\quad\quad y\mapsto \kappa(y, \cdot)$$
    is continuous with represent to the weak topology. Thus, the only thing left to check is that the induced map $Y\to [0, 1]$ is Borel measurable which turns out to be the challenging part.
    
    To establish this, we first show evaluation at a Borel measurable subset of $X$ forms a Borel measurable functions on $\text{Prob}_R(X, \text{Bor})$ \textit{with respect to} the weak topology. Using this as well as the induced map $Y\to\text{Prob}_R(X, \text{Bor}(X))$ is continuous with respect to the weak topology, the composition will be Borel measureable. By construction of the $Y$-family of probability measures, this composition is the induced map on $Y$. To prove the evaluation at Borel measurable subsets of $X$ define Borel measurable functions on $\text{Prob}_R(X, \text{Bor}(X))$, we follow the argument presented in Lemma 4.64 of \cite{parzygnat2017discrete} which establishes the same claim in the case $X$ is compact. As the argument is exactly the same, we omit the proof here.
    
    \begin{lemma}\label{Lemma: Borel evaluation}\

    \noindent Let $X$ be a locally compact Hausdorff topological space. For each $B\in\text{Bor}(X)$, the map
    $$\text{ev}_B:\text{Prob}_R(X, \text{Bor}(X))\to \mathbb{C}\quad\quad \text{ev}_B(\mu)=\mu(B)$$
    is Borel measurable with respect to the weak topology on $\text{Prob}_R(X, \text{Bor}(X))$.
        
    \end{lemma}

    \begin{proposition}\label{Prop: NCP map to weak Markov Kernel}\

    \noindent Let $X$ and $Y$ be locally compact Hausdorff topological spaces, and let $\phi:C_0(X)\to C_b(Y)$ be a non-degenerate completely positive map. For each $y\in Y$, let $\mu_y\in \mathcal{M}_r(X, \text{Bor}(X))$ be the unique positive regular measure associated to the positive linear functional
    $$\text{ev}_y\circ\phi:C_0(X)\to \mathbb{C}\quad\quad (\text{ev}_y\circ \phi)(f)=\phi(f)(y)$$
    from the Riesz-Markov-Kakutani Representation Theorem. Then the map
    $$\kappa_\phi:Y\times\text{Bor}(X)\to [0, 1]\quad\quad \kappa_\phi(y, B)=\mu_y(B)$$
    defines a weakly continuous Markov kernel which uniquely satisfies the following condition: for all $f\in C_0(X)$ and $y\in Y$,
    $$\phi(f)(y)=\int_{X}f(x)\;\kappa_\phi(y, dx).$$
    
    \end{proposition}
    
    \begin{proof}\
    
    \noindent We verify $\kappa_\phi$ satisfies the conditions to be a weakly continuous Markov kernel. We start by verifying $\kappa_\phi(y, \cdot)=\mu_y$ is a Radon probability measure. By the Riesz-Markov-Kakutani Representation Theorem, for all $y\in Y$, the regular positive measure $\mu_y$ on $(X, \text{Bor}(X))$ satisfies the two following conditions:
    \begin{enumerate}
    	\item $||\mu_y||=\mu(X)=||\text{ev}_y\circ \phi||$
    	\item for all $f\in C_0(X)$,
    	$\phi(f)(y)=\int_{X}f(x)\;d\mu_y(x)$.
    \end{enumerate}
    As $||\text{ev}_y\circ \phi||<\infty$, $\mu_y$ is a finite regular measure; thus, $\mu_y$ is a positive Radon measure. Furthermore, as $\text{ev}_y\circ\phi$ is non-degenerate completely positive map,  $||\text{ev}_y\circ \phi||=1$ which implies $\mu_y$ is a Radon probability measure. 
    
   Now we verify the map
   	$$Y\to \text{Prob}_R(X, \text{Bor}(X))\quad\quad y\mapsto \kappa_\phi(y, \cdot)$$
is continuous with respect to the weak topology. Let $(y_\lambda)_{\lambda\in\Lambda}$ be a net in $Y$ which converges to $y_0\in Y$, and let $\epsilon>0$. For each $\lambda\in\Lambda$ and $g\in C_0(X)$ with $0\leq g\leq 1$, we have
\begin{align*}
\left|\int_{X}f(x)\;\kappa_\phi(y_\lambda, dx)\right.&-\left.\int_{X}f(x)\;\kappa_\phi(y_0, dx)\right|\leq \left|\int_{X}f(x)g(x)\;\kappa_\phi(y_\lambda, dx)-\int_{X}f(x)\;\kappa_\phi(y_\lambda, dx)\right|\\   &+\left|\int_{X}f(x)g(x)\;\kappa_\phi(y_\lambda, dx)-\int_{X}f(x)g(x)\;\kappa_\phi(y_0, dx)\right|\\
&+\left|\int_{X}f(x)g(x)\;\kappa_\phi(y_0, dx)-\int_{X}f(x)\;\kappa_\phi(y_0, dx)\right|\\
&\leq \left\|f\right\|_{\sup}\int_{X}|1-g(x)|\kappa(y_\lambda, dx)+\left|\phi(fg)(y_\lambda)-\phi(fg)(y_0)\right|\\
&+ \left\|f\right\|_{\sup}\int_{X}|1-g(x)|\kappa(y_0, dx)\\
&=\left\|f\right\|_{\sup}\int_{X}(1-g(x))\kappa(y_\lambda, dx)+\left|\phi(fg)(y_\lambda)-\phi(fg)(y_0)\right|\\
&+ \left\|f\right\|_{\sup}\int_{X}(1-g(x))\kappa(y_0, dx)
\end{align*}
Let $(g_\gamma)_{\gamma\in\Gamma}$ be an approximate unit for $C_0(X)$. As $\phi$ is non-degenerate completely positive map, the net $(\phi(g_{\gamma}))_{\gamma\in\Gamma}$ converges to $1\in C_b(Y)$ strictly. By Lemma \ref{Lemma: Strict topology and uni con on compact}, the net converges to $1$ uniformly on every compact subset of $Y$. As $Y$ is a locally compact Hausdorff space, there exists an open subset $V\subset Y$ with $y_0\in V$ and $\overline{V}$ compact. Thus there exists $\gamma_0\in\Gamma$ such that for all $\gamma\in\Gamma$ for which $\gamma\geq \gamma_0$
$$\frac{\epsilon}{4||f||_{\sup}}>\sup\limits_{y\in \overline{V}}|1-\phi(g_\gamma)(y)|=\sup\limits_{y\in \overline{V}}\left|\int_X (1-g_\gamma(x))\;\kappa(y, dx)\right|=\sup\limits_{y\in \overline{V}}\int_X (1-g_\gamma(x))\;\kappa(y, dx).$$
As $fg_{\gamma_0}\in C_0(X)$, continuity of $\phi(fg_{\gamma_0})$ implies there exists $\lambda_0'\in \Lambda$ such that for all $\lambda\in\Lambda$ with $\lambda\geq \lambda_0'$,
$$|\phi(fg_{0})(y_\lambda)-\phi(fg_{0})(y_0)|<\frac{\epsilon}{2}$$
As $V\subset Y$ is open with $y_0\in V$ and $(y_\lambda)_{\lambda\in\Lambda}$ converges to $y_0$, there exists $\lambda_0''\in \Lambda$ such that for all $\lambda\in\Lambda$ with $\lambda\geq \lambda_0''$, $y_{\lambda}\in V$. Let $\lambda_0\in\Lambda$ with $\lambda_0\geq \lambda_0'$ and $\lambda_0\geq \lambda_0''$, then for all $\lambda\in\Lambda$ with $\lambda\geq \lambda_0$
$$\left|\int_{X}f(x)\;\kappa_\phi(y_\lambda, dx)-\int_{X}f(x)\;\kappa_\phi(y, dx)\right|<\epsilon.$$
Therefore we conclude the map
$$Y\to \text{Prob}_R(X, \text{Bor}(X))\quad\quad y\mapsto \kappa_\phi(y, \cdot)$$
is continuous with respect to the weak topology. 
    
    Finally, we verify the induced map
    $$Y\to [0, 1]\quad\quad y\mapsto \kappa_\phi(y, B)$$ is Borel measurable for all $B\in\text{Bor}(X)$. Fix $B\in\text{Bor}(X)$.  Note, the induced map on $Y$ is given by the composition
    $$Y\xrightarrow{y\mapsto \mu_y} \text{Prob}_R(X, \text{Bor}(X)\xrightarrow{\text{ev}_B}\mathbb{C}$$
    Since the map $y\mapsto \mu_y$ is continuous with respect to the weak topology and $\text{ev}_B$ is Borel measurable with respect to the weak topology by Lemma \ref{Lemma: Borel evaluation}, then the composition is Borel measurable. Hence the induced map is Borel measurable. Having checked the three conditions, we concluded  $\kappa_\phi$ is a weakly continuous Markov kernel.
    
    Now we prove the uniqueness claim. First, by construction of $\kappa_\phi$, we have for each $y\in Y$ and $f\in C_0(X)$
    $$\phi(f)(y)=\int_{X}f(x)\;d\mu_y(x)=\int_{X}f(x)\;\kappa_\phi(y, dx)$$
    Now suppose $\kappa'$ is a weakly continuous Markov kernel satisfying the condition that for all $y\in Y$ and $f\in C_0(X)$
    $$\phi(f)(y)=\int_{X}f(x)\;\kappa'(y, dx)$$
    By the uniqueness in the Riesz-Markov-Kakutani Representation Theorem, $\kappa'(y, \cdot)=\kappa_\phi(y, \cdot)$ for all $y\in Y$.  Hence $\kappa'=\kappa_\phi$ showing $\kappa_\phi$ is unique.
        
    \end{proof}
    
    \begin{corollary}\label{cor: bijective assignment}\

    \noindent Let $X$ and $Y$ be locally compact Hausdorff topological spaces. The following assignments are well-defined and inverses of each other:
    $$\left\{\begin{matrix} 
      \text{weakly continuous}\\ 
      \text{Markov kernels}\\ 
      \kappa:Y\times \text{Bor}(X)\to [0, 1] 
    \end{matrix}\right\}
    \xrightleftharpoons[\displaystyle\kappa_\phi\mapsfrom \phi]{\quad\displaystyle\kappa\mapsto \phi_\kappa\quad}
    \left\{\begin{matrix} 
      \text{non-degenerate completely}\\ 
      \text{positive maps}\\ 
      \phi:C_0(X)\to C_b(Y) 
    \end{matrix}\right\}$$

    \end{corollary}

    \bibliographystyle{amsalpha} % Or amsalpha
    \bibliography{bibliography}

\end{document}